\documentclass[10pt]{article}
   \usepackage{amsmath, amsfonts, amsthm, amssymb, amscd, enumerate, url, slashed, stmaryrd, upgreek, thmtools, mathdots}
   \usepackage[disable]{todonotes}
   
\begingroup
    \makeatletter
    \@for\theoremstyle:=definition,remark,plain\do{%
        \expandafter\g@addto@macro\csname th@\theoremstyle\endcsname{%
            \addtolength\thm@preskip\parskip
            }%
        }
 \endgroup
 
   \usepackage[titletoc,toc,title]{appendix}

 \usepackage[linktocpage=true,colorlinks,citecolor=blue,linkcolor=blue,urlcolor=cyan,pagebackref]{hyperref}

\usepackage[alphabetic,backrefs]{amsrefs}
            
   \usepackage[parfill]{parskip}
   \usepackage{anysize}
   \marginsize{1in}{1in}{1in}{1in}

\declaretheorem[name=Theorem,numberwithin=section]{thm}
\declaretheorem[name=Proposition,numberlike=thm]{prop}
\declaretheorem[name=Lemma,numberlike=thm]{lemma}
\declaretheorem[name=Corollary,numberlike=thm]{corr}
\declaretheorem[name=Claim,numberlike=thm]{claim}

\declaretheorem[name=Definition,style=definition,qed=$\blacktriangle$,numberlike=thm]{defn}

\declaretheorem[name=Remark,style=definition,qed=$\blacktriangle$,numberlike=thm]{rmk}

\newcounter{commentCounter}
\newcommand{\tm}{\textrm}

\newcommand{\RNum}[1]{\uppercase\expandafter{\romannumeral #1\relax}}

\def\R{\mathbb R}

\def\pt{\partial}
\def\del{\nabla}
\def\G2{\mathrm{G}_2}
\def\g2{\varphi}

\def\cT{\mathcal{T}}

\DeclareMathOperator\Div{div}
\DeclareMathOperator\curl{curl}
\DeclareMathOperator\vol{vol}

\DeclareMathOperator\inj{inj}
\newcommand\tRic{\mathrm{Ric}}

\newcommand{\riem}{\mathrm{Rm}}
\DeclareMathOperator{\id}{id}

\makeatletter
\def\@tocline#1#2#3#4#5#6#7{\relax
  \ifnum #1>\c@tocdepth 
  \else
    \par \addpenalty\@secpenalty\addvspace{#2}%
    \begingroup \hyphenpenalty\@M
    \@ifempty{#4}{%
      \@tempdima\csname r@tocindent\number#1\endcsname\relax
    }{%
      \@tempdima#4\relax
    }%
    \parindent\z@ \leftskip#3\relax \advance\leftskip\@tempdima\relax
    \rightskip\@pnumwidth plus4em \parfillskip-\@pnumwidth
    #5\leavevmode\hskip-\@tempdima
      \ifcase #1
       \or\or \hskip 1.8em \or \hskip 4.7em \else \hskip 3em \fi%
      #6\nobreak\relax
    \hfill\hbox to\@pnumwidth{\@tocpagenum{#7}}\par
    \nobreak
    \endgroup
  \fi}
\makeatother

\newcommand{\hk}{\mathbin{\! \hbox{\vrule height0.3pt width5pt depth 0.2pt \vrule height5pt width0.4pt depth 0.2pt}}}

\numberwithin{equation}{section}

\begin{document}

\title{A gradient flow of isometric $\G2$-structures}

\author{Shubham Dwivedi \\  {\it Department of Pure Mathematics, University of Waterloo} \\ \tt{s2dwived@uwaterloo.ca} \and Panagiotis Gianniotis \\ {\it Department of Mathematics, University of Athens} \\ \tt{pgianniotis@math.uoa.gr} 
 \and Spiro Karigiannis \\ {\it Department of Pure Mathematics, University of Waterloo} \\ \tt{karigiannis@uwaterloo.ca}}

\maketitle

\begin{abstract}
We study a flow of $\G2$-structures that all induce the same Riemannian metric. This isometric flow is the negative gradient flow of an energy functional. We prove Shi-type estimates for the torsion tensor $T$ along the flow. We show that at a finite-time singularity the torsion must blow up, so the flow will exist as long as the torsion remains bounded. We prove a Cheeger--Gromov type compactness theorem for the flow. We describe an Uhlenbeck-type trick which together with a modification of the underlying connection yields a nice reaction-diffusion equation for the torsion along the flow. We define a scale-invariant quantity $\Theta$ for any solution of the flow and prove that it is almost monotonic along the flow. Inspired by the work of Colding--Minicozzi~\cite{colding-minicozzi} on the mean curvature flow, we define an entropy functional and after proving an $\varepsilon$-regularity theorem, we show that low entropy initial data lead to solutions of the flow that exist for all time and converge smoothly to a $\G2$-structure with divergence-free torsion. We also study finite-time singularities and show that at the singular time the flow converges to a smooth $\G2$-structure outside a closed set of finite $5$-dimensional Hausdorff measure. Finally, we prove that if the singularity is of Type-$\tm{\RNum{1}}$ then a sequence of blow-ups of a solution admits a subsequence that converges to a shrinking soliton for the flow. 
\end{abstract}

\tableofcontents

\section{Introduction} \label{introsec}

The existence of torsion-free $\G2$-structures on a manifold is a challenging problem. Geometric flows are a powerful tool to tackle such questions and one hopes that a suitable flow of $\G2$-structures might help in proving the existence of torsion-free $\G2$-structures. There has been a lot of work in this direction. General flows of $\G2$-structures were considered by Karigiannis in~\cite{kar1}. Earlier in~\cite{bryant-remarks}, Bryant introduced the Laplacian flow of closed $\G2$-structures. Several foundational results for the Laplacian flow for closed $\G2$-structures were established in a series of papers~\cite{lotay-wei1, lotay-wei2, lotay-wei3} by Lotay--Wei. The Laplacian flow for co-closed $\G2$-structures was introduced by Karigiannis--McKay--Tsui in~\cite{kar2} and a modified co-flow was studied by Grigorian~\cite{grigorian2}. An approach via gradient flow of energy-type functionals was introduced by Weiss--Witt~\cite{w-w} and Ammann--Weiss--Witt in~\cite{a-w-w}.

In the present paper, we study a different but related problem, in that we use a particular geometric flow to look for a $\G2$-structure which is in some sense optimal. Specifically, we consider a flow $\g2(t)$ of $\G2$-structures on a manifold $M$ that \emph{preserves the Riemannian metric}, which we call the \emph{isometric flow} of $\G2$-structures. This flow is the negative gradient flow of a natural energy functional restricted to the set of $\G2$-structures inducing a fixed metric. The flow seeks a $\G2$-structure amongst those $\G2$-structures inducing the same fixed metric which has minimal $L^2$ norm of torsion.

One possible motivation for studying this isometric flow of $\G2$-structures is that it can be coupled with ``Ricci flow'' of $\G2$-structures, which is a flow of $\G2$-structures that induces \emph{precisely} the Ricci flow on metrics, in contrast to the Laplacian flow which induces Ricci flow plus lower order terms involving the torsion. In effect, one may hope to first flow the $3$-form in a way that improves the metric, and then flow the $3$-form in a way that preserves the metric but still decreases the torsion. More generally, the isometric flow is a particular geometric flow of $\G2$-structures distinct from the Laplacian flow, and both fit into a broader class of geometric flows of $\G2$-structures with good analytic properties. A detailed study of a general class of flows that includes both the Laplacian flow and the isometric flow is currently in preparation by the authors~\cite{dgk-flows}.

We develop a comprehensive foundational theory for the isometric flow. A summary of the main results of the paper is as follows.

In~\textsection\ref{sec:preliminary} we discuss preliminary results on the isometric flow, including the gradient of the energy functional, short-time existence, parabolic rescaling, and solitons.

In~\textsection\ref{sec:estimates} we prove Shi-type estimates for the flow (Theorem~\ref{shiestimatesthm}). We also prove local derivative estimates in Theorem~\ref{localestthm}. Using these we show that the flow~\eqref{divtfloweqn} has a solution as long as the torsion tensor $T$ remains bounded along the flow (Theorem~\ref{ltethm}). We also derive a compactness theorem for solutions along the flow (Theorem~\ref{compactnessthm}).

In~\textsection\ref{rde}, we describe an Uhlenbeck-type trick which together with a modification of the underlying connection yields a nice reaction-diffusion equation for the torsion along the flow (Theorem~\ref{react_diffuse}).

Monotone quantities are a powerful tool in the study of any geometric flow. In ~\textsection\ref{sec:monostuff} we define a quantity $\Theta$ for any solution of the flow. We derive the evolution of $\Theta$ under~\eqref{divtfloweqn} in Lemma~\ref{monotonicity_formula} and prove that it is \emph{almost monotonic} along the flow (Theorem~\ref{almost_mon}). We also prove an $\varepsilon$-regularity result associated to $\Theta$ (Theorem~\ref{eregularity}).

Inspired by work of Colding--Minicozzi in~\cite{colding-minicozzi} and Boling--Kelleher--Streets on the harmonic map heat flow~\cite{boling-kelleher-streets} and work of Kelleher--Streets on the Yang--Mills flow~\cite{kelleher-streets} we define an entropy functional and use it in~Theorem~\ref{thm:lowentropy} to establish that, if we have sufficiently small entropy, then we have long time existence and convergence of the flow to a $\G2$-structure $\g2_{\infty}$ with small divergence-free torsion.

When the entropy is not small the flow may develop singularities in finite time. However, in~\textsection\ref{singular} we prove that we can only have singularities of co-dimension at least $2$. Finally, in~Theorem~\ref{typeI}, we prove that if the singularity is of Type-$\tm{\RNum{1}}$ then a sequence of blow-ups of the flow has a subsequence that converges to a shrinking soliton of the flow.

\textbf{Note.} The almost simultaneous preprint~\cite{grigorian3} by Grigorian has substantial although independent overlap with our results. However, our entropy functional $\lambda$ is different. We also describe an Uhlenbeck-type trick and derive a reaction-diffusion equation for the torsion, and we  obtain results about the structure of singularities for the flow. Moreover, we use a more traditional geometric flows approach, with no use of octonion bundles. The authors believe that both contributions are valuable and complementary. A little bit later, another closely related preprint appeared by Loubeau--S\`a Earp~\cite{earp-loubeau}, in which they consider the more general context of harmonic $G$-structures for a fixed Riemannian metric.

\textbf{Acknowledgements.} All three authors acknowledge the hospitality of the Fields Institute, where a large part of this work was done in 2017 as part of the Major Thematic Program on Geometric Analysis. The second author also acknowledges both the University of Toronto and the University of Waterloo where he spent time as a Fields-Ontario Postdoctoral Fellow during much of this project. Finally, the third author acknowledges funding from NSERC of Canada that helped make this work possible.

\textbf{Notation and Conventions.} We use the symbol $*$ to denote various contractions between tensors whose precise form is not important, and thus we instead use $\star$ for the Hodge star operator. The symbol $C$ is used to denote some positive constant, which may change from line to line in the derivation of an estimate. We very frequently use Young's inequality $ab \leq \tfrac{1}{2\varepsilon} a^2 + \tfrac{\varepsilon}{2} b^2$ for any $\varepsilon, a, b > 0$.

Throughout the paper, we compute in a local orthonormal frame, so all indices are subscripts and any repeated indices are summed over all values from $1$ to $7$. The symbol $\Delta$ always denotes the \emph{analyst's Laplacian} $\Delta = \nabla_k \nabla_k$ which is the negative of the \emph{rough Laplacian} $\nabla^* \nabla$.

Our convention for labelling the Riemann curvature tensor is
\begin{equation*}
R_{ijkm} \frac{\pt}{\pt x^{m}} = (\nabla_{i} \nabla_{j} - \nabla_{j} \nabla_{i})\frac{\pt}{\pt x^{k}}
\end{equation*}
in terms of coordinate vector fields. With this convention, the Ricci tensor is $R_{jk} = R_{ljkl}$, and the Ricci identity is
\begin{equation} \label{ricciidentityeq}
\nabla_{k} \nabla_{i} X_l - \nabla_{i} \nabla_{k} X_l = - R_{kilm} X_m
\end{equation}
Schematically, the Ricci identity implies that
\begin{equation} \label{riccischematic}
\del^m \Delta S - \Delta \del^m S = \sum_{i=0}^{m}\del^{m-i}S*\del^i\riem 
\end{equation}
for any tensor $S$. We also have the Riemannian second Bianchi identity
\begin{equation*}
\nabla_i R_{jkab} + \nabla_j R_{kiab} + \nabla_k R_{ijab} = 0,
\end{equation*}
which when contracted on $i,a$ gives
\begin{equation} \label{riem2Beq}
\nabla_i R_{ibjk} = \nabla_k R_{jb} - \nabla_j R_{kb}.
\end{equation}

\section{Preliminary Results on the Isometric Flow} \label{sec:preliminary}

In this section we discuss several preliminary properties of the isometric flow. This includes a derivation of the fact that it is the negative gradient flow of the energy functional, short-time existence, and parabolic rescaling which we use frequently as a crucial tool. We also discuss solitons for the isometric flow.

\subsection{Review of $\G2$-structures} \label{sec:review}

Let $M^7$ be a smooth manifold. A $\G2$-structure on $M$ is a reduction of the structure group of the frame bundle from $\mathrm{GL}(7, \R)$ to $\G2 \subset \mathrm{SO}(7)$. It is well-known that such a structure exists on $M$ if and only if the manifold is orientable and spinnable, conditions which are respectively equivalent to the vanishing of the first and second Stiefel--Whitney classes. More conveniently from the point of view of differential geometry, a $\G2$-structure on $M$ can also be equivalently defined by a $3$-form $\g2$ on $M$ that satisfies a certain pointwise algebraic non-degeneracy condition. Such a $3$-form nonlinearly (but algebraically, depending only $\g2$ pointwise) induces a Riemannian metric $g_{\g2}$ and an orientation $\vol_{\g2}$ on $M$ and hence a Hodge star operator $\star_{\g2}$. We denote the Hodge dual $4$-form $\star_{\g2} \g2$ by $\psi$. Pointwise we have $| \g2 | = | \psi | = 7$.

There are useful contraction identities involving the $3$-form $\g2$ and the $4$-form $\psi$ of a $\G2$-structure, which we collect here. The proofs of~\eqref{contractphpheq},~\eqref{contractphpseq}, and~\eqref{contractpspseq} below can be found in~\cite{kar1}.

Contractions of $\g2$ with $\g2$:
\begin{equation} \label{contractphpheq}
\begin{aligned}
\g2_{ijk} \g2_{abk} & = g_{ia} g_{jb} - g_{ib} g_{ja} - \psi_{ijab}, \\
\g2_{ijk} \g2_{ajk} & = 6 g_{ia}, \\
\g2_{ijk} \g2_{ijk} & = 42.
\end{aligned}
\end{equation}

Contractions of $\g2$ with $\psi$:
\begin{equation} \label{contractphpseq}
\begin{aligned}
\g2_{ijk} \psi_{abck} & = g_{ia} \g2_{jbc} + g_{ib} \g2_{ajc} + g_{ic} \g2_{abj} - g_{ja} \g2_{ibc} - g_{jb} \g2_{aic} - g_{jc} \g2_{abi}, \\
\g2_{ijk} \psi_{abjk} & = - 4 \g2_{iab}, \\ 
\g2_{ijk} \psi_{aijk} & = 0.
\end{aligned}
\end{equation}

Contractions of $\psi$ with $\psi$:
\begin{equation} \label{contractpspseq}
\begin{aligned}
\psi_{ijkl} \psi_{abkl} & = 4 g_{ia} g_{jb} - 4 g_{ib} g_{ja} - 2 \psi_{ijab}, \\ 
\psi_{ijkl} \psi_{ajkl} & = 24 g_{ia}, \\
\psi_{ijkl} \psi_{ijkl} & = 168.
\end{aligned}
\end{equation}

If $\del$ denotes the Levi-Civita connection of the metric $g$ then $\del \g2$ is interpreted as the torsion of the $\G2$-structure $\g2$. We say that the $\G2$-structure is torsion-free if $\del \g2=0$ and $(M, \g2)$ is then called a $\G2$ manifold. A classical theorem of Fern\`andez--Gray says that $\g2$ is torsion-free if and only if it is closed and co-closed. Manifolds with a torsion-free $\G2$-structures are Ricci-flat and have holonomy contained in the group $\G2$.

In fact the data contained in the torsion $\nabla \g2$ of the $\G2$-structure is equivalent to a $2$-tensor $T$ on $M$ called the \emph{full torsion tensor}. It is defined as the contraction
\begin{equation} \label{fulltorsion}
T_{pq} = \frac{1}{24} \del_p \g2_{ijk} \psi_{qijk}.
\end{equation}
It is more convenient to work with $T$, and henceforth we will simply call $T$ the torsion of the $\G2$-structure. In fact we have
\begin{equation} \label{delpheq}
\begin{aligned}
\del_p \g2_{ijk} & = T_{pm} \psi_{mijk}, \\
\del_p \psi_{ijkl} & = - T_{pi} \g2_{jkl} + T_{pj} \g2_{ikl} - T_{pk} \g2_{ijl} + T_{pl} \g2_{ijk}.
\end{aligned}
\end{equation}
(See~\cite{kar1} for more details.) We write the above expressions in the useful schematic form
\begin{equation} \label{del34eq}
\del \g2 = T*\psi, \qquad \del \psi = T*\g2,
\end{equation}
where recall that $*$ denotes some contraction with the metric.

The torsion $T$ satisfies the ``$\G2$-Bianchi identity'', introduced in~\cite[Theorem 4.2]{kar1}, which is
\begin{equation} \label{G2B}
\del_iT_{jk}-\del_jT_{ik}=T_{ia}T_{jb}\g2_{abk}+\frac 12R_{ijab}\g2_{abk}.
\end{equation}
The important identity~\eqref{G2B} will be used crucially several times in the present paper. 

\subsection{The isometric flow of $\G2$-structures} \label{sec:flow}

In this section we define the isometric flow, and establish that it is a negative gradient flow.

\begin{defn} [Isometric $\G2$-structures]
Two $\G2$-structures $\g2_1$ and $\g2_2$ on $M$ are called \emph{isometric} if they induce the same Riemannian metric, that is if $g_{\g2_1}=g_{\g2_2}$. We will denote the space of $\G2$-structures that are isometric to a given $\G2$-structure $\g2$ by $\llbracket \g2 \rrbracket$.
\end{defn}

\begin{rmk} \label{lin}
The space of \emph{torsion-free} $\G2$-structures that induce the same Riemannian metric was studied by Lin~\cite{lin}. We do not restrict to torsion-free $\G2$-structures in the present paper.
\end{rmk}

Fix an initial $\G2$-structure $\g2_0$ on $M$.

\begin{defn} \label{Edefn}
Define the \emph{energy functional} $E$ on the set $\llbracket \g2_0 \rrbracket$ by
\begin{equation} \label{energy}
E(\g2)=\frac{1}{2} \int_M |T_{\g2}|^2 \vol_{\g2}
\end{equation}
where $T_{\g2}$ is the torsion of $\g2$.
\end{defn} 

Note that $E$ is the same functional considered in~\cite{w-w}, but here we only allow $\g2$ to vary in the class $\llbracket \g2_0 \rrbracket$ of isometric $\G2$-structures, whereas in~\cite{w-w} the functional was considered on the space of \emph{all} $\G2$-structures. 

The functional $E$ in~\eqref{energy} was considered by Grigorian in~\cite{grigorian1} in the context of ``octonionic bundles" over $M$ where he showed that the critical points of the functional are precisely the $\G2$-structures with divergence-free torsion, that is, $\Div T=0$. Note that the underlying metric here is the same for all $\G2$-structures in $\llbracket \g2_0 \rrbracket$, so the divergence is unambiguously defined. A very natural question arises: given any initial $\G2$-structure $\g2_0$ on $M$ what is the `best' $\G2$-structure in the class $\llbracket \g2_0 \rrbracket$. An obvious way to study this question is to consider the negative gradient flow of the functional~\eqref{energy}. (In fact it is more convenient to take the negative gradient flow of $4E$. See Proposition~\ref{gradient}.)
 
Before we can describe this flow, we need to introduce some notation. Let $h$ be a symmetric $2$-tensor on $M$. We define a $3$-form $h \diamond \g2$ on $M$ by the formula
\begin{equation} \label{diamondactioneq}
(h \diamond \g2)_{ijk} = h_{ip} \g2_{pjk} + h_{jp} \g2_{ipk} + h_{kp} \g2_{ijp}. 
\end{equation}
Note from~\eqref{diamondactioneq} that if $h = g$ is the metric, we get
\begin{equation} \label{gdiamondeq}
g \diamond \g2 = 3 \g2.
\end{equation}
Using this notation, the most general flow of $\G2$-structures~\cite{kar1} is given by
\begin{equation} \label{genflow}
\frac{\pt \g2}{\pt t}= h\diamond \g2 + X\hk \psi 
\end{equation}
where $h$ is a time-dependent symmetric $2$-tensor and $X$ is a time-dependent vector field. In this case the flow of the metric $g$ is given by
\begin{equation} \label{metricflow}
\frac{\pt g}{\pt t}=2h.
\end{equation}

To begin we consider the first variation of the torsion $T$ with respect to variations of the $\G2$-structure that preserve the metric.

\begin{lemma} \label{first-variation}
Let $(\g2_t)_{t\in(-\delta,\delta)}$ be a smooth family of $\G2$-structures in the class $\llbracket \g2 \rrbracket$ with $\g2_0= \g2$. By equations~\eqref{genflow} and~\eqref{metricflow}, we can write $\left. \frac{\pt}{\pt t} \right|_{t=0} \g2_t = X \hk \psi$ for some vector field $X$. Let $T_t$ be the torsion of $\g2_t$. Then we have
\begin{equation} \label{first-variation-eq}
\left. \frac{\pt}{\pt t} \right|_{t=0} (T_t)_{ij} = \nabla_i X_j + X_l T_{im} \g2_{lmj}.
\end{equation}
\end{lemma}
\begin{proof}
Since $g_t = g$ for all $t \in (-\delta, \delta)$, the covariant derivative $\nabla$ is independent of $T$. Since $\left. \frac{\pt}{\pt t} \right|_{t=0} \g2_t = X \hk \psi$, by~\cite[Theorem 3.5]{kar1} we have $\left. \frac{\pt}{\pt t} \right|_{t=0} \psi_t = - X \wedge \g2$. That is, we have
\begin{align*}
\left. \frac{\pt}{\pt t} \right|_{t=0} (\g2_t)_{ijk} & = X_p \psi_{pijk}, \\
\left. \frac{\pt}{\pt t} \right|_{t=0} (\psi_t)_{ijkl} & = - X_i \g2_{jkl} + X_j \g2_{ikl} - X_k \g2_{ijl} + X_l \g2_{ijk}.
\end{align*}
From these observations and equation~\eqref{fulltorsion}, we compute
\begin{align*}
24 \left. \frac{\pt}{\pt t} \right|_{t=0} (T_t)_{ij} & = \left. \frac{\pt}{\pt t} \right|_{t=0} ( \nabla_i (\g2_t)_{abc} (\psi_t)_{jabc} ) \\
& = \nabla_i \big( \left. \frac{\pt}{\pt t} \right|_{t=0} (\g2_t)_{abc} \big) \psi_{jabc} + \nabla_i \g2_{abc} \big( \left. \frac{\pt}{\pt t} \right|_{t=0} (\psi_t)_{jabc} \big) \\
& = \nabla_i ( X_p \psi_{pabc} ) \psi_{jabc} + \nabla_i \g2_{abc} ( - X_j \g2_{abc} + X_a \g2_{jbc} - X_b \g2_{jac} + X_c \g2_{jab} ).
\end{align*}
Using~\eqref{delpheq} and the contraction identities~\eqref{contractphpseq} and~\eqref{contractpspseq}, the above becomes
\begin{align*}
24 \left. \frac{\pt}{\pt t} \right|_{t=0} (T_t)_{ij} & = \nabla_i X_p \psi_{pabc} \psi_{jabc} + X_p \nabla_i \psi_{pabc} \psi_{jabc}  \\
& \qquad + T_{ip} \psi_{pabc} (- X_j \g2_{abc} + X_a \g2_{jbc} - X_b \g2_{jac} + X_c \g2_{jab}) \\
& = 24 \nabla_i X_p g_{pj} + X_p (- T_{ip} \g2_{abc} + T_{ia} \g2_{pbc} - T_{ib} \g2_{pac} + T_{ic} \g2_{pab} ) \psi_{jabc} \\
& \qquad - 0 + 3 T_{ip} X_a \g2_{jbc} \psi_{pabc} \\
& = 24 \nabla_i X_j - 0 + 3 T_{ia} X_p \g2_{pbc} \psi_{jabc} + 3 T_{ip} X_a (-4 \g2_{jpa}) \\
& = 24 \nabla_i X_j + 3 T_{ia} X_p (-4 \g2_{pja}) -12 X_a T_{ip} \g2_{jpa} \\
& = 24 \nabla_i X_j + 24 X_a T_{ip} \g2_{apj},
\end{align*}
which is precisely~\eqref{first-variation-eq}.
\end{proof}

Now let $E$ be the energy functional from Definition~\ref{Edefn}, restricted to the set $\llbracket \g2 \rrbracket$ of $\G2$-structures inducing the same metric as $\g2$.

\begin{prop} \label{gradient}
The \emph{gradient} of $4 E : \llbracket \g2 \rrbracket \to \R$ at the point $\g2$ is $- \Div T \hk \psi$, where $T$ is the torsion of $\g2$ and $\psi = \star \g2$. That is, if $(\g2_t)_{t\in(-\delta,\delta)}$ is a smooth family in the class $\llbracket \g2 \rrbracket$ with $\g2_0= \g2$ and $\left. \frac{d}{dt} \right|_{t=0} \g2_t = \eta$, then
\begin{equation*}
\left. \frac{d}{dt} \right|_{t=0} 4 E(\g2_t) = - \int_M \langle \Div T \hk \psi, \eta \rangle \vol_g.
\end{equation*}
\end{prop}
\begin{proof}
Using Lemma~\ref{first-variation} compute
\begin{align*}
\left. \frac{d}{dt} \right|_{t=0} E (\g2_t) & = \left. \frac{d}{dt} \right|_{t=0} \frac{1}{2} \int_ M (T_t)_{ij} (T_t)_{ij} \vol_g \\
& = \int_M T_{ij} ( \nabla_i X_j + X_l T_{im} \g2_{lmj} )\vol_g.
\end{align*}
The second term vanishes because $T_{ij} T_{im}$ is symmetric in $j,m$ and $\g2_{lmj}$ is skew in $j,m$. We integrate by parts on the first term to obtain
\begin{equation*}
\left. \frac{d}{dt} \right|_{t=0} E (\g2_t) = - \int_M X_j \nabla_i T_{ij} \vol_g = - \int_M \langle X, \Div T \rangle \vol_g.
\end{equation*}
Equation~\eqref{contractpspseq} implies that $\langle X \hk \psi, Y \hk \psi \rangle = \tfrac{1}{6} X_p \psi_{pabc} Y_q \psi_{qabc} = 4 X_p Y_p = 4 \langle X, Y \rangle$, so the above equation becomes
\begin{equation} \label{gradienttemp}
\left. \frac{d}{dt} \right|_{t=0} 4 E (\g2_t) = - 4 \int_M X_j \nabla_i T_{ij} \vol_g = - \int_M \langle X \hk \psi, \Div T \hk \psi \rangle \vol_g.
\end{equation}
The space of $3$-forms decomposes into the pointwise orthogonal splitting
\begin{equation*}
\Omega^3 = \Omega^3_1 \oplus \Omega^3_7 \oplus \Omega^3_{27},
\end{equation*}
where $\Omega^3_7 = \{ Y \hk \psi : Y \in \Gamma(TM) \}$. Using this observation, the result follows immediately from~\eqref{gradienttemp}.
\end{proof}

We can now define the isometric flow.
\begin{defn}[The isometric flow]
Let $(M^7, \g2_0)$ be a compact manifold with a $\G2$-structure. Consider the negative gradient flow of the functional $4 E$ restricted to the class $\llbracket \g2 \rrbracket$. By Proposition~\ref{gradient}, this evolution of $\g2$ is given by
\begin{align} \label{divtfloweqn} 
 \begin{cases} 
      & \frac{\pt \g2}{\pt t} = \Div T \hk \psi, \\
      & \g2(0) =\g2_0.
   \end{cases}
\end{align}
We call~\eqref{divtfloweqn} the \emph{isometric flow} of $\G2$-structures. Note from~\eqref{genflow} that $h \equiv 0$ for the isometric flow and hence~\eqref{divtfloweqn} is indeed a flow of isometric $\G2$-structures.
\end{defn}

\subsection{Short time existence} \label{sec:short-time}

The isometric flow~\eqref{divtfloweqn} has short time existence and uniqueness, because it is equivalent to a strictly parabolic flow. This was first proved by Bagaglini in~\cite{bagaglini} using spinorial methods. A proof is also given in Grigorian~\cite[Section 5]{grigorian3} using octonion algebra. In this section we explain how to derive the equivalent strictly parabolic flow, avoiding the use of spinors or octonions. The full details are quite laborious and unenlightening. We need to make extensive use of the various contraction identities in~\eqref{contractphpheq} and~\eqref{contractphpseq}. We present just enough details so that the interested reader can fill in the gaps on their own.

\emph{Note:} In this section only, for brevity, we use $\dot{A}$ to denote the time derivative of $A$.

The starting point is the following result of Bryant.
\begin{prop}[{\cite[Equation (3.6)]{bryant-remarks}}]
Let $(M, \g2)$ be a manifold with $\G2$-structure such that $\g2$ induces the Riemannian metric $g$. Then all the other $\G2$-structures on $M$ inducing the same metric $g$ can be parametrized by a pair $(f, X)$ where $f$ is a function and $X$ is a vector field satisfying $f^2 + |X|^2 = 1$. The explicit formula for the $\G2$-structure $\g2_{(f,X)}$ corresponding to the pair $(f,X)$ is
\begin{equation} \label{bryant}
\g2_{(f,X)} = (f^2 - |X|^2) \g2 - 2f X \hk \psi + 2 X \wedge (X \hk \g2),
\end{equation}
where $\psi = \star_g \g2$ and the norm of $X$ is taken with respect to $g$. Note that the pair $(-f, -X)$ induces the same $\G2$-structure as $(f,X)$ so in fact the $\G2$-structures on $M$ inducing the metric $g$ correspond to sections of an $\mathbb R \mathbb P^7$-bundle over $M$.
\end{prop}

Fix a pair $(f,X)$ with $f^2 + |X|^2 = 1$ and write $\widetilde{\g2}$ for $\g2_{f,X}$. In terms of a local orthonormal frame, equation~\eqref{bryant} is
\begin{equation} \label{eq:bryant3-coordinates}
\begin{aligned}
\widetilde{\g2}_{ijk} & = (1 - 2|X|^2 )\g2_{ijk} - 2 f X_m \psi_{mijk} \\
& \qquad + 2 X_i X_m \g2_{mjk} + 2 X_j X_m \g2_{imk} + 2 X_k X_m \g2_{ijm}.
\end{aligned}
\end{equation}
Since $\widetilde{\g2}$ induces the same metric $g$ as $\g2$, they have the same Hodge star operator $\star$, so we have $\psi_{(f,X)} = \star \g2_{(f,X)}$. Using equation~\eqref{bryant} and the identity $\star(X \wedge \alpha) = (-1)^k X \hk \star \alpha$ for $\alpha$ a $k$-form, we obtain
\begin{align*}
\psi_{(f,X)} & = (1 - 2|X|^2) \star \g2 - 2 f \star(X \hk \psi) + 2 \star(X \wedge (X \hk \g2)) \\
& = (1 - 2|X|^2) \psi + 2 f X \wedge \g2 + 2 X \hk \star (X \hk \g2) \\
& = (1 - 2|X|^2) \psi + 2 f X \wedge \g2 + 2 X \hk (X \wedge \psi).
\end{align*}
Using the fact that $\hk$ is a derivation, this becomes
\begin{align*}
\psi_{(f,X)} & = (1 - 2|X|^2) \psi + 2 f X \wedge \g2 + 2|X|^2 \psi - 2 X \wedge (X \hk \psi) \\
& = \psi + 2 fX \wedge \g2 - 2 X \wedge (X \hk \psi).
\end{align*}
In a local frame this is
\begin{equation} \label{eq:bryant4-coordinates}
\begin{aligned}
\widetilde{\psi}_{qjkl} & = \psi_{qjkl} + 2 f (X_q \g2_{jkl} - X_j \g2_{qkl} + X_k \g2_{qjl} - X_l \g2_{qjk} ) \\
& \qquad - 2 (X_q X_m \psi_{mjkl} + X_j X_m \psi_{qmkl} + X_k X_m \psi_{qjml} + X_l X_m \psi_{qjkm}).
\end{aligned}
\end{equation}
Note that all the contractions above are taken with respect to the fixed metric $g$ that is induced by both $\g2$ and $\widetilde{\g2}$.

Now suppose that $\g2_t$ is evolving by the isometric flow~\eqref{divtfloweqn}. Since the metric is constant, this time-dependent $\G2$-structure will correspond by~\eqref{bryant} to a time-dependent pair $(f, X)$. We write $\widetilde{\g2}$ for $\g2_t$, with torsion $\widetilde{T} = T_t$. The initial condition $\g2_0 = \g2$ corresponds to initial conditions $f_0 = 1$ and $X_0 = 0$.
\begin{prop} \label{prop:fXflow}
Under the isometric flow, the pair $(f,X)$ evolves by
\begin{equation} \label{eq:fXflow1}
\begin{aligned}
\dot f & = \frac{1}{2} \langle X, \Div \widetilde{T} \rangle, \\
\dot X & = - \frac{1}{2} f \Div \widetilde{T} + \frac{1}{2} (\Div \widetilde{T}) \times X,
\end{aligned}
\end{equation}
where $\times$ is the cross product with respect to the initial $\G2$-structure $\g2$, given by $(Y \times X)_k = Y_a X_b \g2_{abk}$, and $\langle \cdot, \cdot \rangle$ is the inner product given by the metric $g$.
\end{prop}
\begin{proof}
Let $\gamma = \dot{\g2}_t$. Since $\g2$ and $\psi$ in equation~\eqref{eq:bryant3-coordinates} are constant in time, differentiating with respect to $t$ we get
\begin{equation*}
\begin{aligned}
\gamma_{ajk} & = - 4 \langle X, \dot X \rangle \g2_{ajk} - 2 \dot f X_m \psi_{majk} - 2 f \dot X_m \psi_{majk} \\
& \qquad + 2 \dot X_a X_m \g2_{mjk} + 2 \dot X_j X_m \g2_{amk} + 2 \dot X_k X_m \g2_{ajm} \\
& \qquad + 2 X_a \dot X_m \g2_{mjk} + 2 X_j \dot X_m \g2_{amk} + 2 X_k \dot X_m \g2_{ajm}.
\end{aligned}
\end{equation*}
Let $\sigma = \Div \widetilde{T} \hk \widetilde{\psi}$. Using~\eqref{eq:bryant4-coordinates} we have
\begin{equation*}
\begin{aligned}
\sigma_{ajk} & = (\Div \widetilde{T})_m \widetilde{\psi}_{majk} \\
& = (\Div \widetilde{T})_m \psi_{majk} + 2 f (\Div \widetilde{T})_m (X_m \g2_{ajk} - X_a \g2_{mjk} + X_j \g2_{mak} - X_k \g2_{maj} ) \\
& \qquad - 2 (\Div \widetilde{T})_m (X_m X_p \psi_{pajk} + X_a X_p \psi_{mpjk} + X_j X_p \psi_{mapk} + X_k X_p \psi_{majp}).
\end{aligned}
\end{equation*}
Under the flow we have $\gamma = \dot{\g2}_t = \Div \widetilde{T} \hk \widetilde{\psi} = \sigma$, so we must have $\gamma_{ajk} = \sigma_{ajk}$. Contracting both sides of this equation with $\g2_{ijk}$ gives an equivalent equation, as the map $\alpha_{ajk} \mapsto \g2_{ijk} \alpha_{ajk}$ is a linear isomorphism from $\Lambda^3 = \Lambda^3_1 \oplus \Lambda^3_7 \oplus \Lambda^3_{27}$ onto $\mathrm{Sym}^2 \oplus \Lambda^2_7$, the space of $2$-tensors with no $\Lambda^2_{14}$ component. (See~\cite{kar1} for details.) Now using the contraction identities~\eqref{contractphpheq} and~\eqref{contractphpseq}, one can compute that
\begin{equation} \label{fXgammaeq}
\g2_{ijk} \gamma_{ajk} = -16 \langle X, \dot X \rangle g_{ia} + 8 (X_i \dot X_a + X_a \dot X_i) - 8(\dot f X_p + f \dot X_p) \g2_{pia},
\end{equation}
and similarly that
\begin{equation} \label{fXsigmaeq}
\begin{aligned}
\g2_{ijk} \sigma_{ajk} & = 4 (\Div \widetilde{T})_p \g2_{pia} + 8 f \langle X, \Div \widetilde{T} \rangle g_{ia} - 8 f (\Div \widetilde{T})_i X_a + 4 f (\Div \widetilde{T})_p X_q \psi_{pqia} \\
& \qquad 4 (\Div \widetilde{T} \times X)_i X_a + 4 (\Div \widetilde{T} \times X)_a X_i - 4 \langle X, \Div \widetilde{T} \rangle X_p \g2_{pia} - 4 |X|^2 (\Div \widetilde{T})_p \g2_{pia}.
\end{aligned}
\end{equation}
Thus from $\gamma = \sigma$, the right hand sides of equations~\eqref{fXgammaeq} and~\eqref{fXsigmaeq} must be equal. If we take the trace of both sides, we find that
\begin{equation} \label{fXtraceeq}
\langle X, \dot X \rangle = - \frac{f}{2} \langle X, \Div \widetilde{T} \rangle.
\end{equation}
On the other hand, if we contract both sides with $\g2_{iak}$, we find that
\begin{equation} \label{fXphieq}
\dot f X_k + f \dot X_k = -\frac{f^2}{2} (\Div \widetilde{T})_k + \frac{f}{2} ( (\Div \widetilde{T}) \times X )_k+ \frac{1}{2} \langle X, \Div \widetilde{T} \rangle X_k.
\end{equation}
Multiplying~\eqref{fXphieq} with $X_k$ and summing over $k$, we get
\begin{equation*}
\dot f |X|^2 + f \langle X, \dot X \rangle = - \frac{f^2}{2} \langle \Div \widetilde{T}, X \rangle + 0 + \frac{1}{2} \langle X, \Div \widetilde{T} \rangle |X|^2.
\end{equation*}
Substituting~\eqref{fXtraceeq} into the above, we obtain the first equation in~\eqref{eq:fXflow1}. Then substituting that back into~\eqref{fXphieq} gives the second equation in~\eqref{eq:fXflow1}. Thus the two equations in~\eqref{eq:fXflow1} are necessary consequences of $\gamma = \sigma$. However, substituting both equations in~\eqref{eq:fXflow1} back into~\eqref{fXgammaeq} and~\eqref{fXsigmaeq} shows that these are in fact sufficient to ensure $\gamma = \sigma$. Thus the proof is complete.
\end{proof}

In fact, from $f^2 = 1 - |X|^2$, it is easy to check that the first equation in~\eqref{eq:fXflow1} is a consequence of the second equation in~\eqref{eq:fXflow1}. Thus the isometric flow~\eqref{divtfloweqn} is completely determined by the single equation $\dot X = - \tfrac{1}{2} f \Div \widetilde{T} + \tfrac{1}{2} (\Div \widetilde{T}) \times X$. In order to establish that this equation is strictly parabolic, we need to express the torsion $T_t = \widetilde{T}$ and its divergence in terms of $(f,X)$.
\begin{lemma} \label{torsionfX}
The torsion $\widetilde{T}$ of $\widetilde{\g2} = \g2_{(f,X)}$ is
\begin{equation} \label{eq:torsionfX}
\begin{aligned}
\widetilde{T}_{pq} & = (1 - 2|X|^2)T_{pq} + 2 T_{pm} X_m X_q + 2 f T_{pm} X_l \g2_{mlq} \\
& \qquad - 2 \del_p X_m X_l \g2_{mlq} + 2 \del_p f X_q - 2 f \del_p X_q.
\end{aligned}
\end{equation}
\end{lemma}
\begin{proof}
Taking $\del_p$ of~\eqref{eq:bryant3-coordinates} gives
\begin{align*}
\del_p \widetilde{\g2}_{ijk} & = -4 \del_p X_m X_m \g2_{ijk} + (1-2|X|^2) \del_p \g2_{ijk} \\
& \qquad - 2 \del_p f X_m \psi_{mijk} - 2f \del_p X_m \psi_{mijk} - 2 f X_m \del_p \psi_{mijk} \\
& \qquad + 2 \del_p X_i X_m \g2_{mjk} + 2 \del_p X_j X_m \g2_{imk} + 2 \del_p X_k X_m \g2_{ijm} \\ & \qquad + 2 X_i \del_p X_m \g2_{mjk} + 2 X_j \del_p X_m \g2_{imk} + 2 X_k \del_p X_m \g2_{ijm} \\
& \qquad + 2 X_i X_m \del_p \g2_{mjk} + 2 X_j X_m \del_p \g2_{imk} + 2 X_k X_m \del_p \g2_{ijm}.
\end{align*}
We now substitute the expressions for $\nabla \g2$ and $\nabla \psi$ from~\eqref{delpheq} into the above expression, and use~\eqref{fulltorsion} to write
\begin{equation*}
24 \widetilde{T}_{pq} = \del_p \widetilde{\g2}_{ijk} \widetilde{\psi}_{qijk}.
\end{equation*}
After an extremely lengthy computation using the various identities in~\eqref{contractphpheq} and~\eqref{contractphpseq}, one indeed obtains the result~\eqref{eq:torsionfX}. We omit the details.
\end{proof}

\begin{corr} \label{cor:divTfX}
The divergence $\Div \widetilde{T}_q = \del_p \widetilde{T}_{pq}$ of the torsion $\widetilde{T}$ of $\widetilde{\g2} = \g2_{(f,X)}$ is
\begin{equation} \label{eq:divTfX}
\begin{aligned}
\Div \widetilde{T}_{q} & = (1-2|X|^2) (\Div T)_q - 4 X_m \del_p X_m T_{pq} + 2 (\Div T)_m X_m X_q + 2 T_{pm} \del_p X_m X_q \\
& \qquad + 2 T_{pm} X_m \del_p X_q + 2 \del_p f T_{pl} X_m \g2_{lmq} + 2 f (\Div T)_l X_m \g2_{lmq} + 2 f T_{pl} \del_p X_m \g2_{lmq} \\
& \qquad - 2 \del_p \del_p X_l X_m \g2_{lmq} - 2 \del_p X_l X_m T_{ps} \psi_{slmq} + 2 \del_p \del_p f X_q - 2 f \del_p \del_p X_q.
\end{aligned}
\end{equation}
\end{corr}
\begin{proof}
This again follows by applying $\del_p$ to equation~\eqref{eq:torsionfX} and using the various identities in~\eqref{contractphpheq} and~\eqref{contractphpseq}. We omit the details.
\end{proof}

We can now apply the above result as follows.
\begin{prop} \label{prop:fXflow2}
Under the isometric flow, the vector field $X$ evolves by
\begin{equation} \label{eq:fXflow2}
\begin{aligned}
\dot X_q & = \Delta X_k + f X_m \del_p X_m T_{pq} - f T_{pm} \del_p X_m X_q - T_{pl} \del_p X_m \g2_{lmq} \\
& \qquad \qquad + |\del f|^2 X_q + |\del X|^2 X_q - |X|^2 \del_p f T_{pq} + T_{pl} \del_p f X_l X_q \\
& \qquad \qquad + T_{ps} \del_p X_l X_a X_q \g2_{sla} - \frac{f}{2} (\Div T)_q + \frac{1}{2} (X \times (\Div T))_q.
\end{aligned}
\end{equation}
\end{prop}
\begin{proof}
Once again this follows from equations~\eqref{eq:fXflow1} and~\eqref{eq:divTfX} after a lengthy calculation, using also the relation $f^2 + |X|^2 = 1$.
\end{proof}

Equation~\eqref{eq:fXflow2} is just a heat equation for the vector field $X$ with lower order terms, and is thus strictly parabolic. Using classical parabolic theory, we have therefore established the following result.

\begin{thm} \label{stethm}
Let $(M^7, \g2_0)$ be a compact manifold with $\G2$-structure. Then the flow~\eqref{divtfloweqn} has a unique solution for a short time $t\in [0, \varepsilon)$.
\end{thm}

\subsection{Parabolic rescaling} \label{sec:rescaling}

As is usual for geometric evolution equations, the natural `parabolic rescaling' of the problem involves scaling the $t$ by $c^2 t$ when we scale the space variables by $c$. In this section we make this precise, as we will crucially use this property frequently in the rest of the paper.

\begin{lemma} \label{lem:rescaling}
Let $c > 0$ be a constant. If $\g2(t)$ is a solution of the isometric flow~\eqref{divtfloweqn} with $\g2(0) = \g2$, then $\widetilde{\g2} (\widetilde t) = c^3 \g2(c^2 t)$ is a solution of~\eqref{divtfloweqn} with $\widetilde{\g2} (0) = c^3 \g2$.
\end{lemma}
\begin{proof}
Define a new $\G2$-structure $\widetilde{\g2} = c^3 \g2$. Then it follows~\cite[Theorem 2.23]{kar1} that $\widetilde g = c^2 g$ and $\widetilde \psi = c^4 \psi$. Hence from~\eqref{fulltorsion} we have $\widetilde T = c T$. (Recall that we are suppressing the writing of the $g^{-1}$ terms because we are using an orthonormal frame.) Therefore as a $1$-form, $\Div_{\widetilde g} \widetilde T = c^{-1} \Div_g T$, and so converting to vector fields using the metric, we have $(\Div_{\widetilde g} \widetilde T) \hk \widetilde \psi = c^{-1} c^{-2} c^4 (\Div_g T) \hk \psi = c (\Div_g T) \hk \psi$. But then it is clear from~\eqref{divtfloweqn} that with $\widetilde t = c^2 t$, we obtain the desired conclusion.
\end{proof}

We note here for later use that if $\widetilde{\g2} = c^3 \g2$, then we also have
\begin{equation} \label{rescaling}
| \widetilde{\del}^j \widetilde\riem |_{\widetilde g} = c^{-(2+j)} | \del^j \riem |_{g}, \qquad | \widetilde{\del}^j \widetilde T |_{\widetilde g} = c^{-(1 + j)} | \nabla^j T |_g.
\end{equation}

\subsection{Solitons for the isometric flow} \label{solitons}

In this section we study the relation between self-similar solutions and solitons for the isometric flow.

Let $\mathcal L_Y$ denote the Lie derivative with respect to $Y$. Consider the identity
\begin{equation*}
(\mathcal L_Y \g2)_{ijk} = (\nabla_Y \g2)_{ijk} + \nabla_i Y_p \g2_{pjk} + \nabla_j Y_p \g2_{ipk} + \nabla_k Y_p \g2_{ijp}.
\end{equation*}
Using equations~\eqref{delpheq} and~\eqref{diamondactioneq} we can rewrite the above as
\begin{equation*}
(\mathcal L_Y \g2)_{ijk} = Y_l T_{lp} \psi_{pijk} + \big( (\nabla Y) \diamond \g2 \big)_{ijk}.
\end{equation*}
The second term above can be written as $h \diamond \g2 + Z \hk \psi$ where $h_{ij} = \tfrac{1}{2} (\nabla_i Y_j + \nabla_j Y_i) = \tfrac{1}{2} (\mathcal L_Y g)_{ij}$ and $Z$ is a vector field on $M$ such that $Z_p \psi_{pijk}$ is the $\Omega^3_7$ component of $(\nabla Y) \diamond \g2$. Because $\Omega^3_1 \oplus \Omega^3_{27}$ is the kernel of $\gamma \mapsto \gamma_{ijk} \psi_{mijk}$, from the contraction identities~\eqref{contractpspseq} and~\eqref{contractphpseq} we deduce that
\begin{align*}
24 Z_m & = Z_l \psi_{lijk} \psi_{mijk} = ( \nabla_i Y_p \g2_{pjk} + \nabla_j Y_p \g2_{ipk} + \nabla_k Y_p \g2_{ijp} ) \psi_{mijk} \\
& = 3 \nabla_i Y_p \g2_{pjk} \psi_{mijk} = - 12 \nabla_i Y_p \g2_{pmi}.
\end{align*}
Thus we have $Z_m = -\tfrac{1}{2} \nabla_i Y_j \g2_{ijm} = -\tfrac{1}{2} (\curl Y)_m$. (See~\cite{kar-notes} for more about the curl operator.)

Combining these observations we can write
\begin{equation} \label{Liederivative}
(\mathcal L_Y \g2)_{ijk} = (Y \hk T)_p \psi_{pijk} - \tfrac{1}{2} (\curl Y)_p \psi_{pijk} + \tfrac{1}{2} (\mathcal L_Y g) \diamond \g2.
\end{equation}

\begin{defn} Let $(\varphi(t))_{t\in(\alpha,\beta)}$ be a solution of the isometric flow~\eqref{divtfloweqn} where $0 \in (\alpha, \beta)$. We say that it is a \emph{self-similar solution} if there exist a function $a(t)$ with $a(0)=1$, a $\G2$-structure $\varphi_0$, and a family of diffeomorphisms $f_t :M\rightarrow M$ with $f_0 = \id_M$ such that
\begin{equation*}
\varphi(t)= (a(t))^3 f_t^* \varphi_0
\end{equation*}
for all $t\in (\alpha,\beta)$. Since $\varphi(t)$ is a solution to the isometric flow, we have
\begin{equation*}
g(t):=g_{\varphi(t)} = g_{\varphi(0)} = f_0^* g_{\varphi_0} = g(0).
\end{equation*}
\end{defn}

\begin{lemma}\label{lemma:selfsimilar}
Given a self-similar solution $(\varphi(t))_{t\in (\alpha,\beta)}$ of the isometric flow, there is a family $X(t)$ of vector fields such that
\begin{equation*}
\Div T_{\varphi(t)}= -\frac{1}{2} \curl_{\varphi(t)} (X(t)) + X(t) \hk T_{\varphi(t)}.
\end{equation*}
In particular, there is a vector field $X_0$ such that $\varphi_0$ satisfies
\begin{equation*}
\Div T_{\varphi_0} = -\frac{1}{2} \curl_{\varphi_0} (X_0) + X_0 \hk T_{\varphi_0}.
\end{equation*}
\end{lemma}
\begin{proof}
Set $\varphi_0=\varphi(0)$ and $g_0=g_{\varphi_0}$, and let $W(t)$ be the infinitesimal generator of $f_t$. That is,
\begin{equation*}
\frac{\pt}{\pt t} f_t = W(t)\circ f_t.
\end{equation*}
With $X(t)= (f_t^{-1})_* W(t)$ we compute
\begin{align} \nonumber
\frac{\pt}{\pt t} \varphi(t) &= 3 a'(t) (a(t))^2 f_t^* \varphi_0 + (a(t))^3 f_t^* (\mathcal L_{W(t)} \varphi_0) \\ \nonumber 
&=3 a'(t) (a(t))^2 f_t^* \varphi_0  + (a(t))^3 \mathcal L_{(f_t^{-1})_* W(t)} f_t^* \varphi_0 \\  \label{ddt_fi}
&=3a'(t) (a(t))^{-1} \varphi(t) + \mathcal L_{X(t)} \varphi(t).
\end{align}
From~\eqref{Liederivative} we also have
\begin{align}
\begin{split} \label{lie}
\mathcal L_{X(t)} \varphi(t) =\frac{1}{2}\mathcal L_{X(t)} g(t) \diamond \varphi(t) +\left(-\frac{1}{2}\curl_{\varphi(t)} X(t) + X(t) \hk T\right) \hk \psi(t).
\end{split}
\end{align}

On the other hand, since $g(t)=g_{\varphi(t)} = (a(t))^2 f_t^* g_0$ we find that
\begin{align} \nonumber
0=\frac{\pt}{\pt t} g(t) &= 2 a'(t) a(t) f_t^* g_0 + (a(t))^2 f_t^*( \mathcal L_{W(t)} g_0) \\ \nonumber
&= 2 a'(t) a(t) f_t^* g_0 + (a(t))^2 \mathcal L_{(f_t^{-1})^* W(t)} g_0 \\ \label{metric}
&= 2 a'(t) (a(t))^{-1} g(t) + \mathcal L_{X(t)} g(t).
\end{align}
Hence, combining~\eqref{lie} and~\eqref{metric}, and using also~\eqref{gdiamondeq}, the expression~\eqref{ddt_fi} becomes 
\begin{align*}
\frac{\pt}{\pt t}\varphi(t) & = \Div T_{\varphi(t)} \hk \psi(t) \\
&=3 a'(t) (a(t))^{-1} \varphi(t) +\frac{1}{2}\mathcal L_{X(t)} g(t) \diamond \varphi(t) +\left(-\frac{1}{2}\curl_{\varphi(t)} X(t) + X(t) \hk T\right) \hk \psi(t) \\
& = 3 a'(t) (a(t))^{-1} \varphi(t) - a'(t) (a(t))^{-1} g(t) \diamond \varphi(t) + \Big( -\frac{1}{2}\curl_{\varphi(t)} X(t) + X(t) \hk T \Big) \hk \psi(t) \\
& = \Big( -\frac{1}{2}\curl_{\varphi(t)} X(t) + X(t) \hk T \Big) \hk \psi(t)
\end{align*}
as claimed.
\end{proof}

\begin{defn} \label{defn:isometric-soliton}
An \emph{isometric soliton} on $(M, g_0)$ is defined to be a triple $(\varphi_0, X_0, c)$ where $\varphi_0$ is a $\G2$-structure on $M$ inducing the Riemannian metric $g_0$, and $X_0$ is a vector field satisfying
\begin{equation*}
\mathcal L_{X_0} g_0= c g_0
\end{equation*}
for some constant $c \in \R$ and 
\begin{equation*}
\Div T_{\varphi_0}= -\frac{1}{2} \curl_{\varphi_0} X_0 + X_0 \hk T_{\varphi_0}.
\end{equation*}
Moreover, it is called shrinking, steady, or expanding, depending on whether $c$ is positive, zero, or negative, respectively.
\end{defn}

We now relate isometric solitons to self-similar solutions of the isometric flow.
\begin{lemma} \label{lemma:solitons}
Let $\varphi_0$ be a $\G2$ structure on $M$ with $g_{\varphi_0}=g_0$, let $c\in \{-1,0,1\}$, and let $X$ be a vector field such that
\begin{equation} \label{eq:iso-sol}
\begin{aligned}
\mathcal L_X g_0&=cg_0,\\
\Div _{g_0}T_{\varphi_0} &= -\frac{1}{2} \curl_{\varphi_0} X + X\hk T_{\varphi_0},
\end{aligned}
\end{equation}
That is, $(\varphi_0,X_0,c)$ is an isometric soliton. 
\begin{itemize}
\item If $c=1$, let $t<0$ and let $f_t: M\rightarrow M$ be a 1-parameter family of diffeomorphisms such that
\begin{align*}
\frac{\pt}{\pt t} f_t &= -\frac{1}{t} X\circ f_t, \\
f_{-1}&= \id_M.
\end{align*}
Then 
\begin{equation*}
\varphi(t)= |t|^{\frac{3}{2}} f_t^* \varphi_0
\end{equation*}
is a self-similar solution of the isometric flow, with $\varphi(-1)=\varphi_0$. Moreover, $(\varphi(t), |t|^{-1}X)$ satisfies
\begin{align*}
\mathcal L_{|t|^{-1} X} g_0 &= |t|^{-1} g_0,\\
\Div_{g_0} T_{\varphi(t)} &= -\frac{1}{2} \curl_{\varphi(t)} \left(|t|^{-1} X \right) +\left(|t|^{-1} X \right) \hk T_{\varphi(t)}.
\end{align*}
\item If $c=0$, let $t\in \mathbb R$ and let $f_t: M\rightarrow M$ be a 1-parameter family of diffeomorphisms such that
\begin{align*}
\frac{d}{dt} f_t &= X\circ f_t, \\
f_{0}&= \id_M.
\end{align*}
Then 
\begin{equation*}
\varphi(t)= f_t^* \varphi_0 
\end{equation*}
is a self-similar solution of the isometric flow, with $\varphi(0)=\varphi_0$. Moreover, $(\varphi(t), |t|^{-1} X)$ satisfies
\begin{align*}
\mathcal L_{|t|^{-1} X} g_0 &= 0, \\ 
 \Div_{g_0} T_{\varphi(t)}  &= -\frac{1}{2} \curl_{\varphi(t)} X + X  \hk T_{\varphi(t)}.
 \end{align*}
\item If $c=-1$, let $t>0$ and let $f_t: M\rightarrow M$ be a 1-parameter family of diffeomorphisms such that
\begin{align*}
\frac{d}{dt} f_t &= \frac{1}{t} X\circ f_t, \\
f_{1}&= \id_M.
\end{align*}
Then 
\begin{equation*}
\varphi(t)= |t|^{\frac{3}{2}} f_t^* \varphi_0
\end{equation*}
is a self-similar solution of the isometric flow, with $\varphi(1)=\varphi_0$. Moreover, $(\varphi(t), |t|^{-1}X)$ satisfies
\begin{align*}
\mathcal L_{|t|^{-1} X} g_0 &= -|t|^{-1} g_0,\\
\Div_{g_0} T_{\varphi(t)} &= -\frac{1}{2} \curl_{\varphi(t)} \left(|t|^{-1} X \right) +\left(|t|^{-1} X \right) \hk T_{\varphi(t)}.
\end{align*}
\end{itemize} 
In particular, the vector fields $X(t)$ in Lemma~\ref{lemma:selfsimilar} are $|t|^{-1} X$ or $X$, in the shrinking/expanding or steady case respectively.
\end{lemma}
\begin{proof}
We only prove the case $c=1$, $t<0$, since the other cases are similar. In this case we have
\begin{align*}
\frac{\pt}{\pt t} f_t^* \varphi_0 &= -f_t^*(\mathcal L_{t^{-1} X} \varphi_0), \\
\frac{\pt}{\pt t} f_t^* g_0 & = -f_t^*(\mathcal L_{t^{-1} X} g_0).
\end{align*}
Now $g(t)=|t| f_t^* g_0$ satisfies
\begin{align*}
\frac{\pt}{\pt t} g(t) &= -f_t^*g_0 -|t| f_t^* (\mathcal L_{t^{-1} X} g_0 )\\
&= -f_t^*g_0 +f_t^* g_0=0.
\end{align*}
Moreover, if $\varphi(t)= |t|^{\frac{3}{2}} f_t^* \varphi_0$ then
\begin{align*}
\frac{\pt}{\pt t} \varphi(t) &= -\frac{3}{2} |t|^{\frac{1}{2}} f_t^* \varphi_0 + |t|^{\frac{3}{2}} \frac{\pt}{\pt t} f_t^* \varphi_0\\
&=-\frac{3}{2|t|} \varphi(t) + |t|^{\frac{1}{2}} f_t^*(\mathcal L_X \varphi_0).
\end{align*}
Using~\eqref{Liederivative},~\eqref{gdiamondeq} and $\mathcal L_X g = g$ from~\eqref{eq:iso-sol}, we get
\begin{align*}
\frac{\pt}{\pt t} \varphi(t) &=-\frac{3}{2|t|} \varphi(t) + |t|^{\frac{1}{2}} f_t^* \Big( \frac{1}{2}\mathcal L_X g\diamond \varphi_0 + \big(-\frac{1}{2} \curl_{\varphi_0} X + X\hk T_{\varphi_0} \big) \hk \psi_0 \Big) \\
&=-\frac{3}{2|t|} \varphi(t) + |t|^{\frac{1}{2}} f_t^* \Big( \frac{3}{2} \varphi_0 + \big(-\frac{1}{2} \curl_{\varphi_0} X + X\hk T_{\varphi_0} \big) \hk \psi_0 \Big) \\
&=-\frac{3}{2|t|} \varphi(t) + \frac{3}{2|t|} \varphi(t) + |t|^{\frac{1}{2}} f_t^* \Big( \big(-\frac{1}{2} \curl_{\varphi_0} X + X\hk T_{\varphi_0} \big) \hk \psi_0 \Big).
\end{align*}
From the hypothesis~\eqref{eq:iso-sol} and the rescaling Lemma~\ref{lem:rescaling} we thus obtain
\begin{align*}
\frac{\pt}{\pt t} \varphi(t) &= |t|^{\frac{1}{2}} f_t^* \left( \Div T_{\varphi_0}\hk \psi_0\right) \\
&= \Div T_{\varphi(t)} \hk \psi(t).
\end{align*}
We conclude that $\varphi(t)$ is a self-similar isometric flow, with $\varphi(-1)=\varphi_0$.

Finally, again by Lemma~\ref{lem:rescaling} and the hypothesis~\eqref{eq:iso-sol} we have
\begin{align}
\begin{split}\label{eq:sol_times}
\Div_{g_0} T_{\varphi(t)} &= \Div_{|t| f_t^* g_0} T_{|t|^{\frac{3}{2}} f_t^* \varphi_0} \\
&=|t|^{-\frac{1}{2}} f_t^*( \Div_{g_0} T_{\varphi_0} ) \\
&=|t|^{-\frac{1}{2}} f_t^* \Big(-\frac{1}{2} \curl_{\varphi_0} X + X\hk T_{\varphi_0} \Big) \\
&=|t|^{1/2} f_t^*\left(-\frac{1}{2} \textrm{curl}_{\varphi_0} |t|^{-1}X + |t|^{-1} X \lrcorner T_{\varphi_0}\right)\\
&=-\frac{1}{2} \curl_{\varphi(t)} ((f_t^{-1})_* |t|^{-1} X) + ((f_t^{-1})_* |t|^{-1} X) \hk T_{\varphi(t)}.
\end{split}
\end{align}

We observe that 
\begin{equation*}
\frac{\partial}{\partial t} (f_t^{-1})_* X = (f_{t}^{-1})_* \left(\mathcal L_{|t|^{-1} X} X \right) = 0,
\end{equation*}
hence $(f_t^{-1})_* X= X$ for all $t<0$. This, together with~\eqref{eq:sol_times}, gives that
\begin{equation*}
\Div_{g_0} T_{\varphi(t)} = -\frac{1}{2} \curl_{\varphi(t)} (|t|^{-1} X) + (|t|^{-1} X) \hk T_{\varphi(t)},
\end{equation*}
completing the proof.
\end{proof}

\begin{rmk} \label{steady-rmk}
If $M$ is compact then every steady soliton in fact satisfies
\begin{equation*}
\Div T=0.
\end{equation*}
This is because $\varphi(t)=f_t^* \varphi_0$ satisfies $E(\varphi(t))=E(\varphi_0)$ for all $t$, and therefore by Proposition~\ref{gradient} we have
\begin{equation*}
\frac{d}{dt} 4E(\varphi(t))=-\int_M |\Div T|^2 d\mu_g=0.
\end{equation*}
It is unclear if there exist any nontrivial expanding or shrinking solitons in the compact case. This is an important question for future study.
\end{rmk}

We now restrict to the special case when $M = \R^7$ and $g = g_{\mathrm{Eucl}}$.
\begin{prop} \label{prop:solitons}
Let $(\g2, Y, c)$ be a soliton for the isometric flow on $\R^7$ with the Euclidean metric $g_{\mathrm{Eucl}}$. Then $Y = \tfrac{c}{2} x + Y_0$, where $x = x^i \frac{\pt}{\pt x^i}$ is the position (radial) vector field on $\R^7$ and $Y_0$ is a Killing vector field on $(\R^7, g_{\mathrm{Eucl}})$. That is, $Y_0$ induces an isometry of Euclidean space.
\end{prop}
\begin{proof}
In terms of the global coordinates $x^1, \ldots, x^7$ on $\R^7$, the equation $\mathcal L_Y g_{\mathrm{Eucl}} = c g_{\mathrm{Eucl}}$ becomes $\pt_i Y_j + \pt_j Y_i = c \delta_{ij}$. It is straightforward to verify that the only solutions are $Y_i = \tfrac{c}{2} x^i + a_{ij} x^j + b_i$ where $a_{ij}$ is skew-symmetric. Thus $Y_0 = a_{ij} x^j \frac{\pt}{\pt x^i} + b_i \frac{\pt}{\pt x^i}$ generates a rigid motion of $(\R^7, g_{\mathrm{Eucl}})$.
\end{proof}

A special class of solitons on $(\R^7, g_{\mathrm{Eucl}})$ are those for which $Y_0 = 0$. In this case, we have $Y = \frac{x}{2} = \frac{x^i}{2} \frac{\pt}{\pt x^i}$, so $(\curl Y)_m = \frac{1}{2} \nabla_i x^j \g2_{ijm} = \frac{1}{2}\delta_{ij} \g2_{ijm} = 0$. Hence, by Lemma~\ref{lemma:solitons} the special class of isometric \emph{shrinking solitons} $(\g2, Y)$ on $(\R^7, g_{\mathrm{Eucl}})$ for which $Y_0 = 0$ are precisely those $\g2$ which satisfy the equation
\begin{equation} \label{r7solitons}
\Div T = \frac{x}{2} \hk T.
\end{equation}
The particular special case of shrinking isometric solitons of the form~\eqref{r7solitons} arises in Theorem~\ref{almost_mon}. See Remark~\ref{almost_mon-solitons}.

It would be interesting to investigate whether any nontrivial examples of this special type of isometric soliton on $\R^7$ actually exist. One would need to solve the underdetermined equations~\eqref{r7solitons} on $\R^7$ under the additional constraint that $g_{\g2}=g_{\mathrm{Eucl}}$. Such solitons are important in the study of Type I singularities for the isometric flow. See Theorem~\ref{typeI} for more details.

\section{Derivative Estimates, Blow-Up Time, and Compactness} \label{sec:estimates}

In this section we first derive the global and local derivative estimates for the torsion $T$ (also known as Bando--Bernstein--Shi estimates) for the flow. We prove a doubling time estimate for the torsion (Proposition~\ref{dtestprop}), under the isometric flow which demonstrates that the assumption of a torsion bound is reasonable. Using the derivative estimates, in~\textsection\ref{lte}, we prove that any solution of the isometric flow exists as long as the torsion remains bounded, and we obtain a lower bound for the blow-up rate of the torsion. Finally, in~\textsection\ref{compactness} we prove a Cheeger--Gromov type compactness theorem for the solutions of the isometric flow.

\subsection{Global derivative estimates of torsion} \label{shiestsec}

Let $(M^7, \g2)$ be a compact manifold with $\G2$-structure and consider the evolution of $\g2$ by the isometric flow~\eqref{divtfloweqn}
\begin{equation*}
\frac{\pt \g2}{\pt t}=\Div T \hk \psi.
\end{equation*}
We first determine the evolution of the torsion under the flow~\eqref{divtfloweqn}. 

\begin{lemma}\label{evolution_of_torsion}
Let $\g2(t)$ be an isometric flow on $M$. Then the torsion evolves by
\begin{equation}
\frac{\pt T_{pq}}{\pt t} = \Delta T_{pq} - \del_iT_{pb}T_{ia}\g2_{abq}+F(\g2,T,\riem,\nabla \riem) \label{evolt3}
\end{equation}
where
\begin{equation} \label{evtoreq}
F(\g2,T,\riem,\nabla \riem)_{pq}= \del_aR_{bp}\g2_{abq}+R_{ipqm}T_{im} -\frac 12 R_{ipab}T_{im}\psi_{mabq}-R_{pm}T_{mq}.
\end{equation}
\end{lemma}
\begin{proof}
Recall from~\cite[Theorem 3.8]{kar1} that for a general flow of $\G2$-structures
\begin{align*}
\frac{\pt \g2}{\pt t}&= h\diamond \g2 + X\hk \psi
\end{align*} 
we have
\begin{align*}
\frac{\pt T_{pq}}{\pt t}&= T_{pl}h_{lq}+T_{pl}X_k\g2_{klq}+\del_kh_{ip}\g2_{kiq}+\del_pX_q.
\end{align*}
Hence for~\eqref{divtfloweqn}, where $h = 0$ and $X = \Div T$, we get
\begin{align}
\frac{\pt T_{pq}}{\pt t}&=T_{pl}(\Div T)_k\g2_{klq}+\del_p(\Div T)_q \nonumber \\
&=T_{pl}\del_iT_{ik}\g2_{klq}+\del_p\del_iT_{iq}. \label{evolT1}
\end{align}
We first compute $\Delta T_{pq}$. Using the $\G2$-Bianchi identity~\eqref{G2B} and the fact that $T_{ia} T_{im}$ is symmetric in $a,m$, we get
\begin{align*}
\del_i\del_iT_{pq}&=\del_i(\del_pT_{iq}+T_{ia}T_{pb}\g2_{abq}+ \frac 12R_{ipab}\g2_{abq}) \\
&= \del_i\del_pT_{iq}+\del_iT_{ia}T_{pb}\g2_{abq}+T_{ia}\del_iT_{pb}\g2_{abq}+T_{ia}T_{pb}T_{im}\psi_{mabq} \\
& \qquad+\frac 12 \del_iR_{ipab}\g2_{abq} + \frac 12 R_{ipab}T_{im}\psi_{mabq} \nonumber \\
&=\del_i\del_pT_{iq}+\del_iT_{ia}T_{pb}\g2_{abq}+T_{ia}\del_iT_{pb}\g2_{abq}+\frac 12\del_iR_{abip}\g2_{abq} + \frac 12 R_{ipab}T_{im}\psi_{mabq}.
\end{align*}
Applying the Riemannian second Bianchi identity to the fourth term above, we get
\begin{align*}
\del_i\del_iT_{pq}&=\del_i\del_pT_{iq}+\del_iT_{ia}T_{pb}\g2_{abq}+T_{ia}\del_iT_{pb}\g2_{abq}+\frac 12(-\del_aR_{biip}-\del_bR_{iaip})\g2_{abq} \\
& \qquad +\frac 12 R_{ipab}T_{im}\psi_{mabq} \nonumber \\
&=\del_i\del_pT_{iq}+\del_iT_{ia}T_{pb}\g2_{abq}+T_{ia}\del_iT_{pb}\g2_{abq}+\frac 12(\del_bR_{ap}-\del_aR_{bp})\g2_{abq} \\
& \qquad +\frac 12 R_{ipab}T_{im}\psi_{mabq} \nonumber \\
&=\del_i\del_pT_{iq}+\del_iT_{ia}T_{pb}\g2_{abq}+T_{ia}\del_iT_{pb}\g2_{abq}-\del_aR_{bp}\g2_{abq} +\frac 12 R_{ipab}T_{im}\psi_{mabq}.
\end{align*}
Commuting covariant derivatives for the first term above with the Ricci identity~\eqref{ricciidentityeq}, we get
\begin{equation} \label{LapT}
\begin{aligned}
\Delta T_{pq}&=\del_p\del_iT_{iq}+R_{pm}T_{mq}-R_{ipqm}T_{im}+\del_iT_{ia}T_{pb}\g2_{abq}+T_{ia}\del_iT_{pb}\g2_{abq} \\
& \qquad-\del_aR_{bp}\g2_{abq} + \frac 12 R_{ipab}T_{im}\psi_{mabq}
\end{aligned}
\end{equation}
Combining equations~\eqref{LapT} and~\eqref{evolT1}, we deduce that
\begin{align}
\frac{\pt T_{pq}}{\pt t} &= \Delta T_{pq} - \del_iT_{pb}T_{ia}\g2_{abq}+\del_aR_{bp}\g2_{abq}+R_{ipqm}T_{im} -\frac 12 R_{ipab}T_{im}\psi_{mabq} -R_{pm}T_{mq}, \label{evolT3}
\end{align}
as claimed.
\end{proof}

We write equation~\eqref{evolT3} schematically as
\begin{align} \label{schemT}
\frac{\pt}{\pt t} T= \Delta T+\del T *T* \g2+ \del \riem* \g2+ \riem* T+ \riem* T* \psi.
\end{align}
For a solution $\g2(t)$ of the isometric flow~\eqref{divtfloweqn}, define
\begin{equation} \label{mathcalTdefn}
\cT(t) = \underset{M}{\text{sup}}\ |T(x,t)|
\end{equation}
where $T(t)$ is the torsion of $\g2(t)$. We next prove a doubling time estimate for the quantity $\cT(t)$, which roughly says that $\cT(t)$ cannot blow up too quickly and therefore the assumption that $|T|$ is bounded for a short time is a reasonable one. Note that if $\cT(0) = 0$, then $\g2(0)$ is torsion-free, and does not flow under~\eqref{divtfloweqn}. Thus in the following proposition we can assume that $\cT(0) > 0$.

\begin{prop}[Doubling-time estimate] \label{dtestprop}
Let $\g2(t)$ be a solution to~\eqref{divtfloweqn} on a compact $7$-manifold $M$ for $t\in [0, \tau]$. Then there exists $\delta > 0$ such that
\begin{equation*}
\cT(t)\leq 2\cT(0) \quad \text{for all $0 \leq t \leq \delta$.}
\end{equation*} 
Moreover, $\delta$ satisfies $\delta \leq \min\{ \tau, \frac{1}{C \cT(0)^2} \}$ for some $C > 0$.
\end{prop}
\begin{proof}
If $| T | \leq 1$ at time $0$, then by continuity we have $| T | \leq 1 + \varepsilon$ for some small $\varepsilon$ for $0 \leq t \leq \delta < \tau$, and since $1 + \varepsilon \leq 2$, the assertion holds. Thus we can assume that $| T | > 1$ at time $0$, and thus by continuity we can assume that $| T | > 1$ for all $0 \leq t \leq \delta'$ for some $0 < \delta' < \tau$.

We first compute a differential inequality for $\cT(t)$ and then use the maximum principle. Since the metric is not evolving under~\eqref{divtfloweqn}, we have
\begin{equation*}
\frac{\pt}{\pt t}|T|^2 = \frac{\pt}{\pt t}(T_{ij}T_{pq}g^{ip}g^{jq}) = 2T_{pq}\frac{\pt T_{pq}}{\pt t},
\end{equation*}
so using~\eqref{schemT}, we obtain 
\begin{align}\label{evolnormT}
\frac{\pt}{\pt t}|T|^2 & = 2 \langle T, \frac{\pt}{\pt t}T \rangle \nonumber \\
& \leq \Delta |T|^2 -2|\del T|^2+C|\del T||T|^2+C|\del \riem||T|+C|\riem||T|^2
\end{align}
where $C$ is a constant. Now since the metric is not evolving and $M$ is compact, both $|\riem|$ and $|\del \riem|$ are bounded by some constant which we still call $C$. Thus we have
\begin{align}\label{evoltnormT1}
\frac{\pt}{\pt t}|T|^2 & \leq \Delta |T|^2-2|\del T|^2+C|\del T||T|^2+C|T|+C|T|^2.
\end{align}
Notice from~\eqref{evolT3} that the third term in~\eqref{evoltnormT1} is due to the $T*(\del T*T*\g2)$ term. We need to estimate this term by using the \emph{explicit} expression for $\del T*T*\g2$ rather than the schematic expression. Using the skew-symmetry of $\g2_{abq}$ in $a,q$ and the $\G2$-Bianchi identity~\eqref{G2B}, we have
\begin{align*}
T_{pq}\del_iT_{pb}T_{ia}\g2_{abq}&= \frac{1}{2}T_{pq}(\del_iT_{pb}-\del_pT_{ib})T_{ia}\g2_{abq}\nonumber \\
&=\frac 12T_{pq}(T_{im}T_{pn}\g2_{mnb}+\frac 12 R_{ipmn}\g2_{mnb})T_{ia}\g2_{abq},
\end{align*}
and hence~\eqref{evoltnormT1} becomes
\begin{align}\label{evolnormT2}
\frac{\pt}{\pt t}|T|^2 &\leq \Delta |T|^2 -2|\del T|^2+C|T|^4+C|T|+C|T|^2.
\end{align}
Since we have $|T | > 1$ for all $0 \leq t \leq \delta'$, we have $|T| < |T|^4$ and $|T|^2< |T|^4$ and hence~\eqref{evolnormT2} becomes
\begin{align}\label{evolnormT4}
\frac{\pt}{\pt t}|T|^2 & \leq \Delta |T|^2-2|\del T|^2+ C|T|^4.
\end{align}
Recall that $\cT(t)=\underset{M}{\text{sup}}\ |T(x,t)|$ is a Lipschitz function, so applying the maximum principle to~\eqref{evolnormT4}, we get
\begin{equation*}
\frac{d}{dt}\cT \leq \frac C2{\cT}^3
\end{equation*}
in the sense of the lim sup of forward difference quotients. Thus we have $\cT^{-3} \frac{\pt}{\pt t} \cT \leq \frac{C}{2}$. Integrating the inequality above from $0$ to $t$ we deduce that 
\begin{align}
\cT(t) \leq \cfrac{\cT(0)}{\sqrt{1-C\cT(0)^2t}}
\end{align}
and hence $\cT(t) \leq 2 \cT(0)$ for all $0 \leq t \leq \delta$ if we take $\delta = \min \Big \{\delta', \cfrac{3}{4C\cT(0)^2} \Big\}$. 
\end{proof}

Next we derive the Shi type estimates for the flow in~\eqref{divtfloweqn}.
\begin{thm}\label{shiestimatesthm}
Suppose that $K>0$ is a constant and $\g2(t)$ is a solution to the isometric flow on a closed manifold $M^7$ with $t\in [0, \frac{1}{K^2}]$. For all $m\in \mathbb{N}$, there exists a constant $C_m$ \emph{depending only on $(M, g)$} such that if
\begin{equation} \label{shiestimateshyp}
\cT \leq K \text{ and } |\del^j\riem|\leq B_jK^{2+j} \quad \text{for all $j\geq 0$ on $M^7\times [0, \tfrac{1}{K^2}]$},
\end{equation}
then for all $t\in [0, \frac{1}{K^2}]$ we have
\begin{equation} \label{shiestimateseqn}
|\del^mT| \leq C_mt^{-\frac{m}{2}}K.
\end{equation}
\end{thm}
Before we give the proof of Theorem~\ref{shiestimatesthm}, we remark that the form of the assumed bounds on $\del^j\riem$ in~\eqref{shiestimateshyp} is precisely as required by the rescaling properties of the curvature in equation~\eqref{rescaling}.
\begin{proof}[Proof of Theorem~\ref{shiestimatesthm}]
Since the proof is quite long, we first summarize the strategy of the proof. The proof is by induction on $m$. We first define a function $f_m(x,t)$ (see~\eqref{fmdefn} for the precise expression) for each $m$, just as in the case of Ricci flow, which satisfies a parabolic differential inequality, and then we use the maximum principle.

For $m=1$ case, we define 
\begin{equation} \label{fdefn}
f=t|\del T|^2+\beta |T|^2
\end{equation}
where $\beta$ is a constant to be determined later. Note that $f(x,0)\leq \beta K^2$. To calculate the evolution of $f$, we first need to calculate the evolution of $|\del T|^2$.

Because the metric is not evolving, by differentiating~\eqref{schemT} we have that
\begin{align*}
\frac{\pt}{\pt t}\del T & = \del (\Delta T+\del T*T*\g2+\del \riem*\g2+\riem*T+\riem*T*\psi) \\
&= \del \Delta T + \del(\del T*T*\g2)+\del^2\riem*\g2+\del \riem*\del \g2 +\del \riem*T+\riem*\del T \\
& \qquad + \del \riem*T*\psi + \riem*\del T*\psi + \riem*T*\del \psi \\
&= \Delta \del  T + \del(\del T*T*\g2)+\del^2\riem*\g2+\del \riem*\del \g2 +\del \riem*T+\riem*\del T \\
& \qquad + \del \riem*T*\psi + \riem*\del T*\psi + \riem*T*\del \psi
\end{align*}
where we have used the Ricci identity in the last equality. Thus we have
\begin{align}
\frac{\pt}{\pt t}|\del T|^2 = 2 \langle \del T, \frac{\pt}{\pt t} \del T \rangle & = \del T* \Big( \Delta \del T + \del(\del T*T*\g2)+\del^2\riem*\g2+\del \riem*\del \g2 +\del \riem*T \nonumber \\
& \qquad \qquad +\riem*\del T + \del \riem*T*\psi + \riem*\del T*\psi + \riem*T*\del \psi \Big) \nonumber \\
& = \Delta |\del T|^2 -2|\del^2T|^2+\del(\del T*T*\g2)*\del T+ \del T*\del^2\riem*\g2 \nonumber \\
& \qquad +\del \riem*\del \g2*\del T +\del \riem*T*\del T + \riem*\del T*\del T \nonumber \\
& \qquad +\del \riem*T*\psi*\del T + \riem*\del T*\del T*\psi + \riem*T*\del \psi*\del T \label{evoldelT1}
\end{align}
From~\eqref{del34eq} we have
\begin{equation} \label{del34eq2}
|\del \g2|\leq CK, \qquad |\del \psi|\leq CK.
\end{equation}
Using~\eqref{del34eq2} and the hypotheses~\eqref{shiestimateshyp} of the theorem, the estimate~\eqref{evoldelT1} becomes
\begin{align}
\frac{\pt}{\pt t}|\del T|^2 & \leq \Delta |\del T|^2-2|\del^2T|^2 + C|\del^2T||T||\del T| +C|\del T|^3 + C|\del T|^2|T|^2  \nonumber \\
& \qquad +C|\del T||\del^2\riem|+C|\del \riem||\del \g2||\del T|+C|\del \riem||T||\del T|+C|\riem||\del T|^2 \nonumber \\
& \qquad +C|\riem||T||\del \psi||\del T| \nonumber \\
& \leq \Delta |\del T|^2 -2|\del^2T|^2 +CK|\del ^2T||\del T|+C|\del T|^3 +CK^2|\del T|^2+C|\del T|K^4  \label{evoldelT2}
\end{align}
for some constant $C$ depending only on the dimension and the order of the derivative. Consider the third term in the right hand side of the inequality~\eqref{evoldelT2}. By Young's inequality, for all $\varepsilon >0$, we have
\begin{equation*}
2K|\del^2T| |\del T| \leq \frac{1}{\varepsilon}K^2|\del T|^2 + \varepsilon |\del^2T|^2 
\end{equation*}
Substituting this into~\eqref{evoldelT2} gives
\begin{align} \label{evoldelT6}
\frac{\pt}{\pt t}|\del T|^2 & \leq \Delta |\del T|^2 - (2-C\varepsilon)|\del^2T|^2 +CK^4|\del T|+CK^2|\del T|^2+C|\del T|^3.
\end{align}

We pause here for an important remark. In the Shi-type estimates for the Laplacian flow of Lotay--Wei~\cite{lotay-wei1}, they assume a bound on $| \del T|$. In contrast, we only assume a bound on $|T|$, not $|\del T|$. This remark has the following consequence. It turns out that the third and fourth terms in~\eqref{evoldelT6} can be dealt with easily, which we do below. However, the presence of the $|\del T|^3$ term on the right hand side of~\eqref{evoldelT6} would cause problems in trying to apply the maximum principle to the function $f$ and cannot be dealt with easily, so we have to work harder. Notice from~\eqref{evoldelT1} that the $|\del T|^3$ term comes from the $\del T*\del(\del T*T*\g2)$ term. We get rid of the problematic term by considering the \emph{explicit} expression for $\del T*\del(\del T*T*\g2)$ rather than the schematic one, and using the $\G2$-Bianchi identity~\eqref{G2B} to get a lower order term. Specifically, the expression for $\del T*T*\g2$ is $\del_iT_{pb}T_{ia}\g2_{abq}$. So we have
\begin{align*}
\del T * \del (\del T*T*\g2)&= \del_jT_{pq}\del_j(\del_iT_{pb}T_{ia}\g2_{abq}) \\
&= \del_jT_{pq}\del_j\del_iT_{pb}T_{ia}\g2_{abq}+\del_jT_{pq}\del_iT_{pb}\del_jT_{ia}\g2_{abq} \\
& \quad +\del_jT_{pq}\del_iT_{pb}T_{ia}\del_j\g2_{abq}
\end{align*}
Since the first and the last term in the above equation do not cause any problems in~\eqref{evoldelT6}, we focus on the second term. Using the fact that $\g2_{abq}$ is skew-symmetric in $a,q$, and the $\G2$-Bianchi identity~\eqref{G2B}, we have
\begin{align}
\del_jT_{pq}\del_iT_{pb}\del_jT_{ia}\g2_{abq}&= \frac 12\del_jT_{pq}\del_jT_{ia}(\del_iT_{pb}-\del_pT_{ib})\g2_{abq} \nonumber \\
&= \frac 12\del_jT_{pq}\del_jT_{ia}(T_{im}T_{pn}\g2_{mnb}+\frac 12 R_{ipmn}\g2_{mnb})\g2_{abq}. \label{evoldelT8}
\end{align}
Thus from~\eqref{evoldelT6} and~\eqref{evoldelT8} and using Young's inequality as before we get
\begin{equation} \label{evoldelT7}
\frac{\pt}{\pt t}|\del T|^2 \leq \Delta |\del T|^2 - (2-C\varepsilon)|\del^2T|^2 +CK^4|\del T|+CK^2|\del T|^2.
\end{equation}
Hence, with a suitably chosen $\varepsilon$ we have
\begin{equation} \label{evoldelT7'}
\frac{\pt}{\pt t}|\del T|^2 \leq \Delta |\del T|^2 + CK^2|\del T|^2 +CK^4|\del T|.
\end{equation}
From~\eqref{evolnormT} and~\eqref{evoldelT7'}, we get
\begin{align*}
\frac{\pt f}{\pt t} & \leq \Delta f +t(CK^4|\del T|+CK^2|\del T|^2) \\
& \qquad +\beta (-2|\del T|^2+C|\del T||T|^2+C|\del \riem||T|+C|\riem||T|^2).
\end{align*}

Using the hypotheses that $\cT = \underset{M}{\text{sup}} \ T(x,t) \leq K$, $|\del^j\riem|\leq K^{2+j}$, and $tK^2\leq 1$, and using Young's inequality on the $|\del T||T|^2$ term, the above inequality becomes
\begin{equation*}
\frac{\pt f}{\pt t} \leq \Delta f + CK^2|\del T|+C|\del T|^2-(2-\varepsilon)\beta |\del T|^2 + C\beta K^4.
\end{equation*}
Using Young's inequality again on the second term above we get
\begin{equation*}
\frac{\pt f}{\pt t} \leq \Delta f+(C -(2-\varepsilon) \beta)|\del T|^2 +C\beta K^4.
\end{equation*}
Now choose $\beta$ large enough so that $C-(2-\varepsilon)\beta \leq 0$, so we have
\begin{equation*}
\frac{\pt f}{\pt t} \leq \Delta f+C\beta K^4.
\end{equation*}
From~\eqref{fdefn} we have $f(x,0) \leq \beta K^2$. Thus, applying the maximum principle to the above inequality and using $t K^2 \leq 1$, we get
\begin{equation}
\underset{x\in M}{\text{sup}}\ f(x,t) \leq \beta K^2 + C\beta tK^4 \leq CK^2.
\end{equation}
From the definition~\eqref{fdefn} of $f$, we conclude that
\begin{align*}
|\del T| \leq CKt^{-\frac {1}{2}}
\end{align*} 
and thus the base case of the induction is complete.

Next we prove the estimate for $m\geq 2$ by induction. Suppose $|\del^jT|\leq C_jKt^{-\frac{j}{2}}$ holds for all $1\leq j<m$. Looking at the definition of $f_m$ in~\eqref{fmdefn} below, it is clear that we need to first determine the evolution equation for $|\del^m T|^2$. Since the metric is not evolving, by differentiating~\eqref{schemT} we have that
\begin{align*}
\frac{\pt}{\pt t}\del^m T&= \del^m(\Delta T+\del T *T* \g2+ \del \riem* \g2+ \riem* T+ \riem* T* \psi) \\
& = \del^m \Delta T + \del^m (\del T*T*\g2) + \sum_{i=0}^m \del^{m+1-i}\riem* \del^i \g2  \nonumber \\
& \qquad + \sum_{i=0}^m \del^{m-i}T* \del^i \riem + \sum_{i=0}^m \del^{m-i}(\riem*T)*\del^i \psi.
\end{align*}
Using the identity~\eqref{riccischematic} with $S = T$, we can write the above equation as
\begin{align}
\frac{\pt}{\pt t}\del^m T&= \Delta \del^m T + \sum_{i=0}^{m} \del^{m-i}T* \del^i \riem  + \del^m(\del T*T*\g2) \nonumber \\
& \qquad + \sum_{i=0}^m \del^{m+1-i}\riem* \del^i \g2 + \sum_{i=0}^m \del^{m-i}(\riem*T)*\del^i \psi. 
\label{evoldelkT1}
\end{align}
Thus we find that
\begin{align}
\frac{\pt}{\pt t}|\del^m T|^2 = 2 \langle \del^m T, \frac{\pt}{\pt t} \del^m T \rangle &= \Delta |\del^mT|^2-2|\del^{m+1}T|^2 + \sum_{i=0}^{m} \del^mT* \del^{m-i}T* \del^i \riem  \nonumber \\
& \qquad + \del^mT*\del^m(\del T*T*\g2) + \sum_{i=0}^m \del^mT*\del^{m+1-i}\riem* \del^i \g2 \nonumber \\
& \qquad + \sum_{i=0}^m \del^mT*\del^{m-i}(\riem*T)*\del^i \psi \label{evolnormdelkT1}
\end{align}
Using the induction hypothesis, we estimate each term in~\eqref{evolnormdelkT1} as follows.

Consider the third term $\sum_{i=0}^{m} \del^{m}T*\del^{m-i} T* \del^i \riem$. When $i=0$ we get
\begin{equation*}
|\del^mT*\del^mT*\riem| \leq CK^2|\del^mT|^2.
\end{equation*}
When $1\leq i\leq m$, using $K^2 t \leq 1$ and the induction hypothesis, we get
\begin{align*}
\Big |\sum_{i=1}^{m} \del^mT* \del^{m-i}T* \del^i \riem \Big | & \leq C|\del^mT|\sum_{i=1}^{m}|\del^{m-i}T||\del^i\riem| \\
& \leq C|\del^mT| \sum_{i=1}^m Kt^{-\frac{(m-i)}{2}}K^{i+2} \\
& \leq CK^3t^{-\frac{m}{2}}|\del^mT|.
\end{align*}
Thus the third term in~\eqref{evolnormdelkT1} can be estimated as
\begin{equation}
\Big|\sum_{i=0}^{m} \del^mT* \del^{m-i}T* \del^i \riem \Big| \leq CK^2|\del^mT|^2+CK^3|\del^mT|t^{-\frac{m}{2}} \label{evolnormdelkT2}
\end{equation}
For the moment we skip the fourth term in~\eqref{evolnormdelkT1} and consider the fifth and sixth terms. We need to first estimate the quantities $\del^i \psi$ and $\del^i \g2$. From~\eqref{del34eq} we have $\del \psi = T * \g2$, and thus
\begin{equation*}
|\del \psi| \leq CK.
\end{equation*}
Schematically have
\begin{equation*}
\del^2 \psi = \del T*\g2 + T*T*\psi
\end{equation*}
and hence
\begin{equation*}
|\del^2 \psi| \leq C(|\del T|+|T|^2) \leq C(Kt^{-\frac{1}{2}}+K^2)=CK(t^{-\frac{1}{2}}+K).
\end{equation*}
Using the same equations again, we have
\begin{equation*}
\del^3 \psi = \del^2T*\g2+\del T*T*\psi + T*T*T*\g2
\end{equation*}
and therefore
\begin{align*}
|\del^3 \psi| \leq C(|\del^2T|+|\del T||T|+|T|^3) \leq CK(t^{-1}+Kt^{-\frac 12}+K^2).
\end{align*}
Similarly, we have
\begin{equation*}
\del^4 \psi = \del^3T*\g2+\del^2T*T*\psi+\del T*\del T*\psi+\del T*T*T*\g2 +T*T*T*T*\psi
\end{equation*}
thus yielding, using the induction hypothesis, that
\begin{align*}
|\del^4 \psi| & \leq C(|\del^3T|+|\del^2T||T|+|\del T|^2+|\del T||T|^2+|T|^4) \\
& \leq CK(t^{-\frac 32}+Kt^{-1}+K^2t^{-\frac 12}+K^3).
\end{align*}
A straightforward induction argument which we omit then shows that for $i\geq 1$ we have
\begin{equation} \label{delpsiestm}
|\del^i \psi| \leq C \sum_{j=1}^{i} K^{j}t^{\frac{j-i}{2}}.
\end{equation}
Because $\g2$ is the Hodge star of $\psi$, and the Hodge star is both parallel and an isometry, we deduce the same estimates for $|\del^i \g2|$ for $i\geq 1$. That is, we have
\begin{equation} \label{delphiestm}
|\del^i \g2| \leq C \sum_{j=1}^{i} K^{j}t^{\frac{j-i}{2}}.
\end{equation}
Using the hypotheses~\eqref{shiestimateshyp} on $|\del^j\riem|$, equation~\eqref{delphiestm}, and $K^2t\leq 1$, the fifth term in~\eqref{evolnormdelkT1} can thus be estimated as
\begin{equation} \label{evolnormdelkT4}
\Big |\sum_{i=0}^{m}\del^mT*\del^{m+1-i}\riem*\del^i\g2 \Big | \leq CK^3|\del^mT|t^{-\frac{m}{2}}.
\end{equation}

Next consider the expression $\del^{m-i}(\riem*T)$, which is part of the sixth term of~\eqref{evolnormdelkT1}. Using the induction hypothesis and $|\riem|\leq K^2$, for $i=0$ we get
\begin{equation*}
|\del^m(\riem*T)| = \Big |\sum_{j=0}^m \del^{m-j}\riem*\del^jT \Big| \leq CK^2|\del^mT|+CK^3t^{-\frac{m}{2}}
\end{equation*}
and for $1\leq i\leq m$ we get
\begin{align*}
|\del^{m-i}(\riem*T)|\leq \Big |\sum_{j=0}^{m-i}\del^{m-i-j}\riem*\del^jT \Big| \leq CK^3t^{\frac{(i-m)}{2}}.
\end{align*}
Hence, using~\eqref{delpsiestm} and the above two estimates, we get
\begin{align*}
\Big |\sum_{i=0}^{m}\del^mT*\del^{m-i}(\riem*T)*\del^i\psi \Big| &\leq CK^2|\del^mT|^2+CK^3|\del^mT|t^{-\frac{m}{2}} \nonumber \\
& \qquad + C|\del^mT|\sum_{i=1}^m \Big( K^3t^{\frac{(i-m)}{2}}\sum_{j=1}^{i}K^jt^{\frac{(j-i)}{2}} \Big).
\end{align*}
Using $K^2t\leq 1$ on the above, the sixth term in~\eqref{evolnormdelkT1} can be estimated as
\begin{equation}
\Big |\sum_{i=0}^{m}\del^mT*\del^{m-i}(\riem*T)*\del^i\psi \Big | \leq CK^2|\del^mT|^2+CK^3|\del^mT|t^{-\frac{m}{2}}. \label{evolnormdelkT5}
\end{equation}

Finally we return to the fourth term in~\eqref{evolnormdelkT1}. We have
\begin{equation*}
|\del^mT*\del^m(\del T*T*\g2)| = \Big |\del^mT*\sum_{i=0}^m\del^{m+1-i}T*\del^{i}(T*\g2)\Big |.
\end{equation*}
We break up the sum over $i$ into four terms: $i = 0$, $i=1$, $1<i<m$, and $i = m$. Thus we have
\begin{align*}
|\del^mT*\del^m(\del T*T*\g2)| & \leq | \del^mT*\del^{m+1}T*T*\g2|+ |\del^mT*\del^mT*(\del T*\g2 + T*T*\psi)| \\
& \qquad + \Big |\del^mT*\sum_{i=2}^{m-1} \del^{m+1-i}T*(\sum_{j=0}^i \del^{i-j}T*\del^j\g2)  \Big | \\
& \qquad  + \Big |\del^mT*\del T* \sum_{i=0}^m \del^{m-i}T*\del^i\g2 \Big |.
\end{align*}
Using the induction hypothesis and equation~\eqref{delphiestm} on the above, the fourth term in~\eqref{evolnormdelkT1} can be estimated as
\begin{align}
|\del^mT*\del^m(\del T*T*\g2)|& \leq CK|\del^mT||\del^{m+1}T| + CKt^{-\frac 12}|\del^mT|^2 + CK^2|\del^mT|^2 \nonumber \\
& \qquad + CK^2t^{-\frac{(m+1)}{2}}|\del^mT|. \label{evolnormdelkT6}
\end{align}

Combining the estimates~\eqref{evolnormdelkT2},~\eqref{evolnormdelkT4},~\eqref{evolnormdelkT5}, and~\eqref{evolnormdelkT6}, equation~\eqref{evolnormdelkT1} thus becomes
\begin{align}
\frac{\pt}{\pt t}|\del^mT|^2 & \leq \Delta |\del^mT|^2-2|\del^{m+1}T|^2 +CK^2|\del^mT|^2 +CK|\del^{m+1}T||\del^mT| \nonumber \\
& \qquad +CK^3|\del^mT|t^{-\frac m2} + CKt^{-\frac 12}|\del^mT|^2 + CK^2t^{-\frac{(m+1)}{2}}|\del^mT|. \label{evolnormdelkT11}
\end{align}

Using Young's inequality for the fourth term in~\eqref{evolnormdelkT11}, we know that for any $\varepsilon >0$ we have
\begin{align*}
K|\del^{m+1}T||\del^mT| \leq \frac{K^2}{2\varepsilon}|\del^mT|^2+\frac{\varepsilon}{2} |\del^{m+1}T|^2
\end{align*}
and hence
\begin{align}\label{evolnormdelkT12}
\frac{\pt}{\pt t}|\del^mT|^2 & \leq \Delta |\del^mT|^2-(2-\frac{C\varepsilon}{2})|\del^{m+1}T|^2 +CK^2|\del^mT|^2 \nonumber \\
& \qquad +CK^3|\del^mT|t^{-\frac m2} + CKt^{-\frac 12}|\del^mT|^2 + CK^2t^{-\frac{(m+1)}{2}}|\del^mT|.
\end{align}
Hence for suitably chosen $\varepsilon$, we deduce that
\begin{align}
\frac{\pt}{\pt t}|\del^mT|^2 & \leq \Delta |\del^mT|^2-|\del^{m+1}T|^2 +CK^2|\del^mT|^2 +CK^3|\del^mT|t^{-\frac{m}{2}} \nonumber \\
& \qquad + CKt^{-\frac{1}{2}}|\del^mT|^2 + CK^2t^{-\frac{(m+1)}{2}}|\del^mT|. \label{evolnormdelkTfinal}
\end{align}
The derivation of~\eqref{evolnormdelkTfinal} in fact holds for $m$ replaced by $m-k$ for any $1 \leq k \leq m-1$. That is, we also have
\begin{align*}
\frac{\pt}{\pt t}|\del^{m-k}T|^2 & \leq \Delta |\del^{m-k}T|^2-|\del^{m+1-k}T|^2 +CK^2|\del^{m-k}T|^2 +CK^3|\del^{m-k}T|t^{-\frac{m-k}{2}} \\
& \qquad + CKt^{-\frac{1}{2}}|\del^{m-k}T|^2 + CK^2t^{-\frac{(m+1-k)}{2}}|\del^{m-k}T|
\end{align*}
for $1 \leq k \leq m-1$. Using the induction hypothesis that~\eqref{shiestimateseqn} holds for all $1\leq k \leq m-1$, the above inequality becomes 
\begin{equation} \label{evolnormdeliTfinal}
\frac{\pt}{\pt t}|\del^{m-k}T|^2 \leq \Delta |\del^{m-k}T|^2 - |\del^{m+1-k}T|^2+CK^4t^{-(m-k)}+CK^3t^{-\frac{1}{2}}t^{-(m-k)}
\end{equation}
for all $k<m$. We emphasize here that we needed to use the induction hypothesis to get our simplified evolution inequality~\eqref{evolnormdeliTfinal} when $1 \leq k \leq m-1$.

With these computations in hand, we define
\begin{equation} \label{fmdefn}
f_m = t^m|\del^mT|^2+\beta_m\sum_{k=1}^{m}\alpha_k^mt^{m-k}|\del^{m-k}T|^2
\end{equation}
for some positive constants $\beta_m$ to be chosen later, where $\alpha_k^m=\frac{(m-1)!}{(m-k)!}$.

Using~\eqref{evolnormdelkTfinal} and~\eqref{evolnormdeliTfinal} we compute that
\begin{align*}
\frac{\pt}{\pt t}f_m & = t^m\frac{\pt}{\pt t}|\del^mT|^2+mt^{m-1}|\del^mT|^2+\beta_m\sum_{k=1}^m\alpha_k^mt^{m-k}\frac{\pt}{\pt t}|\del^{m-k}T|^2 \\
& \qquad +\beta_m\sum_{k=1}^m(m-k)\alpha_k^mt^{m-k-1}|\del^{m-k}T|^2 \\
& \leq t^m \Big( \Delta |\del^mT|^2-|\del^{m+1}T|^2 +CK^2|\del^mT|^2 +CK^3|\del^mT|t^{-\frac m2} \\ & \qquad \qquad + CKt^{-\frac 12}|\del^mT|^2 + CK^2t^{-\frac{(m+1)}{2}}|\del^mT| \Big) \\
& \qquad + mt^{m-1}|\del^mT|^2 + \beta_m \sum_{k=1}^m (m-k)\alpha_k^mt^{m-k-1}|\del^{m-k}T|^2 \\& \qquad +\beta_m\sum_{k=1}^m \alpha_k^mt^{m-k}(\Delta |\del^{m-k}T|^2 - |\del^{m+1-k}T|^2+CK^4t^{-(m-k)} + CK^3t^{-\frac 12}t^{-(m-k)}).
\end{align*}
Observe that in the first summation above, the term for $k=m$ vanishes. We reindex the second term in the last line above to sum from $k=0$ to $k=m-1$, and throw away the negative term corresponding to $k=0$. Collecting terms, the above then becomes
\begin{align*}
\frac{\pt}{\pt t}f_m & \leq \Delta f_m +(CK^2t^m + mt^{m-1} +CKt^{\frac{(2m-1)}{2}}-\beta_m\alpha_1^mt^{m-1})|\del^mT|^2+CK^3|\del^mT|t^{\frac{m}{2}} \\
& \qquad +CK^2t^{\frac{(m-1)}{2}}|\del^mT| +\beta_m \sum_{k=1}^{m-1} \big ((m-k)\alpha_k^m - \alpha_{k+1}^m \big )t^{m-k-1}|\del^{m-k}T|^2 \\
& \qquad + C\beta_m\sum_{k=1}^m\alpha_k^m (K^4+K^3t^{-\frac{1}{2}}).
\end{align*}
Using Young's inequality on the third and the fourth terms above, we have
\begin{equation*}
CK^3|\del^mT|t^{\frac m2} \leq CK^4 + CK^2|\del^mT|^2t^m
\end{equation*}
and
\begin{equation*}
CK^2t^{\frac{(m-1)}{2}}|\del^mT| \leq Ct^{m-1}|\del^mT|^2 + CK^4,
\end{equation*}
and hence we obtain
\begin{align*}
\frac{\pt}{\pt t}f_m & \leq \Delta f_m +(CK^2t^m+mt^{m-1}+CKt^{\frac{(2m-1)}{2}}+Ct^{m-1}-\beta_m\alpha_1^mt^{m-1})|\del^mT|^2 \\
& \qquad +\beta_m \sum_{k=1}^m \big ((m-k)\alpha_k^m - \alpha_{k+1}^m \big )t^{m-k-1}|\del^{m-k}T|^2 + C\beta_m\sum_{k=1}^m\alpha_k^m (K^4+K^3t^{-\frac{1}{2}}).
\end{align*}
Now we choose $\beta_m$ sufficiently large and use the fact that $(m-k)\alpha_k^m-\alpha_{k+1}^m=0$ for $1 \leq k \leq m-1$ to deduce that
\begin{equation} \label{evolfm}
\frac{\pt}{\pt t}f_m \leq \Delta f_m + CK^4 + CK^3t^{-\frac{1}{2}}.
\end{equation}

Since $m \geq 1$, from the definition~\eqref{fmdefn} of $f_m$ we have that $f_m(0) = \beta_m\alpha_m^m|T|^2 \leq \beta_m\alpha_m^m K^2$, so applying the maximum principle to~\eqref{evolfm} and using $K^2 t \leq 1$ gives
\begin{equation*}
\underset{x\in M}{\text{sup}} f_m(x,t)\leq \beta_m\alpha_m^mK^2+CK^4t + CK^3t^{\frac 12} \leq CK^2.
\end{equation*}
From the definition~\eqref{fmdefn} of $f_m$, we finally conclude that
\begin{align*}
|\del^mT| \leq CK t^{-\frac{m}{2}},
\end{align*}
and the inductive step is complete.
\end{proof}

One of our goals is to study the long-time existence of the flow. We seek a criterion that characterizes the blow-up time for the flow. This will be established in Theorem~\ref{ltethm} later. In order to prove Theorem~\ref{ltethm} later, we require the following corollary to Theorem~\ref{shiestimatesthm}, whose proof is an adaptation of the argument in the case of Ricci flow, and can be found in~\cite[\textsection 6.7]{chow-knopf}.

\begin{corr} \label{shiestcorr}
Let $(M^7, \g2(t))$ be a solution to the isometric flow. Suppose there exists $K>0$ such that
\begin{equation*}
|T(x,t)|_g \leq K \qquad \text{for all $x \in M$ and $t \in [0,\tau]$},
\end{equation*}
where $\tau > \frac{1}{K^2}$ and $|\del^j\riem|\leq C_jK^{2+j}$ for all $j\geq 0$. Then for all $m \in \mathbb{N}$ there exists a constant $C_m$ \emph{depending only on $(M,g)$} such that
\begin{equation} \label{shiestcorreqn}
|\del^m T(x,t)|_g \leq C_mK^{1+m} \qquad \text{for all $x \in M$ and $t \in [\tfrac{1}{K^2}, \tau]$.}
\end{equation} 
\end{corr} 
\begin{proof}
Fix $t_0\in \Big[ \frac{1}{K^2}, \tau \Big]$ and let $\tau_0=t_0-\frac{1}{K^2}$. Let $\bar{t}=t-\tau_0$ and let $\bar{\g2}(\bar{t})$ solve the Cauchy problem
\begin{align*}
\frac{\pt}{\pt t}\bar{\g2}(\bar{t}) & = \Div \bar{T} \hk \bar{\psi} \\
\bar{\g2}(0)&= \g2(\tau_0)
\end{align*}
Then by the uniqueness of solutions to the isometric flow given in Theorem~\ref{stethm}, we deduce that $\bar{\g2}(\bar{t})=\g2(\bar{t}+\tau_0)=\g2(t)$ for $\bar{t}\in [ 0, \tfrac{1}{K^2} ]$. So by the hypothesis on the solution $\g2(t)$, we have
\begin{equation*}
|\bar{T}(x,\bar{t})|\leq K \qquad \text{ for all $x \in M$ and $\bar{t} \in [ 0, \tfrac{1}{K^2} ]$.}
\end{equation*}
Applying Theorem~\ref{shiestimatesthm} we have constants $\bar{C}_m$ depending only on $m$ such that
\begin{equation*}
|\bar{\del}^m \bar{T}(x, \bar{t})| \leq \bar{C}_mK\bar{t}^{-\frac m2}
\end{equation*}
for all $x\in M$ and $\bar{t} \in ( 0, \tfrac{1}{K^2} ]$.

Now when $\bar{t}\in [ \tfrac{1}{2K^2}, \frac{1}{K^2} ]$ then
\begin{equation*}
\bar{t}^{\frac{m}{2}} \geq 2^{-\frac{m}{2}}K^{-m},
\end{equation*}
so taking $\bar{t}=\frac{1}{K^2}$, we find that
\begin{equation*}
|\del^m T(x, t_0)| \leq 2^{\frac{m}{2}}\bar{C}_mK^{1+m} \qquad \text{for all $x \in M$.}
\end{equation*}
Since $t_0 \in [ \tfrac{1}{K^2}, \tau ]$ was arbitrary, we obtain~\eqref{shiestcorreqn}.
\end{proof}

\subsection{Local estimates of the torsion}\label{localest}

In this section we prove the local estimates on the derivatives of the torsion. The proof is similar to the local bounds on the higher derivatives of a solution of the harmonic map heat flow by Grayson--Hamilton~\cite{HamGray96} and to the local derivative estimates of the curvature for the Yang-Mills flow which was proved by Weinkove~\cite{weinkove}. We first define the parabolic cylinder
\begin{equation*}
P_r(x_0, t_0) = \{ (x,t)\in M\times \mathbb{R} \mid d(x,x_0) \leq r, t_0 - r^2 \leq t \leq t_0 \}.
\end{equation*}
We need the following lemma, which is proved in~\cite[Lemma 2.1]{HamGray96}. We state the particular version that is given in~\cite[Lemma 2.1]{weinkove}.

\begin{lemma} \label{localestlemma}
Let $M$ be a compact manifold and $F$ be a smooth function on $M \times [0, \infty)$. Let $x_0 \in M$ and $t_0 \geq 0$. There exists a constant $s>0$ and for every $\gamma <1$ a constant $C_{\gamma}$, such that the following holds. Let $r \leq s$. If at any point in the parabolic cylinder $P_r(x_0, t_0)$ for which $F \geq 0$, we have
\begin{equation*}
\frac{\pt F}{\pt t} \leq \Delta F - F^2,
\end{equation*}
then
\begin{equation*}
F \leq \frac{C_{\gamma}}{r^2}
\end{equation*}
on the smaller parabolic cylinder $P_{\gamma r}(x_0,t_0)$.
\end{lemma}

\begin{rmk} \label{localestlemma_ncpct}
From the proof of~\cite[Lemma 2.1]{HamGray96} we deduce that Lemma~\ref{localestlemma} in fact also holds when $M$ is complete, noncompact, with \emph{bounded geometry}. That is, we require that there are $D_m<+\infty$ for $m\geq 0$, and $i_0>0$ such that
\begin{align*}
|\nabla^m \riem| &\leq D_m \qquad \textrm{in } M,\\
\inj(M,g) &\geq i_0.
\end{align*}
This observation is used for the noncompact case of Lemma~\ref{monotonicity_formula}.
\end{rmk}

We now state and prove the local estimates for the derivatives of the torsion.
\begin{thm}\label{localestthm}
Let $\g2(t)$ be a solution to the isometric flow on $M^7$. Let $x_0 \in M$ and $t_0 \geq 0$ such that $\g2(t)$ is defined at least up to time $t_0$. There exists a constant $s>0$ and constants $C_m$ for $m\geq 1$ such that the following holds. Whenever $|T|\leq K$ and $|\del^j \riem|\leq B_j K^{2+j}$ for all $j\geq 0$ in some parabolic cylinder $P_{r}(x_0, t_0)$ with $r\leq s$ and $K\geq \frac{1}{r^2}$, then we have
\begin{align} \label{localestthmeqn}
|\del^m T| \leq C_m K^{m+1}
\end{align}
on the much smaller parabolic cylinder $P_{\frac{r}{2^m}}(x_0, t_0)$.
\end{thm}
\begin{proof}
The proof is similar to the proof of Theorem~\ref{shiestimatesthm} and is by induction on $m$. We have already derived all the evolution equations required for the proof in~\textsection\ref{shiestsec}. By the discussion between the statement and the proof of Theorem~\ref{shiestimatesthm}, we can assume that $K \geq 1$. 

We first prove the $m=1$ case. Define the function
\begin{equation} \label{hdefn}
h=(8K^2+|T|^2)|\del T|^2.
\end{equation}
Applying Young's inequality to the third term of~\eqref{evoldelT7}, we get
\begin{equation} \label{localtemp}
\frac{\pt}{\pt t}|\del T|^2 \leq \Delta |\del T|^2 - (2-C\varepsilon)|\del^2T|^2 + CK^2|\del T|^2 + C K^6.
\end{equation}
Now using~\eqref{localtemp} and~\eqref{evolnormT4}, and the fact that $|T| \leq K$, we find from~\eqref{hdefn} that
\begin{align} \nonumber
\frac{\pt h}{\pt t} & \leq (\Delta |T|^2-2|\del T|^2+CK^4)|\del T|^2  \\
& \qquad + (8K^2+|T|^2)(\Delta |\del T|^2 -(2-C\varepsilon) |\del^2T|^2+CK^2|\del T|^2+CK^6). \label{localtemp2}
\end{align}
Observe that
\begin{equation*}
\del |T|^2 = 2 \langle T, \del T \rangle \leq 2 |T| | \del T | \leq 2K | \del T |
\end{equation*}
and similarly
\begin{equation*}
\del |\del T|^2 = 2 \langle \del T, \del^2 T \rangle \leq 2 |\del T| | \del^2 T |.
\end{equation*}
Combining the above two estimates and Cauchy-Schwarz gives
\begin{equation*}
\langle \del |T|^2, \del |\del T|^2 \rangle \geq - |\del |T|^2| \, |\del |\del T|^2| \geq - 4K |\del T|^2 |\del^2 T|.
\end{equation*}
Using the above we compute directly from~\eqref{hdefn} that
\begin{align} \nonumber
\Delta h & = (\Delta |T|^2)|\del T|^2 + (8K^2+|T|^2)\Delta |\del T|^2 + 2 \langle \del |T|^2, \del |\del T|^2 \rangle \\
& \geq (\Delta |T|^2)|\del T|^2 + (8K^2+|T|^2)\Delta |\del T|^2-8K|\del T|^2|\del^2T|. \label{localtemp3}
\end{align}
From~\eqref{localtemp2} and~\eqref{localtemp3} and $|T| \leq K$, we get
\begin{align} \nonumber
\frac{\pt h}{\pt t} & \leq \Delta h - 2|\del T|^4 + CK^4|\del T|^2 + (8K^2+|T|^2)(-(2-C\varepsilon)|\del^2T|^2+CK^2|\del T|^2 + CK^6) \\ \nonumber
& \qquad +8K|\del T|^2|\del^2T| \\ \label{localtemp4}
& \leq \Delta h -2 |\del T|^4 +CK^4|\del T|^2 -(2-C\varepsilon) (8K^2+|T|^2)|\del^2T|^2 + CK^8 + 8K|\del T|^2|\del^2T|.
\end{align}
We want to use Young's inequality on both the $C K^4 |\del T|^2$ and the $8K |\del T|^2 |\del^2 T|$ terms above, so that the net amount of $|\del T|^4$ terms that remain are still strictly negative and the net amount of $|\del^2 T|^2$ terms that remain are also negative and can be discarded. This is a delicate balancing act. Explicitly, let $\delta, \gamma > 0$ and write
\begin{align*}
C K^4 |\del T|^2 & \leq \frac{C \delta}{2} |\del T|^4 + \frac{C}{2 \delta} K^8, \\
8 K |\del T|^2 |\del^2 T| & \leq \frac{4}{\gamma} K^2 |\del T|^4 + 4 \gamma |\del^2 T|^2.
\end{align*}
Then~\eqref{localtemp4} becomes
\begin{align*}
\frac{\pt h}{\pt t} & \leq \Delta h -2 |\del T|^4 + \frac{C \delta}{2} |\del T|^4 + \frac{C}{2 \delta} K^8 - (2-C\varepsilon) (8K^2+|T|^2)|\del^2T|^2 + CK^8 \\
& \qquad + \frac{4}{\gamma} K^2 |\del T|^4 + 4 \gamma |\del^2 T|^2 \\
& \leq \Delta h + \Big( -2 + \frac{C \delta}{2} + \frac{4 K^2}{\gamma} \Big) |\del T|^4 + \Big( 4 \gamma - (2 - C\varepsilon) (8 K^2 + |T|^2) \Big) |\del^2 T|^2 + \tilde C K^8.
\end{align*}
We want to ensure that
\begin{equation} \label{localinequalities}
-2 + \frac{C \delta}{2} + \frac{4 K^2}{\gamma} < -\frac{1}{2}, \qquad \text{and} \qquad 4 \gamma - (2 - C\varepsilon) (8 K^2 + |T|^2) < 0.
\end{equation}
The second inequality in~\eqref{localinequalities} is satisfied if we choose
\begin{equation} \label{gammaeq}
\gamma < (2 - C\varepsilon) 2K^2.
\end{equation}
Then, assuming $C \delta < 3$, the first inequality in~\eqref{localinequalities} and~\eqref{gammaeq} 
can be combined to yield
\begin{equation*}
\frac{8K^2}{3 - C \delta} < \gamma < (2 - C \varepsilon) 2 K^2.
\end{equation*}
It is clear that if $\delta$ and $\varepsilon$ are chosen sufficiently small then $\gamma$ will exit satisfying the above condition.

With these choices of $\varepsilon$, $\gamma$, and $\delta$, we can discard the $| \del^2 T|^2$ term (which now has a negative coefficient), and we are left with
\begin{equation} \label{localtemp5}
\frac{\pt h}{\pt t} \leq \Delta h - \frac{1}{2} |\del T|^4 + C K^8.
\end{equation}
From~\eqref{hdefn} and $|T| \leq K$, we have $h \leq 9 K^2 |\del T|^2$, so~\eqref{localtemp5} finally becomes
\begin{equation} \label{localtemp6}
\frac{\pt h}{\pt t} \leq \Delta h - \frac{h^2}{CK^4} + CK^8.
\end{equation}
Now define, for \emph{the same constant $C$ as above}, the function
\begin{equation} \label{hhatdefn}
F = \frac{h}{CK^4} - K^2.
\end{equation}

We compute using~\eqref{hhatdefn} and~\eqref{localtemp6} that
\begin{align} \nonumber
\frac{\pt F}{\pt t} &\leq \frac{1}{CK^4}(\Delta h - \frac{h^2}{CK^4}+CK^8) \\ \nonumber
&= \Delta F - (F+K^2)^2 + K^4 \\ \label{localtemp7}
& \leq \Delta F - F^2. 
\end{align}

Let $(x,t) \in P_r (x_0, t_0)$. If $F(x,t) \leq 0$, then by the definition of $F$ in~\eqref{hhatdefn} we have $| \del T |^2 \leq \tfrac{C}{8} K^4 \leq C K^4$ at such a point. If $F(x,t) \geq 0$, then since~\eqref{localtemp7} holds, by Lemma~\ref{localestlemma} with $\gamma = \tfrac{1}{2}$ we have 
\begin{equation*}
F \leq \frac{C_{\gamma}}{r^2} \leq C_{\gamma} K \qquad \text{on $P_{\tfrac{r}{2}}(x_0, t_0)$}.
\end{equation*}
Using the above, along with equation~\eqref{hhatdefn} and our assumption that $K \geq 1$, we deduce that
\begin{equation*}
h \leq CK^4 (C_{\gamma} K+K^2) \leq \tilde C K^6
\end{equation*}
and thus from~\eqref{hdefn} that
\begin{equation*}
|\del T| \leq CK^2 \qquad \text{on $P_{\tfrac{r}{2}}(x_0, t_0)$},
\end{equation*}
which establishes the base case of the induction.

Now assume inductively that~\eqref{localestthmeqn} holds for all $k<m$. We prove the theorem for $m$. Choose $B$ to be a constant such that 
\begin{equation} \label{local1}
K^m \leq B \leq CK^m \qquad \text{and} \qquad |\del^{m-1}T| \leq B
\end{equation}
for some $C > 1$. (We can take $B = C_{m-1} K^m$ if we take $C_{m-1} > 1$.) Using this $B$, define a function $h_m$ by
\begin{equation} \label{hmdefn}
h_m = (8B^2+|\del^{m-1}T|^2)|\del^mT|^2.
\end{equation}
We estimate each term in the evolution~\eqref{evolnormdelkT1} of $|\del^mT|^2$ using the induction hypothesis~\eqref{localestthmeqn} for $k<m$. For the third term on the right hand side of~\eqref{evolnormdelkT1}, we get
\begin{align}
\Big| \sum_{i=0}^m \del^mT*\del^{m-i}T*\del^i\riem \Big|& \leq |\del^mT*\del^mT*\riem|+ \Big |\sum_{i=1}^m \del^mT*\del^{m-i}T*\del^i\riem \Big| \nonumber \\
& \leq CK^2|\del^mT|^2+CK^{m+3}|\del^mT| \label{localevol1}
\end{align}
where we have used the hypothesis on $|\del^j\riem|$ and the induction hypothesis in the last inequality. Note that following the same procedure that lead to~\eqref{delpsiestm} with assumption~\eqref{localestthmeqn} instead we get
\begin{equation} \label{localdelphiest}
|\del^i\g2|\leq CK^i \qquad \text{and} \qquad |\del^i\psi|\leq CK^i.
\end{equation}
Thus for the fourth term on the right hand side of~\eqref{evolnormdelkT1} is
\begin{equation*}
\del^mT*\del^m(\del T*T*\g2) = \sum_{i=0}^m \nabla^{m+1-i} T * \nabla^{i} (T*\g2).
\end{equation*}
We decompose the sum above into four parts, corresponding to $i = 0$, $i=1$, $2 \leq i \leq m-1$, and $i=m$. Then using~\eqref{localdelphiest} we compute
\begin{align}
|\del^mT*\del^m(\del T*T*\g2)|&=| \del^mT*\del^{m+1}T*T*\g2|+ |\del^mT*\del^mT*(\del T*\g2 + T*T*\psi)| \nonumber \\
& \qquad + \Big |\del^mT*\sum_{i=2}^{m-1} \del^{m+1-i}T*(\sum_{j=0}^i \del^{i-j}T*\del^j\g2)  \Big | \nonumber \\
& \qquad  + \Big |\del^mT*\del T* \sum_{i=0}^m \del^{m-i}T*\del^i\g2 \Big | \nonumber \\
& \leq CK|\del^mT||\del^{m+1}T|+CK^2|\del^mT|^2+CK^{m+3}|\del^mT|. \label{localevol2}
\end{align}

For the fifth term on the right hand side of~\eqref{evolnormdelkT1}, using~\eqref{localdelphiest} we have
\begin{equation} \label{localevol3}
\Big | \del^mT*\del^{m+1-i}\riem*\del^i\g2  \Big | \leq CK^{m+3}|\del^mT|.
\end{equation}

Similarly, for the last term on the right hand side of~\eqref{evolnormdelkT1} we have
\begin{equation*}
\sum_{i=0}^m \del^mT*\del^{m-i}(\riem*T)*\del^i\psi = \sum_{i=0}^m \del^mT*\Big ( \sum_{j=0}^{m-i} \del^{m-i-j}\riem*\del^jT \Big) *\del^i\psi.
\end{equation*}
We split the double sum above into two parts, the first part corresponding to $i=0, j=m$ and the second part corresponding to the rest. Then using the hypothesis on $|\del^j\riem|$, the induction hypothesis, and~\eqref{localdelphiest} we have
\begin{equation}
\Big | \sum_{i=0}^m \del^mT*\del^{m-i}(\riem*T)*\del^i\psi  \Big | \leq CK^2|\del^mT|^2 + CK^{m+3}|\del^mT|. \label{localevol4}
\end{equation}

Substituting the estimates~\eqref{localevol1},~\eqref{localevol2},~\eqref{localevol3} and~\eqref{localevol4} into~\eqref{evolnormdelkT1} we get
\begin{align*}
\frac{\pt}{\pt t}|\del^mT|^2 & \leq \Delta |\del^mT|^2-2|\del^{m+1}T|^2+CK|\del^mT||\del^{m+1}T|+CK^2|\del^mT|^2+CK^{m+3}|\del^mT|.
\end{align*}
Now we use Young's inequality on the third term and the last term above to write
\begin{align*}
K|\del^mT||\del^{m+1}T| &\leq \frac{K^2|\del^mT|^2}{2\varepsilon}+\frac{\varepsilon|\del^{m+1}T|^2}{2}, \\
K^{m+3}|\del^mT| &\leq \frac{K^{2m+4}}{2}+\frac{K^2|\del^mT|^2}{2}.
\end{align*}
Substituting these into the expression for $\frac{\pt}{\pt t}|\del^mT|^2$ above gives
\begin{equation} \label{localevolfinal1}
\frac{\pt}{\pt t}|\del^mT|^2 \leq \Delta |\del^mT|^2-(2-C\varepsilon)|\del^{m+1}T|^2+CK^2|\del^mT|^2+CK^{2m+4}.
\end{equation}

The derivation of~\eqref{localevolfinal1} in fact holds for $m$ replaced by $m-1$. That is, we also have
\begin{equation*}
\frac{\pt}{\pt t}|\del^{m-1}T|^2 \leq \Delta |\del^{m-1}T|^2-(2-C\varepsilon)|\del^{m}T|^2+CK^2|\del^{m-1}T|^2+CK^{2m+2}.
\end{equation*}
Using the induction hypothesis, the above inequality becomes 
\begin{equation} \label{localevolfinal2}
\frac{\pt}{\pt t}|\del^{m-1}T|^2 \leq \Delta |\del^{m-1}T|^2-(2-C\varepsilon)|\del^mT|^2+CK^{2m+2}.
\end{equation}

From~\eqref{localevolfinal1} and~\eqref{localevolfinal2} and the definition~\eqref{hmdefn} of $h_m$, we have
\begin{align*}
\frac{\pt}{\pt t}h_m & \leq (8B^2+|\del^{m-1}T|^2)(\Delta |\del^mT|^2-(2-C\varepsilon)|\del^{m+1}T|^2+CK^2|\del^mT|^2+CK^{2m+4}) \\
& \qquad  +|\del^mT|^2(\Delta |\del^{m-1}T|^2-(2-C\varepsilon)|\del^mT|^2+CK^{2m+2}).
\end{align*}
Using~\eqref{local1} and throwing away some but not all of the negative terms, this inequality becomes
\begin{align} \nonumber
\frac{\pt}{\pt t}h_m & \leq (8B^2+|\del^{m-1}T|^2)\Delta |\del^mT|^2+|\del^mT|^2\Delta |\del^{m-1}T|^2-(2-C\varepsilon)8B^2|\del^{m+1}T|^2 \\ \label{hmtemp1}
& \qquad -(2-C\varepsilon)|\del^mT|^4 + CK^{2m+2}|\del^mT|^2+CK^{4m+4}.
\end{align}

Observe that from the inductive hypothesis~\eqref{localestthmeqn} for $k<m$ and~\eqref{local1} we have
\begin{equation*}
\del |\del^{m-1} T|^2 = 2 \langle \del^{m-1} T, \del^m T \rangle \leq 2 |\del^{m-1} T| | \del^m T | \leq 2 B | \del^m T|
\end{equation*}
and also that
\begin{equation*}
\del |\del^m T|^2 = 2 \langle \del^m T, \del^{m+1} T \rangle \leq 2 |\del^m T| | \del^{m+1} T |.
\end{equation*}
Combining the above two estimates and Cauchy-Schwarz gives
\begin{equation*}
\langle \del^{m-1} |T|^2, \del^m |\del T|^2 \rangle \geq - |\del^{m-1} |T|^2| \, |\del |\del^{m} T|^2| \geq - 4B |\del^m T|^2 |\del^{m+1} T|.
\end{equation*}
Using the above we compute directly from~\eqref{hmdefn} that
\begin{align} \nonumber
\Delta h_m & = (\Delta |\del^{m-1}T|^2)|\del^m T|^2 + (8B^2+|\del^{m-1}T|^2)\Delta |\del^m T|^2 + 2 \langle \del |\del^{m-1}T|^2, \del |\del^m T|^2 \rangle \\
& \geq |\del^mT|^2\Delta |\del^{m-1}T|^2 + (8B^2+|\del^{m-1}T|^2)\Delta |\del^mT|^2 - 8B|\del^mT|^2|\del^{m+1}T|. \label{hmtemp2}
\end{align}
From~\eqref{hmtemp1} and~\eqref{hmtemp2} we get
\begin{align*}
\frac{\pt}{\pt t}h_m & \leq \Delta h_m - (2-C\varepsilon)8B^2|\del^{m+1}T|^2-(2-C\varepsilon)|\del^mT|^4+CK^{2m+2}|\del^mT|^2+CK^{4m+4} \\
& \qquad + 8 B |\del^m T|^2 |\del^{m+1} T|.
\end{align*}
Applying Young's inequality on the final term we have
\begin{align*}
\frac{\pt}{\pt t}h_m & \leq \Delta h_m-(2-C\varepsilon)8B^2|\del^{m+1}T|^2-(2-C\varepsilon)|\del^mT|^4+CK^{2m+2}|\del^mT|^2+CK^{4m+4} \\
& \qquad + 4 B^2\delta |\del^{m+1}T|^2 + \frac{4}{\delta} |\del^mT|^4.
\end{align*} 
Just as in the base case, we now have a delicate balancing act. We want to choose $\delta$ and $\varepsilon$ above that the net amount of $|\del^m T|^4$ terms that remain are still strictly negative and the net amount of $|\del^{m+1} T|^2$ terms that remain are also negative and can be discarded. Explicitly, we demand that
\begin{equation*}
-(2-C\varepsilon) 8 B^2 + 4 B^2 \delta < 0, \qquad \text{and} \qquad -(2-C\varepsilon) + \frac{4}{\delta} < -\frac{3}{4}.
\end{equation*}
These can be rearranged to yield
\begin{equation*}
\frac{16}{5 -4 C \varepsilon} < \delta < 4 - 2C\varepsilon.
\end{equation*}
It is clear that if $\varepsilon$ is chosen sufficiently small then $\delta$ will exit satisfying the above condition.

With these choices of $\varepsilon$ and $\delta$, we are left with
\begin{equation*}
\frac{\pt h}{\pt t} \leq \Delta h_m - \frac{3}{4} |\del^mT|^4+CK^{2m+2}|\del^mT|^2+CK^{4m+4}.
\end{equation*}
Using Young's inequality on the third term, the above becomes
\begin{equation} \label{hmtemp3}
\frac{\pt h}{\pt t} \leq  \Delta h_m-\frac{1}{2}|\del^mT|^4 +CK^{4m+4}.
\end{equation}
From~\eqref{hmdefn} and $|\del^{m-1} T| \leq B \leq C K^m$ in~\eqref{local1}, we have $h_m \leq C K^{2m} |\del^m T|^2$, so~\eqref{hmtemp3} finally becomes
\begin{equation} \label{hmtemp4}
\frac{\pt}{\pt t}h_m \leq \Delta h_m-\frac{h_m^2}{CK^{4m}}+CK^{4m+4}.
\end{equation}

As in the $m=1$ case, for \emph{the same constant $C$ as above}, define the function
\begin{equation} \label{hmhatdefn}
F = \frac{h_m}{CK^{4m}} - K^2.
\end{equation}
We compute using~\eqref{hmhatdefn} and~\eqref{hmtemp4} that
\begin{align} \nonumber
\frac{\pt F}{\pt t} &\leq \frac{1}{CK^{4m}}(\Delta h_m - \frac{h_m^2}{CK^{4m}}+CK^{4m+4}) \\ \nonumber
&= \Delta F - (F+K^2)^2 + K^4 \\ \label{hmtemp5}
& \leq \Delta F - F^2. 
\end{align}

Let $(x,t) \in P_r (x_0, t_0)$. If $F(x,t) \leq 0$, then by the definition of $F$ in~\eqref{hmhatdefn} and $K^m \leq B$ in~\eqref{local1} we have $| \del^m T |^2 \leq \tfrac{B^{-2}}{8} h_m \leq CK^{-2m} K^{4m+2} = C K^{2m+2}$ at such a point. If $F(x,t) \geq 0$, then since~\eqref{hmtemp5} holds, by Lemma~\ref{localestlemma} with $\gamma = \tfrac{1}{2}$ we have 
\begin{equation*}
F \leq \frac{C_{\gamma}}{r^2} \leq C_{\gamma} K \qquad \text{on $P_{\tfrac{r}{2}}(x_0, t_0)$}.
\end{equation*}
Using the above, along with equation~\eqref{hmhatdefn} and our assumption that $K \geq 1$, we deduce that
\begin{equation*}
h_m \leq CK^{4m} (C_{\gamma} K+K^2) \leq \tilde C K^{4m+2}
\end{equation*}
and thus from~\eqref{hmdefn} and $K^m \leq B$ in~\eqref{local1} that
\begin{equation*}
|\del^m T| \leq CK^{m+1} \qquad \text{on $P_{\tfrac{r}{2^m}}(x_0, t_0)$},
\end{equation*}
which establishes the inductive step.
\end{proof}

\subsection{Characterization of the blow-up time} \label{lte}

Let $M$ be a compact $7$-manifold and let $\g2_0$ be a $\G2$-structure on $M$. Then starting with $\g2_0$, there exists a unique solution $\g2(t)$ of the isometric flow on a maximal time interval $[0, \tau)$ where \emph{maximal} means that either $\tau=\infty$ or $\tau < \infty$. The case $\tau < \infty$ means that there does not exist any $\varepsilon >0$ such that $\bar{\g2}(t)$ is a solution of the isometric flow for $t\in [0, \tau + \varepsilon)$ with $\bar{\g2}(t) = \g2(t)$ for $t\in [0,\tau)$. We call $\tau$ the \emph{singular time} for the flow.

In this section, we use the global derivative estimates~\eqref{shiestimateseqn} to prove that the quantity $\cT(t)$ defined in~\eqref{mathcalTdefn} must blow up at a finite time singularity along the flow. Explicitly, we prove the following result.
\begin{thm} \label{ltethm}
Let $M^7$ be compact and let $\g2(t)$ be a solution to the isometric flow~\eqref{divtfloweqn} in a maximal time interval $[0, \tau)$. If $\tau < \infty$, then $\cT$ satisfies
\begin{equation} \label{ltethmeqn1}
\lim_{t \nearrow \tau} \cT(t) = \infty
\end{equation}
and there is a lower bound on the blow-up rate of $\cT(t)$ given by
\begin{equation} \label{ltethmeqn2}
\cT(t)\geq \cfrac{1}{\sqrt{C(\tau -t)}}
\end{equation}
for some constant $C>0$.
\end{thm}
\begin{proof}
We prove the contrapositive of the theorem. That is, we show that if $\cT$ remains bounded along a sequence of times approaching $\tau$, then the solution can be extended past $\tau$. Let $\g2(t)$ be a solution to the isometric flow which exists on a maximal time interval $[0, \tau]$. We first prove by contradiction that
\begin{equation} \label{lteproof1}
\limsup_{t \nearrow \tau} \cT(t) = \infty.
\end{equation}
Suppose that~\eqref{lteproof1} does not hold, so there exists a constant $K>0$ such that
\begin{equation} \label{lteproof2}
\underset{M\times [0, \tau]}{\textup{sup}} \cT(t) = \underset{M\times [0, \tau]}{\textup{sup}}|T(x,t)|_g \leq K.
\end{equation}
Note that since the metric does not evolve along the flow, we use the metric $g$ induced by the initial $\G2$-structure. We have from~\eqref{shiestimateseqn} and~\eqref{lteproof2} that
\begin{equation} \label{lteproof3}
\Big |\frac{\pt}{\pt t}\g2 \Big |_g = |\Div T \hk \psi|_g \leq CKt^{-\frac{1}{2}} 
\end{equation}
for some uniform positive constant $C$. For any $0<t_1<t_2<\tau$, we have
\begin{equation} \label{lteproof4}
|\g2(t_2)-\g2(t_1)|_g \leq \int_{t_1}^{t_2} \Big|\frac{\pt}{\pt t}\g2 \Big|dt \leq CK(\sqrt{t_2}-\sqrt{t_1})
\end{equation}
which implies that $\g2(t)$ converges to a $3$-form $\g2(\tau)$ continuously as $t\rightarrow \tau$. Since $\g2(t)$ is a $\G2$-structure, we know that for all $t\in [0,\tau)$ we have
\begin{align} \label{lteproof5}
g(u,v)\vol_g = -\frac 16(u\hk \g2(t))\wedge (v\hk \g2(t)) \wedge \g2(t)
\end{align}
where $\vol_g$ is the volume form of $g$. Since $g$ and $\vol_g$ do not change along the flow, as $t\rightarrow \tau$ the left hand side of~\eqref{lteproof5} tends to a positive definite $7$-form valued bilinear form and thus the limit $3$-form is a $positive$ $3$-form and so is a $\G2$-structure. Moreover from the right hand side of~\eqref{lteproof5} we see that the limit $\g2(\tau)$ induces the same metric $g$. Thus, the solution $\g2(t)$ of the isometric flow can be extended continuously to the time interval $[0, \tau]$. We now show that the extension is actually smooth, which gives our required contradiction.

We pause to prove the following.
\begin{claim} \label{lteclaim}
For all $m \in \mathbb{N}$, there exist constants $C_m$ such that
\begin{equation*}
\sup_{M\times [0, \tau)} \Big| \del^m \g2(t) \Big|_g \leq C_m.
\end{equation*}
\end{claim}
\begin{proof}[Proof of Claim~\ref{lteclaim}]
The proof is by induction on $m$. For $m=1$, at any $(x,t)\in M\times [0, \tau)$, we have
\begin{equation*}
\frac{\pt }{\pt t}\del \g2 = \del \frac{\pt}{\pt t}\g2 = \del (\Div T\hk \psi) = \del (\Div T)*\psi + \Div T*T*\g2.
\end{equation*}
Here we are again using the fact that the metric does not evolve along the flow. We know from~\eqref{lteproof2} and Corollary~\ref{shiestcorr} that both $|\del (\Div T)|\leq A$ and $|\Div T| \leq A$ on the time interval $(\frac{1}{K^2}, \tau)$ for some $A=A(m, K)$. Since $|\del (\Div T)|$ and $|\Div T|$ are bounded on $[0, \frac{1}{K^2}]$ by some constant $B=B(K)$ we get that
\begin{equation*}
\Big|\frac{\pt}{\pt t}\del \g2 \Big| \leq \max{(CA, CB)} = \tilde C,
\end{equation*}
and thus by integration we have
\begin{equation*}
|\del \g2(t)|_g \leq |\del \g2(0)|_g + \int_{0}^{\tau} \Big| \frac{\pt}{\pt t}\del \g2(t)\Big | dt \leq |\del \g2(0)|_g +\tilde C \tau \leq C_1
\end{equation*}
because $\tau < \infty$. (This is where we crucially use that the maximal existence time is finite.) We have thus established the $m=1$ case of the claim.

For the general case of the claim, we have
\begin{equation}
\Big |\frac{\pt}{\pt t}\del^m \g2 \Big |_g = |\del^m (\Div T \hk \psi)|_g \leq C \sum_{i=0}^m |\del^{m-i}(\Div T)||\del ^i \psi|. \label{lteproof6}
\end{equation}
By the induction hypothesis, we may assume that $\Big|\frac{\pt}{\pt t} \del^p \g2 \Big|$ and hence $|\del^p(\Div T\hk \psi)|$ has been estimated for all $0\leq p <m$. Since $\del^i \psi$ contains $\del^{i-1}T$ as the highest order term, we just need to estimate the $|\del^m(\Div T)|$ term. But again it follows from~\eqref{lteproof2} and Corollary~\ref{shiestcorr} that $|\del^m(\Div T)|\leq A$ for some $A=A(m,K)$ on $(\frac{1}{K^2}, \tau)$ and $|\del^m(\Div T)|\leq B$ for some $B(m,K)$ on $[0, \frac{1}{K^2}]$. Thus from~\eqref{lteproof6} we get that
\begin{equation} \label{lteproof7}
\Big| \frac{\pt}{\pt t} \del^m \g2 \Big|_g \leq C'_m,
\end{equation}
and the inductive step now follows from~\eqref{lteproof7} by integration. This completes the proof of Claim~\ref{lteclaim}. 
\end{proof}

We now return to the proof of Theorem~\ref{ltethm}. Let $U$ be the domain of a fixed local coordinate chart. We know that $\g2(\tau)$ is a continuous limit of $\G2$-structures and in $U$ it satisfies
\begin{equation} \label{lteproof8}
\g2_{ijk}(\tau) = \g2_{ijk}(t)+ \int_{t}^{\tau} (\Div T(s)\hk \psi(s))_{ijk}ds.
\end{equation}
Let $\alpha = (a_1,...,a_r)$ be any multi-index with $|\alpha| = a_1 + \cdots + a_r =m\in \mathbb{N}$. We know from Claim~\ref{lteclaim} and~\eqref{lteproof7} that
\begin{equation*}
\frac{\pt^m}{\pt x^{\alpha}}\g2_{ijk} \qquad \text{and} \qquad \frac{\pt^m}{\pt x^{\alpha}}(\Div T\hk \psi)_{ijk}
\end{equation*}
are uniformly bounded on $U\times [0, \tau)$. So from~\eqref{lteproof8} we have that $\frac{\pt^m}{\pt x^{\alpha}}\g2_{ijk}(\tau)$ is bounded on $U$ and hence $\g2(\tau)$ is a smooth $\G2$-structure. Moreover, from~\eqref{lteproof8} we have
\begin{equation*}
\Big | \frac{\pt^m}{\pt x^{\alpha}}\g2_{ijk}(\tau)-\frac{\pt^m}{\pt x^{\alpha}}\g2_{ijk}(t) \Big | \leq C(\tau-t)
\end{equation*}
and thus $\g2(t)\rightarrow \g2(\tau)$ uniformly in any $C^m$ norm as $t\rightarrow \tau$, for $m\geq 2$.

Now, since $\g2(\tau)$ is smooth, Theorem~\ref{stethm} gives a solution $\bar{\g2}(t)$ of the isometric flow with $\bar{\g2}(0) = \g2(\tau)$ for a short time $0\leq t < \varepsilon$. Since $\g2(t)\rightarrow \g2(\tau)$ smoothly as $t\rightarrow \tau$, it follows that
\begin{equation*}
\bar{\g2}(t) = 
        \begin{cases}
             \g2(t) & 0\leq t < \tau \\
            \bar{\g2}(t-\tau) & \tau \leq t < \tau + \varepsilon
        \end{cases}
\end{equation*}
is a solution of the isometric flow which is smooth and satisfies $\bar{\g2}(0)=\g2(0)$. This contradicts the maximality of $\tau$. Thus we indeed have
\begin{equation} \label{lteproof9}
\limsup_{t\nearrow \tau}\cT(t)=\infty,
\end{equation}
which is equation~\eqref{lteproof1}. Thus, if $\lim_{t \nearrow \tau} \cT(t)$ exists, it must be $\infty$.

Next we show that in fact~\eqref{ltethmeqn1} is true. Suppose not. Then there exists $K_0< \infty$ and a sequence of times $t_i \nearrow \tau$ such that $\cT(t_i) \leq K_0$. By the doubling time estimate in Proposition~\ref{dtestprop}, we get that
\begin{equation*}
\cT(t)\leq 2\cT(t_i)\leq 2K_0
\end{equation*}
for all times $t\in [t_i, \min\{ \tau, t_i+\frac{1}{C K_0^2} \}]$. Since $t_i\nearrow \tau$ as $i\rightarrow \infty$, there exists $i_0$ large enough such that $t_{i_0}+\frac{C}{K_0^2} \geq \tau$. (Here again we crucially use the fact that $\tau$ is assumed to be finite.) But this implies that
\begin{equation*}
\sup_{M\times [t_{i_0}, \tau]} \cT(x,t) \leq 2K_0
\end{equation*} 
which cannot happen as we have already shown above that this leads to a contradiction to the maximality of $\tau$. This completes the proof of~\eqref{ltethmeqn1}.

To obtain the lower bound of the blow-up rate~\eqref{ltethmeqn2}, we apply the maximum principle to~\eqref{evolnormT4}. We get
\begin{equation*}
\frac{d}{dt} \cT(t)^2 \leq C\cT(t)^4
\end{equation*}
which implies that
\begin{equation} \label{lteproof10}
\frac{d}{dt} \cT(t)^{-2} \geq -C.
\end{equation}
Since we proved above that $\lim_{t\rightarrow \tau}\cT(t)=\infty$, we have
\begin{equation*}
\lim_{t\rightarrow \tau} \cT(t)^{-2} =0.
\end{equation*}
Integrating~\eqref{lteproof10} from $t$ to $t_0\in (t, \tau)$ and taking the limit as $t_0\rightarrow \tau$, we get 
\begin{equation*}
\cT(t)\geq \cfrac{1}{\sqrt{C(\tau -t)}}.
\end{equation*}
This completes the proof of Theorem~\ref{ltethm}.
\end{proof}

Combining Proposition~\ref{dtestprop} and Theorem~\ref{ltethm}, we deduce the following result about the minimal existence time.
\begin{corr}\label{ltecorr}
Let $\g2_0$ be a $\G2$-structure on a compact $7$-manifold $M$ with
\begin{equation*}
\cT \leq K
\end{equation*}
for some constant $K$. Then the unique solution of the isometric flow with initial $\G2$-structure $\g2_0$ exists at least for time $t\in [0, \frac{1}{CK^2}]$ where $C$ is the uniform constant from Proposition~\ref{dtestprop}.
\end{corr}

\subsection{Compactness} \label{compactness}

In this section, we prove a Cheeger--Gromov type compactness theorem for solutions to the isometric flow for $\G2$-structures. We also give a local version of the compactness theorem. Recall the following definition from~\cite{lotay-wei1}.

\begin{defn} \label{defn_converge}
Let $(M_i, \g2_i, p_i)$ be a sequence of $7$-manifolds with $\G2$-structures $\g2_i$ and $p_i\in M_i$ for each $i$. Suppose the metric $g_i$ on $M_i$ associated to the $\G2$-structure $\g2_i$ is complete for each $i$. Let $M$ be a $7$-manifold with $p\in M$ and $\g2$ be a $\G2$-structure on $M$. We say that the sequence $(M_i, \g2_i, p_i)$ converges to $(M, \g2, p)$ in the Cheeger--Gromov sense and write
\begin{equation*}
(M_i, \g2_i, p_i) \rightarrow (M, \g2, p) \qquad \text{as } i\rightarrow \infty
\end{equation*} 
if there exists a sequence of compact subsets $\Omega_i\subset M$ exhausting $M$ with $p\in $int$(\Omega_i)$ for each $i$, a sequence of diffeomorphisms $F_i:\Omega_i\rightarrow F_i(\Omega_i)\subset M_i$ with $F_i(p)=p_i$ such that
\begin{equation*}
F_i^*\g2_i \rightarrow \g2 \qquad \text{ as } i\rightarrow \infty
\end{equation*} 
in the sense that $F_i^*\g2_i-\g2$ and its covariant derivatives of all orders (with respect to any fixed metric) converge uniformly to zero on every compact subset of $M$.
\end{defn}

Lotay--Wei proved the following very general compactness theorem for $\G2$-structures in~\cite[Theorem 7.1]{lotay-wei1}.

\begin{thm} \label{strccompactnessthm}
Let $M_i$ be a sequence of smooth $7$-manifolds and for each $i$ we let $p_i\in M_i$ and $\g2_i$ be a $\G2$-structure on $M_i$ such that the metric $g_i$ on $M_i$ induced by $\g2_{i}$ is complete on $M_i$. Suppose that 
\begin{equation} \label{strccompactnessthmeqn}
\sup_i \sup_{x\in M_i} \big (|\del_{g_i}^{k+1}T_{i}(x)|^2_{g_i}+|\del^k_{g_i} \riem_{g_i}(x)|^2_{g_i} \big )^{\frac 12} < \infty
\end{equation}
for all $k\geq 0$ and
\begin{equation*}
\inf_i \inj (M_i,g_i,p_i)>0,
\end{equation*}
where $T_i$, $\riem_{g_i}$ are the torsion and the Riemann curvature tensor of $\g2_i$ and $g_i$ respectively and $\inj (M_i, g_i, p_i)$ denotes the injectivity radius of $(M_i,g_i)$ at $p_i$.

Then there exists a $7$-manifold $M$, a $\G2$-structure $\g2$ on $M$ and a point $p\in M$ such that, after passing to a subsequence, we have 
\begin{equation*}
(M_i, \g2_i, p_i) \rightarrow (M, \g2, p) \qquad \text{as } i\rightarrow \infty.
\end{equation*} 
\end{thm}

The idea of the proof is to use Cheeger--Gromov compactness theorem~\cite[Theorem 2.3]{hamilton-compactness} for complete pointed Riemannian manifolds to get a complete Riemannian $7$-manifold $(M,g)$ and $p\in M$ such that, after passing to a subsequence 
\begin{equation*}
(M_i, g_i, p_i) \rightarrow (M, g, p) \qquad \text{as } i\rightarrow \infty.
\end{equation*}
That is, there exist nested compact sets $\Omega_i\subset M$ exhausting $M$ with $p\in \text{int}(\Omega_i)$ for all $i$ and diffeomorphisms $F_i:\Omega_i \rightarrow F_i(\Omega_i) \subset M_i$ with $F_i(p)=p_i$ such that $F_i^*g\rightarrow g$ smoothly as $i\rightarrow \infty$ on any compact subset of $M$. We then use the diffeomorphisms from the above convergence to pull-back the $\G2$-structure to get $\G2$-structures $\g2_i$ on $\Omega_i$ and using~\eqref{strccompactnessthmeqn} we show that covariant derivatives of all orders of $\g2_i$ are uniformly bounded. The Arzel\'a--Ascoli theorem~\cite[Corollary 9.14]{andrews-hopper} then implies that there is a $3$-form $\g2$ such that after passing to a subsequence, $\g2_i\rightarrow \g2$ as $i\rightarrow \infty$. We then show that $\g2$ is a $\G2$-structure and it induces the metric $g$ and hence we get that $(M_i, \g2_i, p_i) \rightarrow (M, \g2, p)$ as $\ i\rightarrow \infty$.

We note that if all the metrics in the sequence $(M_i, \g2_i, g_i)$ are the same then the limiting $\G2$-structure $\g2$ induces the same metric.

We now state and prove the compactness theorem for the isometric flow of $\G2$-structures.
\begin{thm} \label{compactnessthm}
Let $M_i$ be a sequence of compact $7$-manifolds and let $p_i\in M_i$ for each $i$. Let $\g2_i(t)$ be a sequence of solutions to the isometric flow~\eqref{divtfloweqn} for $\G2$-structures on $M_i$ for $t\in (a,b)$, where $-\infty \leq a<0<b\leq \infty$. Suppose that
\begin{equation} \label{cmpctthm1}
\sup_i \sup_{x\in M_i, t\in (a,b)} |T_i(x,t)|_{g_i} < \infty
\end{equation}
where $T_i$ denotes the torsion of $\g2_i(t)$, and the injectivity radius satisfies
\begin{equation} \label{cmpctthm2}
\inf_i \inj(M_i, g_i(0), p_i)>0.
\end{equation}
Suppose further that there are uniform constants $C_k$, for all $k\geq 0$, such that 
\begin{equation}
\sup_i |\del^k\riem_i|_{g_i}\leq C_k.
\end{equation}

Then there exists a $7$-manifold $M$, a point $p\in M$ and a solution $\g2(t)$ of the flow~\eqref{divtfloweqn} on $M$ for $t\in (a,b)$ such that, after passing to a subsequence, 
\begin{equation*}
(M_i, \g2_i(t), p_i) \rightarrow (M, \g2(t), p) \qquad \text{as } i\rightarrow \infty.
\end{equation*}
\end{thm}
The proof is similar in spirit to the compactness theorem for the Ricci flow by Hamilton~\cite{hamilton-compactness}. See also the compactness theorem for the Laplacian flow for closed $\G2$-structures by Lotay--Wei~\cite{lotay-wei1}. The idea is to show that the bounds on the $\G2$-structure and on covariant derivatives and time derivatives of the $\G2$-structure at time $t=0$ extend to bounds on the $\G2$-structures and covariant derivatives of the $\G2$-structures at subsequent times in the presence of bounds on the torsion and covariant derivatives of the torsion for all time.
\begin{proof}[Proof of Theorem~\ref{compactnessthm}]
From the derivative estimates~\eqref{shiestimateseqn}, Corollary~\ref{shiestcorr} and~\eqref{cmpctthm1}, we have
\begin{equation} \label{cmpctthm3}
|\del^m_{g_i(t)}T_i(x,t)|\leq C_m.
\end{equation}
Since $M_i$ is compact for each $i$, we know $|\riem_i|_{g_i}$ is bounded. Assumption~\eqref{cmpctthm2} allows us to use Theorem~\ref{strccompactnessthm} for $t=0$ to extract a subsequence of $(M_i, \g2_i(0), p_i)$ which converges to a complete limit $(M, \g2_{\infty}(0), p)$. So there exist compact subsets $\Omega_i\subset M$ exhausting $M$ with $p\in \text{int}(\Omega_i)$ for each $i$ and diffeomorphisms $F_i:\Omega_i \rightarrow F_i(\Omega_i)\subset M_i$ with $F_i(p)=p_i$ such that $F_i^*g_i(0)\rightarrow g_{\infty}(0)$ and $F_i^*\g2_i(0)\rightarrow \g2_{\infty}(0)$ smoothly on any compact subset $\Omega\subset M$ as $i\rightarrow \infty$. Fix a compact subset $\Omega \times [c,d]\subset M\times (a,b)$ and let $i$ be sufficiently large so that $\Omega \subset \Omega_i$. Let $\bar{g}_i(t)=F_i^*g_i(t)$. Now since $\g2_{i}(t)$ are all solutions to the isometric flow, we have $g_i(t)=g_i(0)$ for each $i$. Thus we trivially have
\begin{equation*}
\sup_{\Omega\times [c,d]} |\del^m_{\bar{g}_i(0)}\bar{g}_i(t)|_{\bar{g}_i(0)} = 0.
\end{equation*} 
Since the limit metric $g_{\infty}(0)$ is uniformly equivalent to $g_i(0)$, we get
\begin{equation*}
\sup_{\Omega\times [c,d]} |\del^m_{\bar{g}_i(\infty)}\bar{g}_i(t)|_{\bar{g}_i(\infty)} \leq C_m
\end{equation*}
for some positive constants $C_m$ and similarly
\begin{equation*}
\sup_{\Omega \times [c,d]} \Big |\frac{\pt^l}{\pt t^{l}} \del^m_{\bar{g}_{\infty}(0)}\bar{g}_i(t) \Big |_{\bar{g}_{\infty}(0)} \leq C_{m,l}
\end{equation*} 
for some positive constants $C_{m,l}$.

Now let $\bar{\g2}_i(t)=F_i^*\g2_i(t)$. Then $\bar{\g2}_i(t)$ is a sequence of solutions of the isometric flow on $\Omega \subset M$ for $t\in [c,d]$. Using~\eqref{cmpctthm3} and proceeding in a similar way as in the proof of Claim~\ref{lteclaim}, we deduce that there exist constants $C_m$, independent of $i$, such that
\begin{equation}
\sup_{\Omega\times [c,d]} |\del^m_{\bar{g}_i(0)}\bar{\g2}_i(t)|_{\bar{g}_i(0)} \leq C_m
\end{equation}
and since $\bar{g}_i(0)$ and $\bar{\g2}(0)$ converge uniformly to $g_{\infty}(0)$ and $\bar{\g2}_{\infty}(0)$ with all their covariant derivatives on $\Omega$, we have
\begin{equation}
\sup_{\Omega\times [c,d]} |\del^m_{\bar{g}_{\infty}(0)}\bar{\g2}_i(t)|_{\bar{g}_{\infty}(0)} \leq C_m.
\end{equation}
Moreover, because the time derivatives can be written in terms of the spatial derivatives using the evolution equations of the isometric flow, we get for some uniform constants $C_{m,l}$ that
\begin{equation}
\sup_{\Omega \times [c,d]} \Big |\frac{\pt^l}{\pt t^{l}} \del^m_{\bar{g}_{\infty}(0)}\bar{\g2}_i(t) \Big |_{\bar{g}_{\infty}(0)} \leq C_{m,l}.
\end{equation} 
It now follows from the Arzel\'a--Ascoli theorem that there exists a subsequence of $\bar{\g2}_i(t)$ that converges smoothly on $\Omega \times [c,d]$. A diagonal subsequence argument then produces a subsequence that converges smoothly on any compact subset of $M\times (a,b)$ to a solution $\bar{\g2}_{\infty}(t)$ of the isometric flow.
\end{proof}

The compactness theorem for the Ricci flow has natural applications in the analysis of singularities of the Ricci flow. We would also like to have a similar application for the isometric flow. The idea is to consider shorter and shorter time intervals leading up to a singularity of the isometric flow and to rescale the solutions on each of these time intervals to obtain solutions with uniformly bounded torsion. By doing this we hope that the limiting manifold will tell us something about the nature of the singularity and more information, such as whether the singularity is modelled on a soliton.

More precisely, suppose $M^7$ is a compact manifold and let $\g2(t)$ be a solution to the isometric flow on a maximal time interval $[0, \tau)$ with $\tau<\infty$. Theorem~\ref{ltethm} then implies that $\cT(t)$ defined in~\eqref{mathcalTdefn} satisfies $\lim_{t\nearrow \tau} \cT(t)=\infty$. Consider a sequence of points $(x_i, t_i)$ with $t_i\nearrow \tau$ and
\begin{equation*}
\cT(x_i,t_i) = \sup_{x\in M,\ t\in [0,t_i]} |T(x,t)|_g.
\end{equation*}

Define a sequence of parabolic dilations of the isometric flow 
\begin{equation}
\g2_i(t) = \cT(x_i, t_i)^3\g2(t_i+\cT(x_i,t_i)^{-2}t)
\end{equation}
and define
\begin{equation}
\cT_{\g2_i}(x,t)=|T_i(x,t)|_{g_i}.
\end{equation}
If $\widetilde{\g2} = c^3\g2$ then, as explained in the proof of Lemma~\ref{lem:rescaling}, we have
\begin{equation*}
\widetilde{\Div}\widetilde{T}\hk \widetilde{\psi} = c^3 \Div T\hk \psi.
\end{equation*}
Hence, for each $i$, we have that $(M, \g2_i(t))$ is a solution of the isometric flow~\eqref{divtfloweqn} on the time interval $[-t_i\cT(x_i,t_i)^2, (\tau-t_i)\cT(x_i,t_i)^2]$. Note that for each $i$ and for all $t\leq 0$ we have
\begin{equation*}
\sup_{M} |\cT_{\g2_i}(x,t)| = \frac{|T_i(x,t)|_{g_i}}{\cT(x_i,t_i)} \leq 1
\end{equation*}
by the definition of $\cT(x_i,t_i)$. Thus by the doubling time estimate Proposition~\ref{dtestprop} and Corollary~\ref{ltecorr}, there exists a uniform $b>0$ such that
\begin{equation*}
\sup_{i} \sup_{M\times (a,b)} |\cT_{\g2_i}(x,t)| \leq 2
\end{equation*}
for any $a<0$. Thus, if we have $\inf_{i} \inj (M, g_i(0), x_i)>0$, then using the compactness Theorem~\ref{compactnessthm}, we can extract a subsequence of $(M, \g2_i(t), x_i)$ that converges to a solution $(M_{\infty}, \g2_{\infty}(t), x_{\infty})$ of the isometric flow. 

Just as in the Ricci flow (see~\cite[\textsection 3.1]{chow-etal1}), from the proof of the compactness theorem for the isometric flow, we can prove a local version of Theorem~\ref{compactnessthm} without much difficulty.

\begin{thm}[Local compactness] \label{localcompactnessthm}
Let $\{(M_i, \g2_i(t), x_i) \}_{i\in \mathbb{N}}$, $x_i\in M_i$ and $t\in (a,b)$ be a sequence of compact pointed solutions of the isometric flow. If there exist $\rho >0$, $C_0 < \infty$ independent of $i$ such that 
\begin{equation*}
|T_i|_{g_i} \leq C_0 \qquad \text{in } B_{g_i}(x_i, \rho) \times (a,b)
\end{equation*} 
and
\begin{equation*}
\inj_{g_i}(x_i) > 0,
\end{equation*}
and if there exist uniform constants $C_k$, for all $k\geq 0$, such that
\begin{equation*}
|\del^k\riem_i|_{g_i} \leq C_k \qquad \text{in } B_{g_i}(x_i, \rho)\times (a,b),
\end{equation*}
then there exists a subsequence such that $\{(B_{g_i}(x_i, \rho), \g2_i(t), x_i) \}_{i\in \mathbb{N}}$ converges as $i\rightarrow \infty$ to a pointed solution $(B_{\infty}, \g2_{\infty}(t), x_{\infty}),\ t\in (a,b)$ of the isometric flow, smoothly on any compact subset of $B_{\infty}\times (a,b)$. Furthermore, $B_{\infty}$ is an open manifold and the metric $g_{\infty}(t)$ of $\g2_{\infty}(t)$ is complete on the closed ball $\overline{B_{g_{\infty}}(x_{\infty}, r)}$ for all $r< \rho$.
\end{thm}

\section{A Reaction-Diffusion Equation for the Torsion} \label{rde}

Recall from Lemma~\ref{evolution_of_torsion} that the evolution equation for the torsion under the isometric flow is
\begin{equation*}
\frac{\pt T_{pq}}{\pt t} = \Delta T_{pq} - \del_iT_{pb}T_{ia}\g2_{abq}+F(\g2,T,\riem,\nabla \riem), 
\end{equation*}
where $F(\g2,T,\riem,\nabla \riem)$ is given by~\eqref{evtoreq}. This evolution equation fails to be of the reaction-diffusion type due to the presence of the first order term $\del_iT_{pb}T_{ia}\g2_{abq}$.

On the other hand, reaction-diffusion equations are important because one can apply Hamilton's maximum principle for systems~\cite{hamilton-4manifolds} to relate the behaviour of a system of PDEs to that of a system of ODEs. For the Ricci flow this point of view has been remarkably successful and has led to the discovery of many preserved conditions, which has been crucial in the study of the flow.

This section is devoted to the study of the curious term $\del_iT_{pb}T_{ia}\g2_{abq}$. We discover that part of this term can be absorbed into the diffusion part of the equation, leaving out some reaction terms. In order to do this, however, we need to express the equation with respect to a different connection and also apply an Uhlenbeck-type trick, where we evolve the gauge along the flow in a particular manner.

\subsection{A modified connection} \label{sec:modifiedconnection}

Let $(M, \g2)$ be a manifold with $\G2$-structure. We write $\times$ for the vector cross product induced by $\g2$ on vector fields (equivalently $1$-forms) defined locally by $(X \times Y)_k = X_i Y_j \g2_{ijk}$.

Let $\iota:E\rightarrow TM$ be a vector bundle isomorphism and let $h=\iota^* g$. This is a fibre metric on $E$. In what follows $\{ e_i : 1 \leq i \leq 7 \}$ denotes a local $g$-orthonormal frame for $TM$ and $\{v_a : 1 \leq a \leq 7 \}$ denotes a local $h$-orthonormal frame for $E$.

Given any fixed constant $\alpha$, define a connection $D$ on the vector bundle $E$ by
\begin{equation} \label{modified_connection}
\begin{split}
D_X \sigma &= \iota^{-1}\left(\nabla_X  \iota(\sigma) +\alpha X_k T_{km}(\iota(\sigma))_l \g2_{mlp} e_p \right), \\
&=\iota^{-1}\left(\nabla_X  \iota(\sigma)  +\alpha (X\hk T) \times \iota(\sigma) \right),
\end{split}
\end{equation}
for any smooth section $\sigma$ on $E$ and smooth vector field $X$ on $M$.

\begin{lemma} \label{h_comp}
For any choice of $\alpha$, the connection $D$ on $E$ defined in~\eqref{modified_connection} is compatible with the fibre metric $h$. 
\end{lemma}
\begin{proof}
Given any point $p\in M$ consider two local sections of $E$ near $p$ of the form $\sigma_i = \iota^{-1}(Y_i)$ for $i=1,2$, such that $\nabla_X Y_i=0$ at $p$. Thus, at $p$ we have 
\begin{align*}
\iota(D_X\sigma_i) = \alpha  X_k T_{km} (Y_i)_l \g2_{mlp}e_p
\end{align*}
for each $i$. This gives

\begin{align*}
(D_X h)(\sigma_1,\sigma_2)&=X(h(\sigma_1,\sigma_2)) - h(D_X \sigma_1,\sigma_2) - h(\sigma_1, D_X \sigma_2)\\
&=X(g(Y_1,Y_2)) -g(\iota(D_X \sigma_1), \iota(\sigma_2)) - g(\iota(\sigma_1),\iota(D_X \sigma_2))\\
&=-\alpha  X_k T_{km} (Y_1)_l \g2_{mlp} (Y_2)_p -\alpha  X_k T_{km} (Y_2)_l \g2_{mlp} (Y_1)_p\\
&=0
\end{align*}
by the skew-symmetry of $\g2$.
\end{proof}

\begin{rmk}
Let $\widetilde{\g2} =(\iota \otimes \iota \otimes \iota)^* \g2 \in \Gamma(\Lambda^3 E^*)$, and let $X_i =\iota(\sigma_i)$ for $i=1,2,3$. Suppose further that at a point $p$, we have $\nabla X_i=0$. A direct computation then gives that, at the point $p$, we have
\begin{align*}
(D_b \widetilde{\g2})(\sigma_1,\sigma_2,\sigma_3)&=(\nabla_b \g2)(X_1,X_2,X_3)-\widetilde{\g2}(D_b \sigma_1, \sigma_2,\sigma_3)-\widetilde{\g2}(\sigma_1, D_b \sigma_2, \sigma_3)-\widetilde{\g2}(\sigma_1,\sigma_2,D_b \sigma_3) \\
&=T_{bm}\psi_{mijk} (X_1)_i (X_2)_j (X_3)_k -\alpha \g2_{pij} T_{bm} (X_1)_l \g2_{mlp} (X_2)_i (X_3)_j\\
&\qquad -\alpha \g2_{ipj} T_{bm} (X_2)_l \g2_{mlp} (X_1)_i (X_3)_j -\alpha \g2_{ijp} T_{bm} (X_3)_l \g2_{mlp} (X_1)_i (X_2)_j.
\end{align*}
Using the contraction identity in~\eqref{contractphpheq} the above becomes
\begin{align*}
(D_b \widetilde{\g2})(\sigma_1,\sigma_2,\sigma_3)&=T_{bm}\psi_{mijk} (X_1)_i (X_2)_j (X_3)_k\\
&\qquad -\alpha T_{bm} (X_1)_l (X_2)_i (X_3)_j (g_{im} g_{jl} - g_{il} g_{jm} - \psi_{ijml})\\
&\qquad -\alpha T_{bm} (X_1)_i (X_2)_l (X_3)_j ( g_{jm} g_{il} - g_{jl} g_{im} - \psi_{jiml} )\\
&\qquad -\alpha T_{bm} (X_1)_i (X_2)_j (X_3)_l ( g_{im} g_{jl} - g_{il} g_{jm} - \psi_{ijml}),
\end{align*}
and after relabelling indices we finally get
\begin{align*}
(D_b \widetilde{\g2})(\sigma_1,\sigma_2,\sigma_3)&=T_{bm}\psi_{mijk} (X_1)_i (X_2)_j (X_3)_k\\
&\qquad -\alpha T_{bm} (X_1)_i (X_2)_j (X_3)_k (g_{jm} g_{ki} - g_{ji} g_{km} - \psi_{jkmi})\\
&\qquad -\alpha T_{bm} (X_1)_i (X_2)_j (X_3)_k ( g_{km} g_{ij} - g_{kj} g_{im} - \psi_{kimj} )\\
&\qquad -\alpha T_{bm} (X_1)_i (X_2)_j (X_3)_k ( g_{im} g_{jk} - g_{ik} g_{jm} - \psi_{ijmk})\\
&=T_{bm}\psi_{mijk} (X_1)_i (X_2)_j (X_3)_k\\
&\qquad +\alpha T_{bm} (X_1)_i (X_2)_j (X_3)_k(\psi_{jkmi}+\psi_{kimj}+\psi_{ijmk})\\
&=(1+3\alpha) T_{bm} (X_1)_i (X_2)_j (X_3)_k \psi_{mijk}.
\end{align*}
Thus, we deduce that if we choose $\alpha=-\frac{1}{3}$ then $\widetilde{\g2}$ is $D$-parallel. But it turns out that this choice is not the correct choice for our purposes.
\end{rmk}

Given any $A\in\Gamma(T^*M\otimes T^*M)$, we define $\widetilde A\in \Gamma(T^*M\otimes E^*)$ by $\widetilde A=(\id \otimes \iota)^* A$, where $\id$ denotes the identity map. That is,
\begin{equation*}
\widetilde A(X,\sigma)=A(X,\iota(\sigma)).
\end{equation*}
By coupling the (dual of the) Levi-Civita connection $\nabla$ on $T^* M$ with the (dual of the) connection $D$ on $E^*$, we get an induced connection on $T^* M \otimes E^*$, which we also denote by $D$.

\begin{lemma} \label{DA}
Given any $A\in\Gamma(T^*M\otimes T^*M)$, the $D$-covariant derivative of $\widetilde A$ is given by
\begin{equation*}
(D_X \widetilde A)(Y,\sigma)= (\nabla_X A)(Y,\iota(\sigma)) - \alpha Y_i A_{ip} X_k T_{km} \iota(\sigma)_l \g2_{mlp}.
\end{equation*}
\end{lemma}
\begin{proof}
As in the proof of Lemma~\ref{h_comp}, at any point $p$ we can choose local vector fields $Y,W$ near $p$ satisfying $\nabla_X Y = \nabla_X W = 0$ at $p$ and $\iota(\sigma)=W$. Then at $p$, we have
\begin{equation*}
\iota(D_X \sigma)= \alpha X_k T_{km} W_l \g2_{mlp}e_p,
\end{equation*}
and thus
\begin{align*}
(D_X \widetilde A)(Y, \sigma)&= X ( \widetilde A(Y,\sigma))-\widetilde A(\nabla_X Y,\sigma)-\widetilde A(Y, D_X \sigma)\\
&=X ( A(Y, \iota(\sigma)))-A(\nabla_X Y,\iota(\sigma))- A(Y, \iota(D_X \sigma))\\
&=(\nabla_X A) (Y, \iota(\sigma))  -\alpha Y_i A_{ip} X_k T_{km}\iota(\sigma)_l  \g2_{mlp}.
\end{align*}
\end{proof}

\begin{lemma}
Let $A\in \Gamma(T^*M\otimes T^*M)$ and consider the associated tensor $\widetilde A\in \Gamma(T^*M\otimes E^*)$ defined by $\widetilde A= (\id\otimes \iota)^*A$. Let $\Delta_D$ denote the associated Laplacian on $(T^*M,\nabla) \otimes (E^*,D)$. Then we have
\begin{equation*}
\begin{aligned}
\Delta_{D} \widetilde A_{ia} &=\Delta A (e_i,\iota(v_a)) - \alpha^2 \left[|T|^2 A_{iq} -(A\circ T^t \circ T)_{iq}\right] \iota(v_a)_q \\
&\qquad-2\alpha\nabla_k A_{ip} T_{km} \iota(v_a)_l \g2_{mlp} - \alpha A_{ip} \nabla_k T_{km} \iota(v_a)_l \g2_{mlp}.
\end{aligned}
\end{equation*}
In particular, the torsion $T$ of the $\G2$-structure satisfies
\begin{equation} \label{LaplaceDT}
\begin{aligned}
\Delta_{D} \widetilde T_{ia} &=\Delta T (e_i,\iota(v_a)) - \alpha^2 \left[|T|^2 T_{iq} -(T\circ T^t \circ T)_{iq}\right] \iota(v_a)_q \\
&\qquad - 2\alpha\nabla_k T_{ip} T_{km} \iota(v_a)_l \g2_{mlp} - \alpha T_{ip} \nabla_k T_{km} \iota(v_a)_l \g2_{mlp}.
\end{aligned}
\end{equation}
\end{lemma}
\begin{proof}
Using Lemma~\ref{DA} and taking the local frames $\{ e_i \}$ and  $\{v_a\}$ to be $\nabla$-parallel and $D$-parallel at a point $p$, respectively, we obtain
\begin{align*}
D_k D_k \widetilde A_{ia}&= e_k((\nabla_k A)(e_i,\iota(v_a))-\alpha A_{ip}T_{km}\iota(v_a)_l  \g2_{mlp})\\
&=(\nabla_k \nabla_k A)(e_i,\iota(v_a)) +(\nabla_k A)(e_i, \nabla_k \iota(v_a)) - \alpha \nabla_k A_{ip} T_{km} \iota(v_a)_l \g2_{mlp} \\
& \qquad - \alpha A_{ip} \nabla_k T_{km} \iota(v_a)_l \g2_{mlp} -\alpha A_{ip} T_{km} \nabla_k \iota(v_a)_l \g2_{mlp} -\alpha A_{ip} \iota(v_a)_l T_{km} \nabla_k \g2_{mlp}.
\end{align*}
Applying the definition~\eqref{modified_connection} of $D$ and~\eqref{delpheq} this becomes
\begin{align*}
\Delta_{D} \widetilde A_{ia}&=\Delta A (e_i,\iota(v_a)) +\nabla_k A_{ip}\left( [\iota( D_k v_a)]_p -\alpha T_{km} \iota(v_a)_l \g2_{mlp} \right)\\
&\qquad - \alpha \nabla_k A_{ip} T_{km} \iota(v_a)_l \g2_{mlp}- \alpha A_{ip} \nabla_k T_{km} \iota(v_a)_l \g2_{mlp} \\
&\qquad -\alpha A_{ip} T_{km} \left( [\iota(D_k v_a)]_l -\alpha T_{kj} \iota(v_a)_q \g2_{jql} \right)  \g2_{mlp} -\alpha A_{ip} \iota(v_a)_l T_{km} \nabla_k \g2_{mlp}\\
&=\Delta A (e_i,\iota(v_a)) -2\alpha\nabla_k A_{ip} T_{km} \iota(v_a)_l \g2_{mlp} \\
&\qquad - \alpha A_{ip} \nabla_k T_{km} \iota(v_a)_l \g2_{mlp} + \alpha^2 A_{ip} T_{km}  T_{kj} \iota(v_a)_q \g2_{jql}   \g2_{mlp} \\
&\qquad-\alpha A_{ip} \iota(v_a)_l T_{km} T_{kq} \psi_{qmlp}.
\end{align*}
Since $T_{km} T_{kq}$ is symmetric in $k,q$ and $\psi_{qmlp}$ is skew in $k,q$, the last term above vanishes, and we get
\begin{align*}
\Delta_{D} \widetilde A_{ia}&=\Delta A (e_i,\iota(v_a)) -2\alpha\nabla_k A_{ip} T_{km} \iota(v_a)_l \g2_{mlp} \\
&\qquad - \alpha A_{ip} \nabla_k T_{km} \iota(v_a)_l \g2_{mlp} + \alpha^2 A_{ip} T_{km}  T_{kj} \iota(v_a)_q \g2_{jql}   \g2_{mlp}.
\end{align*}
Finally, using~\eqref{contractphpheq} the term $A_{ip} T_{km}  T_{kj}  \g2_{jql}   \g2_{mlp}$ becomes
\begin{align*}
A_{ip} T_{km}  T_{kj}  \g2_{jql}   \g2_{mlp} &= A_{ip} T_{km}  T_{kj}  \g2_{jql}   \g2_{pml}\\
&=A_{ip} T_{km}  T_{kj}( g_{jp} g_{qm} - g_{jm}g_{qp} -\psi_{jqpm})\\
&= (A\circ T^t \circ T)_{iq} - |T|^2 A_{iq},
\end{align*}
and the proof is complete.
\end{proof}

\subsection{An Uhlenbeck-type trick} \label{sec:uhlenbeck}

Suppose that a family $\g2(t)\in \Omega^3(M)$ of $\G2$-structures evolves by
\begin{equation*}
\frac{\partial}{\partial t} \g2_{ijk}= X_l \psi_{lijk}
\end{equation*}
for some vector field $X$ and that a family $\iota_t:E\rightarrow TM$ of vector bundle isomorphisms evolves by
\begin{equation*}
\frac{\pt}{\pt t}\iota(\sigma)_k= \beta X_i \iota(\sigma)_j \g2_{ijk} = \beta X \times \iota(\sigma),
\end{equation*}
for some constant $\beta$.

Let $h_t=\iota_t^* g$. For $\sigma_1,\sigma_2 \in \Gamma(E)$, we observe that
\begin{equation}\label{preserve_metric}
\begin{aligned}
\frac{\pt}{\pt t} h_t(\sigma_1,\sigma_2)&=\frac{\pt}{\pt t} g(\iota_t(\sigma_1),\iota_t(\sigma_2))\\
&=g\Big(\frac{\pt}{\pt t}\iota_t(\sigma_1),\iota_t(\sigma_2)\Big) + g\Big( \iota_t(\sigma_1), \frac{\pt}{\pt t}\iota_t(\sigma_2)\Big)\\
&= \beta g_{ms} \iota_t(\sigma_2)_s X_k \iota_t(\sigma_1)_l \g2_{klm} + \beta g_{sm} \iota_t(\sigma_1)_s X_k \iota_t(\sigma_2)_l \g2_{klm}\\
&=\beta \iota_t(\sigma_2)_m X_k \iota_t(\sigma_1)_l \g2_{klm} +\beta \iota_t(\sigma_1)_m X_k \iota_t(\sigma_2)_l \g2_{klm}\\
&=0.
\end{aligned}
\end{equation}
Therefore, there is a fixed fibre metric $h$ on $E$ such that $h=\iota_t^* g$ for all $t$.

\begin{rmk}
A direct computation gives that the section $\widetilde{\g2}= \iota^*{\g2}$ satisfies an $ODE$. Explicitly,
\begin{align*}
\frac{\pt}{\pt t} \widetilde{\g2}(\sigma_1,\sigma_2,\sigma_3)&=\frac{\pt}{\pt t} \g2(\iota_t(\sigma_1),\iota_t(\sigma_2),\iota_t(\sigma_3))\\
&= X_l\psi_{lijk} \iota(\sigma_1)_i \iota(\sigma_2)_j \iota(\sigma_3)_k +\beta \g2_{ijk} X_l \iota(\sigma_1)_m \g2_{lmi} \iota(\sigma_2)_j \iota(\sigma_3)_k \\
&\qquad+\beta \g2_{ijk} X_l \iota(\sigma_2)_m \g2_{lmj} \iota(\sigma_1)_i \iota(\sigma_3)_k +\beta \g2_{ijk} X_l \iota(\sigma_3)_m \g2_{lmk} \iota(\sigma_1)_i \iota(\sigma_2)_j.
\end{align*}
Using the identity~\eqref{contractphpheq}, we get
\begin{align*}
\frac{\pt}{\pt t} \widetilde{\g2}(\sigma_1,\sigma_2,\sigma_3)&=X_l\psi_{lijk} \iota(\sigma_1)_i \iota(\sigma_2)_j \iota(\sigma_3)_k\\
&\qquad+\beta X_l \iota(\sigma_1)_m \iota(\sigma_2)_j \iota(\sigma_3)_k (g_{jl} g_{km} - g_{jm} g_{kl} - \psi_{jklm})\\
&\qquad+\beta X_l \iota(\sigma_1)_i \iota(\sigma_2)_m \iota(\sigma_3)_k (g_{kl} g_{im} - g_{km} g_{il} - \psi_{kilm})\\
&\qquad+\beta X_l \iota(\sigma_1)_i \iota(\sigma_2)_j \iota(\sigma_3)_m (g_{il} g_{jm} - g_{im}g_{jl} -\psi_{ijlm} )\\
&=X_l\psi_{lijk} \iota(\sigma_1)_i \iota(\sigma_2)_j \iota(\sigma_3)_k+\beta X_l  \iota(\sigma_1)_i \iota(\sigma_2)_j \iota(\sigma_3)_k \psi_{ijkl} \\
&\qquad+\beta X_l  \iota(\sigma_1)_i \iota(\sigma_2)_j \iota(\sigma_3)_k \psi_{ijkl} +\beta X_l \iota(\sigma_1)_i \iota(\sigma_2)_j \iota(\sigma_3)_k \psi_{ijkl}\\
&=(1+3\beta)X_l  \iota(\sigma_1)_i \iota(\sigma_2)_j \iota(\sigma_3)_k \psi_{ijkl}.
\end{align*}
We observe that if we choose $\beta=-\frac{1}{3}$, then $\widetilde{\g2}$ is constant. However, we will see that this is not the right choice.
\end{rmk}

In the particular case of the isometric flow we have $X=\Div T$. As in Section~\ref{sec:modifiedconnection}, let $\{v_a : 1 \leq a \leq 7\}$ be a local $h$-orthonormal frame for $E$, and let $\{e_i : 1 \leq i \leq 7 \}$ be a local $g$-orthonormal frame for $TM$. Let $A\in \Gamma(T^*M\times T^*M)$. Then $\widetilde A:= (\id\otimes \iota_t)^*A$ can be expressed with respect to these frames as
\begin{equation*}
\widetilde A_{ia}:=A(e_i,\iota_t(v_a)).
\end{equation*}

Now consider the evolution $A(t)\in \Gamma(T^*M\times T^*M)$ under the isometric flow. Then we have
\begin{align}\nonumber
\frac{\pt}{\pt t} \widetilde A_{ia} &= \Big(\frac{\pt}{\pt t} A \Big)(e_i,\iota_t(v_a)) +A(e_i, \frac{\pt}{\pt t} \iota_t(v_a))\\ \nonumber
&= \Big(\frac{\pt}{\pt t} A \Big)(e_i,\iota_t(v_a)) + \beta A( e_i, (\Div T)_k \iota_t(v_a)_l  \g2_{klm} e_m )\\ \label{uhlenbeck}
&= \Big(\frac{\pt}{\pt t} A \Big)(e_i,\iota_t(v_a)) + \beta A_{im}\nabla_p T_{pk} \iota_t(v_a)_l \g2_{klm}.
\end{align}

\subsection{The evolution of the torsion} \label{sec:evolution-torsion}

In this section we show that for particular choices of constants $\alpha,\beta$, pulling back the torsion on $TM^*\otimes E^*$ and expressing the evolution equation of Lemma~\ref{evolution_of_torsion} in terms of the modified connection $D$ results in the cancellation of the problematic first order term in~\eqref{evolT3}. Hence the torsion satisfies a reaction-diffusion equation, with respect to the Laplacian induced by the modified connection. Specifically, the following theorem takes $\alpha = -\tfrac{1}{2}$ and $\beta = \tfrac{1}{2}$.

\begin{thm} \label{react_diffuse}
There is a vector bundle $E$ isomorphic to $TM$, such that if a family of isomorphisms $\iota_t:E \rightarrow TM$ evolves by 
\begin{equation}
\frac{\pt}{\pt t}\iota_t(\sigma)= \frac{1}{2} \Div T \times \iota_t(\sigma)
\end{equation}
and $E$ is equipped with the family of connections $D$ given by
\begin{equation}
D_X \sigma = \iota_t^{-1}\left(\nabla_X  \iota_t(\sigma) -\frac{1}{2} (X\hk T)\times \iota_t(\sigma) \right),
\end{equation}
then there is a fibre metric $h$ on $E$ with $h=\iota_t^* g$ for all $t$, such that $D$ is compatible with $h$, and such that $\widetilde T=(\id\otimes \iota_t)^* T$ evolves by
\begin{equation} \label{torsion_reaction_diffusion}
\Big(\frac{\partial}{\partial t} - \Delta_{D}\Big) \widetilde T_{ia} = \frac{1}{4} \big(|\widetilde T|^2 \widetilde T - T\circ T^t \circ \widetilde T\big)_{ia} +(\id\otimes \iota_t)^*(F(\g2,T,\riem,\nabla \riem))_{ia}
\end{equation}
where $\Delta_D$ is the induced Laplacian on $(T^*M,\nabla)\otimes (E^*,D)$ and
\begin{equation} \label{Ftemp}
F(\g2,T,\riem,\nabla \riem)_{im}=\nabla_k R_{li}\g2_{klm}+R_{kiml}T_{kl} -\frac{1}{2} R_{kipq}T_{kl}\psi_{lpqm}-R_{ik}T_{km}.
\end{equation}
\end{thm}
\begin{proof}
Recall from Lemma~\ref{evolution_of_torsion} that under the isometric flow the torsion evolves by
\begin{equation} \label{tev}
\frac{\partial}{\partial t} T_{im} = \Delta T_{im} - \nabla_kT_{il}T_{kp}\g2_{plm}+F(\g2,T,\riem,\nabla \riem)_{im},
\end{equation}
where $F$ is given by~\eqref{Ftemp}.

Equations~\eqref{LaplaceDT} and~\eqref{uhlenbeck} imply that 
\begin{align*}
\Big( \frac{\partial}{\partial t} - \Delta_D \Big) \widetilde T_{ia}&=\Big(\frac{\pt}{\pt t}  T\Big) (e_i, \iota(v_a))+ \beta T_{im}\nabla_p T_{pk} \iota(v_a)_l \g2_{klm}\\
&\qquad -\Delta T (e_i,\iota(v_a)) + \alpha^2 \left[|T|^2 T_{iq} -(T\circ T^t \circ T)_{iq}\right] \iota(v_a)_q \\
&\qquad +2\alpha\nabla_k T_{ip} T_{km} \iota(v_a)_l \g2_{mlp} + \alpha T_{ip} \nabla_k T_{km} \iota(v_a)_l \g2_{mlp},
\end{align*}
which becomes
\begin{align*}
\Big( \frac{\partial}{\partial t} - \Delta_D \Big)  \widetilde T_{ia}&=\Big( \frac{\partial}{\partial t} - \Delta\Big) T(e_i,\iota(v_a))   + \alpha^2 \left[|T|^2 T_{iq} -(T\circ T^t \circ T)_{iq}\right] \iota(v_a)_q\\
&\qquad +(\beta+\alpha)T_{ip} \nabla_k T_{km} \iota(v_a)_l \g2_{mlp}  +2\alpha\nabla_k T_{ip} T_{km} \iota(v_a)_l \g2_{mlp}.
\end{align*}
Thus, using~\eqref{tev}, we obtain
\begin{align*}
\Big( \frac{\partial}{\partial t} - \Delta_D \Big) \widetilde T_{ia} &=-\nabla_k T_{il} T_{kp} \g2_{plm} \iota(v_a)_m+ \alpha^2 \left[|T|^2 T_{iq} -(T\circ T^t \circ T)_{iq}\right] \iota(v_a)_q + F_{im}\iota(v_a)_m\\
&\qquad +(\beta+\alpha)T_{ip} \nabla_k T_{km} \iota(v_a)_l \g2_{mlp}  +2\alpha\nabla_k T_{ip} T_{km} \iota(v_a)_l \g2_{mlp}\\
&=-(1+2\alpha) \nabla_k T_{il} T_{kp} \g2_{plm} \iota(v_a)_m +(\beta+\alpha)T_{ip} \nabla_k T_{km} \iota(v_a)_l \g2_{mlp} \\
& \qquad +\alpha^2 \left[|T|^2 T_{iq} -(T\circ T^t \circ T)_{iq}\right] \iota(v_a)_q + F_{im}\iota(v_a)_m.
\end{align*}
Hence choosing $\alpha=-\tfrac{1}{2}$ and $\beta = \tfrac{1}{2}$ completes the proof.
\end{proof}

\subsection{Second variation of the energy $E$} \label{sec:second-variation}

A similar modification of the Levi-Civita connection is also helpful to simplify the second variation of the energy functional, as described in the following proposition.

\begin{lemma} \label{second_variation}
Let $\bar{\g2}$ be $\G2$-structure on $(M,g)$ which is a critical point for the energy functional
\begin{equation*}
E(\g2) =\frac{1}{2} \int_M |T_\g2 |^2 \vol_g
\end{equation*}
with respect to variations preserving the metric. By Proposition~\ref{gradient}, this means that
\begin{equation*}
\Div T_{\bar{\g2}}=0.
\end{equation*}
Given any variation $(\g2_t)_{t\in(-\delta,\delta)}$ in the class $\llbracket \bar{\g2} \rrbracket$ with $\g2_0=\bar{\g2}$ and $X\in \Gamma(TM)$ satisfying 
\begin{equation*}
\left.\frac{d}{dt}\right|_{t=0} \g2_t = X\hk \bar\psi,
\end{equation*}
we have
\begin{equation} \label{secondvar}
\begin{aligned}
\left.\frac{d^2}{dt^2}\right|_{t=0} 4E(\g2_t) &= \int_M \left( |\nabla X|^2 -T_{km}\nabla_k X_p \bar{\g2}_{mlp} X_l  \right) \vol_g\\
&=\int_M \Big( |D X|^2 - \frac{1}{4}\big(|T|^2 |X|^2 - (T^t\circ T)(X,X) \big)\Big) \vol_g,
\end{aligned}
\end{equation}
where $D$ is the connection on $TM$ given by
\begin{equation} \label{secondvarcon}
D_Y X_p= \nabla_Y X_p -\frac{1}{2} Y_k T_{km} X_l \bar{\g2}_{mlp}.
\end{equation}
\end{lemma}
\begin{proof}
Since $\g2_t$ induce the same metric $g$ for all $t\in(-\delta,\delta)$ there is a family of vector fields $\mathcal X_t$ such that
\begin{equation*}
\frac{d}{dt} \g2_t = \mathcal X_t \hk \psi_t,
\end{equation*}
and $\mathcal X_0=X$. Therefore, by Proposition~\ref{gradient} we have
\begin{equation*}
\frac{d}{dt} 4 E(\g2_t) = - \int_M \nabla_p (T_{\g2_t})_{pq} (\mathcal X_t)_q \vol_g.
\end{equation*}
We now write $T$ for $T_{\bar{\g2}}$. Using $\nabla_p T_{pk} = 0$ and equations~\eqref{first-variation-eq} and~\eqref{delpheq}, we have
\begin{align*}
\left. \frac{d^2}{dt^2}\right|_{t=0} 4 E(\g2_t) &= -\int_M \bigg( \nabla_p  \Big( \left.\frac{d}{dt}\right|_{t=0} (T_{\g2_t} )_{pq} \Big) (\mathcal X_0)_q + \nabla_p T_{pq} \Big( \left.\frac{d}{dt}\right|_{t=0} \mathcal X_t \Big) \bigg) \vol_g\\
&= -\int_M \nabla_p (X_k T_{pl} \bar{\g2}_{klq} + \nabla_p X_q ) X_q  \vol_g\\
&=-\int_M ( T_{pl} \nabla_p X_k \bar{\g2}_{klq} + T_{pl} X_k T_{pm} \bar{\psi}_{mklq} + \Delta X_q) X_q \vol_g.
\end{align*}
The second term vanishes by symmetry considerations. Integrating by parts in the last term, we get
\begin{equation*}
\left. \frac{d^2}{dt^2}\right|_{t=0} 4 E(\g2_t) =\int_M \left(|\nabla X|^2 -  T_{pl} \nabla_p X_k \bar{\g2}_{klq} X_q\right)\vol_g,
\end{equation*}
which is the first line of~\eqref{secondvar}.

To deduce the second line of~\eqref{secondvar} we compute using~\eqref{secondvarcon} that
\begin{align*}
|DX|^2 &= (\nabla_k X_p - \frac{1}{2} T_{km} X_l \bar{\g2}_{mlp})(\nabla_k X_p - \frac{1}{2} T_{ka} X_b \bar{\g2}_{abp})\\
&= |\nabla X|^2 - \nabla_k X_p T_{km} X_l \bar{\g2}_{mlp} +\frac{1}{4} T_{km} T_{ka} X_l X_b \bar{\g2}_{mlp}\bar{\g2}_{abp}.
\end{align*}
Applying the contraction identity~\eqref{contractphpheq}, we obtain
\begin{align*}
|DX|^2 &= |\nabla X|^2 - \nabla_k X_p T_{km} X_l \bar{\g2}_{mlp} + \frac{1}{4} T_{km} T_{ka} X_l X_b (g_{ma} g_{lb} - g_{mb} g_{la} - \psi_{mlab} \\
& = |\nabla X|^2 - \nabla_k X_p T_{km} X_l \bar{\g2}_{mlp} + \frac{1}{4}\left(|T|^2 |X|^2 - (T^t\circ T)(X,X) \right).
\end{align*}
The above expression is rearranged to give the second line of~\eqref{secondvar}.
\end{proof}

\begin{rmk} \label{rmk:connection}
Both Theorem~\ref{react_diffuse} and Lemma~\ref{second_variation} make use of the connection $D$ from~\eqref{modified_connection} with the particular value $\alpha = - \tfrac{1}{2}$. It would be interesting to determine the geometric significance, if any, of this particular choice.
\end{rmk}

\section{Monotonicity, Entropy, $\varepsilon$-Regularity, and Consequences} \label{sec:monostuff}

In this section we first consider a quantity $\Theta$ that is \emph{almost monotonic} along the isometric flow. Then we introduce the \emph{entropy}, and use the almost monotonicity formula to prove an $\varepsilon$-regularity result and to prove that small entropy controls torsion. These in turn are used, together with the results from \textsection\ref{sec:estimates} to establish long-time existence and convergence of the flow given small entropy and to obtain results about the structure of singularities for the flow.

\subsection{An almost monotonicity formula} \label{sec:monofor}

Given a complete Riemannian manifold $(M,g)$ with \emph{bounded curvature} and $(x_0,t_0)\in M \times \mathbb R$, we denote by $u_{(x_0,t_0)}$ the \emph{kernel of the backwards heat equation} on $M$ starting at $\delta_{x_0}$ at time $t_0$. Explicitly, for $t\in (-\infty,t_0)$ we have
\begin{equation} \label{eq:back-heat}
\begin{aligned}
\Big( \frac{\pt}{\pt t} + \Delta_g \Big) u_{(x_0,t_0)} &= 0,\\
\lim_{t\rightarrow t_0-} u_{(x_0,t_0)} &= \delta_{x_0}.
\end{aligned}
\end{equation}
We also define the smooth function $f_{(x_0,t_0)}$ by the relation
\begin{equation} \label{ufrelation}
u_{(x_0,t_0)}=\frac{e^{-f_{(x_0,t_0)}}}{(4\pi(t_0 - t))^{\frac{7}{2}}}.
\end{equation}

\begin{defn} \label{defn:Theta}
Let $(\varphi(t))_{t\in [0,t_0)}$ be an isometric flow on $M$ inducing the metric $g$ and define
\begin{equation} \label{eq:Theta}
\Theta_{(x_0,t_0)}(\varphi(t))=(t_0 - t)\int_M |T_{\varphi(t)}|^2 u_{(x_0,t_0)} \vol_g.
\end{equation}
From the discussion in Section~\ref{sec:rescaling}, it follows that the quantity $\Theta_{(x_0,t_0)}$ is invariant under parabolic rescaling. In what follows we will simply write $u$ for $u_{(x_0,t_0)}$.
\end{defn}
One can think of $\Theta$ as a kind of ``localized energy'', but we will not use this terminology.

\begin{lemma} \label{monotonicity_formula}
Let $(\g2(t))_{t\in [0,t_0)}$ be an isometric flow on a complete Riemannian manifold $(M,g)$ with bounded geometry, as in Remark~\ref{localestlemma_ncpct}. If $M$ is noncompact suppose further that the torsion $T_{\g2(t)}$ has at most polynomial growth. Then $\Theta_{(x_0,t_0)}(\g2(t))$ evolves by
\begin{equation}
\begin{split} \label{mon_form}
\frac{d}{dt} \Theta_{(x_0,t_0)}(\varphi(t)) &= \int - 2 (t_0-t)  |\Div T - \nabla f \hk T |^2 u \vol_g \\
&\qquad-\int 2(t_0-t)  T_{lq}T_{pq}\Big( \nabla_p \nabla_l u - \frac{\nabla_p u \nabla_l u}{u} + \frac{u g_{pl}}{2(t_0-t)} \Big) \vol_g\\
&\qquad+\int (t_0 - t) ( 2 \nabla_a R_{bp} \varphi_{abq} T_{pq} - R_{lpab}T_{lm}\psi_{mabq} T_{pq} ) u \vol_g\\
&\qquad -\frac{1}{2} \int (t_0 - t) R_{lpab} (2 T_{la} T_{pb} - T_{lm} T_{pn} \psi_{abmn} + R_{lpab} - \frac{1}{2} R_{lpmn} \psi_{abmn}) u \vol_g.
\end{split}
\end{equation}
\end{lemma}
\begin{proof}
To justify the following argument in the case when $M$ is noncompact, note that the local derivative estimates of Theorem~\ref{localestthm} imply that the derivatives of the torsion also have polynomial growth. This is because for any $(x,t)\in M\times (0,t_0)$ we can apply Theorem~\ref{localestthm} in parabolic balls $P_r(x,t)$ of a uniform radius $r=\min(s,\sqrt{t})$ and some $K_{(x,t)}>r^{-2}$ that grows at most polynomially at infinity. Then Theorem~\ref{localestthm} provides bounds 
\begin{equation}
|\nabla^m T|(x,t) \leq C_m K_{(x,t)}^{m+1},
\end{equation}
that also grow at most polynomially at infinity. We also need the fact that the heat kernel decays exponentially~\cite[Corollary 3.1]{li-yau}. With these observations all the integrals below are well-defined even in the noncompact case. 

We now proceed with the proof. Using~\eqref{eq:Theta} and~\eqref{eq:back-heat} we compute
\begin{align*}
\frac{d}{dt} \Theta_{(x_0,t_0)}(\varphi(t))&= \int_{M\times \{t\}}  \bigg( \Big( (t_0-t) \frac{\pt}{\pt t} |T|^2  -|T|^2 \Big) u + (t_0-t) |T|^2 \frac{\pt u}{\pt t} \bigg) \vol_g \\
&=\int \Big( 2 (t_0-t) T_{pq} \frac{\pt  T_{pq}}{\pt t} u - |T|^2  u - (t_0-t) |T|^2 \Delta u \Big) \vol_g.
\end{align*}
Substituting the evolution of $T_{pq}$ from~\eqref{evolT1} we get
\begin{align*}
\frac{d}{dt} \Theta_{(x_0,t_0)}(\varphi(t))&= \int \Big( 2 (t_0-t) T_{pq} (\nabla_p\nabla_i T_{iq} + T_{pl}\del_iT_{ik}\g2_{klq} )u - |T|^2  u - (t_0-t) |T|^2 \Delta u \Big) \vol_g.
\end{align*}
The part of the first term with only one derivative of torsion vanishes because $T_{pq} T_{pl}$ is symmetric in $q,l$. Integrating by parts on the first and third terms, we obtain
\begin{align*}
\frac{d}{dt} \Theta_{(x_0,t_0)}(\varphi(t)) & = \int \Big( -2 (t_0-t) \left( |\Div T|^2 u  + T_{pq}\nabla_i T_{iq} \nabla_p u \right)- |T|^2  u +(t_0-t) \nabla_l |T|^2 \nabla_l u \Big) \vol_g \\
&=\int \Big( - 2 (t_0-t) \left( |\Div T|^2 u + T_{pq}\nabla_i T_{iq} \nabla_p u\right)  - |T|^2  u +2 (t_0-t)\nabla_l T_{pq} T_{pq} \nabla_l u \Big) \vol_g.
\end{align*}
Using the $\G2$-Bianchi identity~\eqref{G2B} we obtain 
\begin{align*}
\frac{d}{dt} \Theta_{(x_0,t_0)}(\varphi(t)) &=\int \big[- 2 (t_0-t) \left( |\Div T|^2 u + T_{pq}\nabla_i T_{iq} \nabla_p u\right)  - |T|^2  u \big] \vol_g \\
&\qquad+\int \big[ 2 (t_0-t) ( \nabla_p T_{lq}+T_{la}T_{pb}\g2_{abq} +\frac{1}{2} R_{lpab}\g2_{abq} ) T_{pq} \nabla_l  u \big] \vol_g.
\end{align*}
Integrating by parts on the first term of the second line gives 
\begin{align*}
\frac{d}{dt} \Theta_{(x_0,t_0)}(\varphi(t)) &=\int \big[- 2 (t_0-t) \left( |\Div T|^2 u + T_{pq}\nabla_i T_{iq} \nabla_p u\right)  - |T|^2  u \big] \vol_g \\
&\qquad+\int \big[ 2 (t_0-t) ( T_{la}T_{pb}\g2_{abq} +\frac{1}{2} R_{lpab}\g2_{abq} ) T_{pq} \nabla_l  u \big] \vol_g\\
&\qquad-\int 2(t_0-t) \left[ T_{lq} \nabla_p T_{pq} \nabla_l u + T_{lq}T_{pq}\nabla_p \nabla_l u \right] \vol_g,
\end{align*}
The first term in the second line vanishes by symmetry considerations. The two terms where $\Div T$ appears linearly combine. Using also that $\nabla_l u = - u \nabla_l f$ from~\eqref{ufrelation}, we have
\begin{align*}
\frac{d}{dt} \Theta_{(x_0,t_0)}(\varphi(t)) &=\int \left[- 2 (t_0-t) \left( |\Div T|^2  - 2T_{pq}\nabla_i T_{iq} \nabla_p f\right) u  - |T|^2  u \right] \vol_g \\
&\qquad+\int  (t_0-t)  R_{lpab}\g2_{abq}  T_{pq} \nabla_l  u \vol_g-\int 2(t_0-t)  T_{lq}T_{pq}\nabla_p \nabla_l u \vol_g.
\end{align*}
Next we complete the square and manipulate the expression algebraically to get
\begin{align*}
\frac{d}{dt} \Theta_{(x_0,t_0)}(\varphi(t)) &=\int - 2 (t_0-t)  |\Div T - \nabla f \hk T |^2 u \vol_g \\
&\qquad-\int 2(t_0-t)  T_{lq}T_{pq}\Big( \nabla_p \nabla_l u - \frac{\nabla_p u \nabla_l u}{u} + \frac{u g_{pl}}{2(t_0-t)} \Big) \vol_g\\
&\qquad+\int  (t_0-t)  R_{lpab}\g2_{abq}  T_{pq} \nabla_l  u \vol_g.
\end{align*}
The first two terms above are now in the form in which they appear in~\eqref{mon_form}. Thus, in order to complete the proof we need to show that the last term above can be written as
\begin{equation} \label{monotemp}
\begin{aligned}
& \int (t_0-t)  R_{lpab}\g2_{abq}  T_{pq} \nabla_l  u \vol_g \\
& \qquad = \int (t_0 - t) \Big( 2 \nabla_a R_{bp} \varphi_{abq} T_{pq} - R_{lpab}T_{lm}\psi_{mabq} T_{pq} \Big) u \vol_g\\
&\qquad \qquad -\frac{1}{2} \int (t_0 - t) R_{lpab} (2 T_{la} T_{pb} - T_{lm} T_{pn} \psi_{abmn} + R_{lpab} - \frac{1}{2} R_{lpmn} \psi_{abmn}) u \vol_g.
\end{aligned}
\end{equation}
To establish~\eqref{monotemp}, we first integrate by parts to get
\begin{align*}
& \int  (t_0-t)  R_{lpab}\g2_{abq}  T_{pq} \nabla_l  u \vol_g \\
& \qquad = -\int (t_0-t) \left(\nabla_l R_{lpab} \varphi_{abq} T_{pq} +R_{lpab}\nabla_l \varphi_{abq} T_{pq} + R_{lpab} \g2_{abq} \nabla_l T_{pq}\right)u \vol_g.
\end{align*}
In the above expression, we use the identity~\eqref{riem2Beq} in the first term, equation~\eqref{delpheq} in the second term, and the skew-symmetry of $R_{lpab}$ in $l,p$ in the third term to obtain
\begin{align*}
& \int (t_0-t)  R_{lpab}\g2_{abq}  T_{pq} \nabla_l  u \vol_g \\
& \qquad = \int (t_0 - t) \Big( (\nabla_a R_{bp} - \nabla_b R_{ap}) \varphi_{abq} T_{pq} - R_{lpab}T_{lm}\psi_{mabq} T_{pq} \Big) u \vol_g\\
& \qquad \qquad -\frac{1}{2}\int_M (t_0 - t) R_{lpab} \g2_{abq} (\nabla_l T_{pq} - \nabla_p T_{lq}) u \vol_g.
\end{align*}
Using the skew-symmetry of $\varphi$ in the first line and the $\G2$-Bianchi identity~\eqref{G2B} in the second line, the above becomes
\begin{align*}
\int (t_0-t)  R_{lpab}\g2_{abq}  T_{pq} \nabla_l  u \vol_g & = \int (t_0 - t) \left( 2 \nabla_a R_{bp} \varphi_{abq} T_{pq} - R_{lpab}T_{lm}\psi_{mabq} T_{pq} \right) u \vol_g\\
&\qquad -\frac{1}{2}\int (t_0 - t) R_{lpab} \g2_{abq} (T_{lm}T_{pn} +\frac{1}{2} R_{lpmn}) \g2_{mnq} u \vol_g.
\end{align*}
The first term on the right hand side above is now in the form in which it appears in~\eqref{monotemp}. Thus, in order to complete the proof we need to show that the second term above can be written as
\begin{equation} \label{monotemp2}
\begin{aligned}
& -\frac{1}{2}\int (t_0 - t) R_{lpab} \g2_{abq} (T_{lm}T_{pn} +\frac{1}{2} R_{lpmn}) \g2_{mnq} u \vol_g \\
& \qquad = -\frac{1}{2} \int (t_0 - t) R_{lpab} (2 T_{la} T_{pb} - T_{lm} T_{pn} \psi_{abmn} + R_{lpab} - \frac{1}{2} R_{lpmn} \psi_{abmn}) u \vol_g.
\end{aligned}
\end{equation}
Applying the contraction identity~\eqref{contractphpheq} to the left hand side above and using symmetries of $R_{lpab}$ we get
\begin{align*}
& -\frac{1}{2}\int (t_0 - t) R_{lpab} \g2_{abq} (T_{lm}T_{pn} +\frac{1}{2} R_{lpmn}) \g2_{mnq} u \vol_g \\
& \qquad = -\frac{1}{2} \int (t_0 - t) R_{lpab} (T_{lm}T_{pn} +\frac{1}{2} R_{lpmn}) ( g_{am} g_{bn} - g_{an} g_{bm} - \psi_{abmn} ) u \vol_g \\
& \qquad = -\frac{1}{2} \int (t_0 - t) R_{lpab} (2 T_{la} T_{pb} - T_{lm} T_{pn} \psi_{abmn} + R_{lpab} - \frac{1}{2} R_{lpmn} \psi_{abmn}) u \vol_g
\end{align*}
as required. This completes the proof of~\eqref{mon_form}.
\end{proof}

Next we prove an \emph{almost monotonicity formula} for the quantity $\Theta_{(x_0, t_0)}(\g2(t))$.

\begin{thm}[Almost monotonicity formula] \label{almost_mon}
This theorem has two versions, as follows.
\begin{enumerate}[(1)]
\item Let $(M^7,g)$ be compact and let $(\g2(t))_{t\in [0,t_0)}$ be an isometric flow inducing the metric $g$. Then for any $x_0\in M$ and $\max\{0,t_0-1\} <\tau_1<\tau_2<t_0$, there exist $K_1, K_2 > 0$ depending only on the geometry of $(M, g)$ such that the following monotonicity formula holds:
\begin{equation} \label{eq:AM-1}
\Theta_{(x_0,t_0)}(\varphi(\tau_2)) \leq K_1 \Theta_{(x_0,t_0)}(\varphi(\tau_1)) + K_2(\tau_2-\tau_1) ( E(\varphi(0))+1 ).
\end{equation}
\item Let $(M,g) = (\mathbb R^7,g_{\mathrm{Eucl}})$ and let $(\g2(t))_{t\in [0,t_0)}$ be an isometric flow inducing $g_{\mathrm{Eucl}}$. Then for any $x_0\in \mathbb R^7$ and $0\leq \tau_1 <\tau_2 <t_0$ we have strict monotonicity
\begin{equation*}
\Theta_{(x_0,t_0)}(\varphi(\tau_2)) \leq  \Theta_{(x_0,t_0)}(\varphi(\tau_1))
\end{equation*}
with equality if and only if for all $t\in [\tau_1,\tau_2]$
\begin{equation*}
\Div T_{\g2(t)} = \frac{x-x_0}{2(t_0-t)} \hk T_{\g2(t)}.
\end{equation*}
\end{enumerate}
\end{thm}
\begin{proof}
(1) Let $\max\{ 0, t_0 - 1\} < t < t_0$. We first control the last two terms in~\eqref{mon_form} in terms of the geometry of $(M,g)$. For the third term in~\eqref{mon_form}, since the curvature is not evolving and thus is uniformly bounded, applying Young's inequality, $\int u \vol_g =1$, and $t_0 - t \leq 1$, we obtain 
\begin{align}
\begin{split} \label{curv_term_1}
& \int (t_0 - t) \left( 2 \nabla_a R_{bp} \varphi_{abq} T_{pq} - R_{lpab}T_{lm}\psi_{mabq} T_{pq} \right) u \vol_g \\
& \qquad \leq C\left(\int (t_0 - t) |\nabla \tRic|^2 u \vol_g+ \int (t_0-t) |T|^2 u \vol_g \right)\\
& \qquad \leq C\left( t_0-t + \int (t_0-t) |T|^2 u \vol_g \right) \\
& \qquad \leq C\Big(1 + \Theta_{(x_0,t_0)}(\varphi(t)) \Big).
\end{split}
\end{align}
For the fourth term in~\eqref{mon_form} again using the same facts as above, we get 
\begin{align}
\begin{split} \label{curv_term_2}
& -\frac{1}{2} \int (t_0 - t) R_{lpab} (2 T_{la} T_{pb} - T_{lm} T_{pn} \psi_{abmn} + R_{lpab} - \frac{1}{2} R_{lpmn} \psi_{abmn}) u \vol_g \\
& \qquad \leq C \int (t_0 - t) |\riem|^2 u \vol_g + C \int (t_0 - t) |\riem|^2 |T|^2 u \vol_g \\
& \qquad \leq C\left( t_0-t + \int (t_0-t) |T|^2 u \vol_g \right) \\
& \qquad \leq C\Big(1 + \Theta_{(x_0,t_0)}(\varphi(t)) \Big).
\end{split}
\end{align}

We now focus on the second term of~\eqref{mon_form}. Our estimates depend on the control of the backwards heat kernel allowed by the geometry. Recall that from~\cite{Ham93_monotonicity}, there are constants $A,B < +\infty$ depending only on $(M,g)$ such that for $t\in (t_0-1,t_0)$ we have
\begin{equation} \label{Ham2}
\nabla_p \nabla_l u - \frac{\nabla_p u \nabla_l u}{u} + \frac{u g_{pl}}{2(t_0-t)} \geq -\frac{A}{2}  \left( 1+  u \log \left(\frac{B}{(t_0-t)^{\frac{7}{2}} } \right) \right) g_{pl}. 
\end{equation}
Therefore, following~\cite{Ham93_monotonicity}, we can bound for $t\in [t_0-1,t_0)$ the second term in~\eqref{mon_form} using~\eqref{Ham2} as
\begin{align*}
& -\int 2(t_0-t) T_{lq}T_{pq} \left( \nabla_p \nabla_l u - \frac{\nabla_p u \nabla_l u}{u} + \frac{u g_{pl}}{2(t_0-t)} \right) \vol_g \\
& \qquad \leq C (t_0 - t) \int |T|^2 \vol_g + C\log \left(\frac{B}{(t_0-t)^{\frac{7}{2}}} \right) \int (t_0-t) |T|^2 u \vol_g\\
& \qquad \leq C E(\varphi(t)) + C\log \left(\frac{B}{(t_0-t)^{\frac{7}{2}}} \right)\Theta_{(x_0,t_0)}(\varphi(t)).
\end{align*}
Since $\g2(t)$ is evolving by the negative gradient flow of $4E$, the function $E(\g2(t))$ is non-increasing, thus we can say
\begin{equation}
\begin{split} \label{AM-term2}
& -\int 2(t_0-t) T_{lq}T_{pq} \left( \nabla_p \nabla_l u - \frac{\nabla_p u \nabla_l u}{u} + \frac{u g_{pl}}{2(t_0-t)} \right) \vol_g \\
& \qquad \leq C E(\varphi(0)) + C\log \left(\frac{B}{(t_0-t)^{\frac{7}{2}}} \right)\Theta_{(x_0,t_0)}(\varphi(t)).
\end{split}
\end{equation}

Combining the estimates~\eqref{curv_term_1},~\eqref{curv_term_2}, and~\eqref{AM-term2}, from~\eqref{mon_form} we obtain
\begin{equation} \label{ODI}
\begin{aligned}
\frac{d}{dt} \Theta_{(x_0,t_0)}(\varphi(t)) &\leq - 2 (t_0-t)\int \left|\Div T - \nabla f \hk T \right|^2 u \vol_g \\
& \qquad + C_1\left[1+\log \left(\frac{B}{(t_0-t)^{\frac{7}{2}}} \right) \right] \Theta_{(x_0,t_0)}(\varphi(t)) + C_2 [ E(\varphi(0))+1 ].
\end{aligned}
\end{equation}
Now consider the function
\begin{equation*}
\zeta(t)=\left(-1-\log B+\frac{7}{2}\log(t_0-t) - \frac{7}{2} \right)(t_0-t)
\end{equation*}
which satisfies 
\begin{equation*}
\zeta'(t)=1+\log \frac{B}{(t_0-t)^{\frac{7}{2}}}.
\end{equation*}
Multiplying~\eqref{ODI} by the integrating factor $e^{-C_1 \zeta(t)}$, we obtain
\begin{equation*}
\begin{aligned}
&\frac{d}{dt} [e^{-C_1 \zeta(t)} \Theta_{(x_0,t_0)}(\varphi(t))] \\
& \qquad =e^{-C_1 \zeta(t)} \left(\frac{d}{dt} \Theta_{(x_0,t_0)}(\varphi(t)) - C_1 \zeta'(t)  \Theta_{(x_0,t_0)}(\varphi(t)) \right)\\
& \qquad \leq e^{-C_1 \zeta(t)} \left( - 2 (t_0-t)\int \left|\Div T - \nabla f \hk T \right|^2 u \vol_g +C_2 [ E(\varphi(0))+1 ]\right).
\end{aligned}
\end{equation*}
Dropping the first term above, which is nonpositive, we deduce that
\begin{equation}\label{ODI2}
\frac{d}{dt} [e^{-C_1 \zeta(t)} \Theta_{(x_0,t_0)}(\varphi(t))] \leq C_2 e^{-C_1\zeta(t)} ( E(\varphi(0))+1 ) \leq K ( E(\varphi(0))+1 ),
\end{equation}
for some constant $K<+\infty$, since $\zeta(t)$ is bounded for $t\in [t_0-1,t_0)$. Hence, for any $t_0-1<\tau_1<\tau_2<t_0$ we can integrate~\eqref{ODI2} to obtain
\begin{equation*}
e^{-C_1 \zeta(\tau_2)} \Theta_{(x_0,t_0)}(\varphi(\tau_2)) \leq e^{-C_1 \zeta(\tau_1)} \Theta_{(x_0,t_0)}(\varphi(\tau_1)) + K(\tau_2-\tau_1) ( E(\varphi(0))+1 ),
\end{equation*}
from which the result~\eqref{eq:AM-1} follows.

(2) When $(M,g)=(\mathbb R^7,g_{\mathrm{Eucl}})$, the backwards heat kernel is explicitly
\begin{equation*}
u(x,t) = \frac{e^{-\frac{|x-x_0|^2}{4(t_0-t)}}}{(4\pi (t_0 - t))^{\frac{7}{2}}},
\end{equation*}
with $f(x,t)=\frac{|x-x_0|^2}{4(t_0-t)}$ and thus it satisfies
\begin{equation*}
\nabla_p \nabla_l u - \frac{\nabla_p u \nabla_l u}{u} + \frac{u g_{pl}}{2(t_0-t)} = 0.
\end{equation*}
Therefore, because there are no curvature terms in~\eqref{mon_form} and $\nabla f = \frac{x-x_0}{2(t_0-t)}$, we obtain
\begin{align*}
\frac{d}{dt} \Theta_{(x_0,t_0)} (\g2(t)) = -2\int |\Div T - \frac{x-x_0}{2(t_0-t)} \hk T |^2 u \vol_g \leq 0,
\end{align*}
which immediately implies the result.
\end{proof}

\begin{rmk} \label{almost_mon-solitons}
We note that in Theorem~\ref{almost_mon} (2), the case of equality corresponds to a particular special type of \emph{shrinking isometric soliton} on $(\R^7, g_{\mathrm{Eucl}})$, as described in~\eqref{r7solitons}.
\end{rmk}

\subsection{Entropy and $\varepsilon$-regularity} \label{sec:entropy-regularity}

The energy functional, although quite natural, has the disadvantage that it is not scale invariant. As a result, it is not strong enough to control the small scale behaviour of a $\G2$-structure $\g2$. In this section, motivated by analogous functionals for the mean curvature flow~\cite{colding-minicozzi}, the high dimensional Yang-Mills flow~\cite{kelleher-streets} and the Harmonic map heat flow~\cite{boling-kelleher-streets}, we introduce an \emph{entropy} functional, and use it and the almost monotonicity of Section~\ref{sec:monofor} to establish an $\epsilon$-regularity result, as well as to show that small entropy controls torsion.

\begin{defn} \label{defn_entropy}
Let $(M,\g2)$ be a compact manifold with $\G2$-structure inducing the Riemannian metric $g$. Let $u_{(x,t)}(y,s)= u^g_{(x,t)}(y,s)$ denote the backwards heat kernel, with respect to $g$, that becomes $\delta_{(x,t)}$ as $s\rightarrow t$. For $\sigma > 0$ we define
\begin{equation} \label{eq:entropy}
\lambda(\g2,\sigma) = \max_{(x,t)\in M\times (0,\sigma]} \left\{t \int_M |T_{\g2}|^2(y) u_{(x,t)}(y,0) \vol_g(y)\right\}.
\end{equation}
We call $\lambda(\g2, \sigma)$ the \emph{entropy} of $(M, \g2)$. The precise value of $\sigma$ is not important, only that $\sigma > 0$. One should think of $\sigma$ as the ``scale'' at which we are analyzing the flow.
\end{defn}

Note that in the definition above the maximum is achieved, because $M$ is assumed to be compact. Moreover, the entropy functional $\lambda$ is invariant under parabolic rescaling in the sense that
\begin{equation*}
\lambda( c^3 \g2, c^2 \sigma ) = \lambda(\g2,\sigma).
\end{equation*}
To see this, first note that $u^{c^2 g}_{(x,c^2 t)} (y,0) = c^7 u^{g}_{(x,t)} (y,0)$. Using this and the discussion from Section~\ref{sec:rescaling} to compute
\begin{align*}
\lambda( c^3 \g2, c^2 \sigma ) &= \max_{(x,c^2 t)\in M\times (0,c^2 \sigma]}\left\{ c^2 t \int_M |T_{c^3 \g2}|^2(y) u^{c^2 g}_{(x,c^2 t)} (y,0) \vol_{c^2 g} (y) \right\}\\
&=\max_{(x, t)\in M\times (0, \sigma]}\left\{  t \int_M |T_{\g2}|^2(y) u^{ g}_{(x, t)} (y,0) \vol_{ g} (y) \right\}\\
&= \lambda(\g2,\sigma).
\end{align*}

We need the following technical result, which is a consequence of work of Hamilton~\cite{Ham93_harnack}.
\begin{lemma} \label{nearby}
Let $(M,g)$ be a compact Riemannian manifold and let $(\g2(t))_{t\in [0,t_0)}$ be an isometric flow with $g_{\g2(t)}=g$ and $E(\g2(0))\leq E_0$. Then, for every $\varepsilon>0$ there exist $\delta=\delta(\varepsilon,g,E_0)>0$  and $\bar r=\bar r(\varepsilon, g,E_0)>0$ such that
\begin{equation*}
\text{if, at the point $x_0\in M$, we have } \Theta_{(x_0,t_0)} (\g2(t_1)) < \delta \text{ for some $t_1\in [0,t_0)$},
\end{equation*} 
then, for $r=\min(\bar r,\sqrt{t_0-t_1})$, we have that
\begin{equation*}
\text{ for every $(x,t)\in B(x_0,r)\times [t_0-r^2,t_0]$ we have } \Theta_{(x,t)}(\g2(t_1)) < \varepsilon.
\end{equation*}
\end{lemma}
\begin{proof}
By~\cite[Theorem 3.1]{Ham93_harnack} and the symmetry of the heat kernel (see also~\cite[Theorem 3.2]{HamGray96}), for every $\eta>0$ and $C>1$ there is an $\bar r (\eta,C,g)>0$ such that for any  $r\in (0,\bar r]$,  $(y,s)\in B(x,r)\times [t_0-r^2,t_0]$ and $(z,t_1)\in M\times [t_0-1,t_0-r^2)$ we have
\begin{equation}
(s-t_1) u_{(y,s)}(z,t_1) \leq C (t_0-t_1) u_{(x_0,t_0)}(z,t_1) + \eta.
\end{equation}
Multiplying both sides by $|T|^2(z, t_1)$ and integrating with respect to $z$ we conclude from~\eqref{eq:Theta} that
\begin{equation}
\Theta_{(y,s)}(\g2(t_1)) \leq C \Theta_{(x_0,t_0)} (\g2(t_1)) +\eta E(\g2(t_1)) \leq C \delta +\eta E_0 <\varepsilon,
\end{equation}
for every $(y,s)\in B(x,r)\times [t_0-r^2,t_0]$, where $r=\min(\bar r, \sqrt{t_0-t_1})$, as long as we choose $\eta, \delta>0$ small enough, which proves the result.
\end{proof}

We can now establish an $\varepsilon$-regularity result for the isometric flow, using our almost monotonicity formula from part (1) of Theorem~\ref{almost_mon}.
\begin{thm}[$\varepsilon$-regularity] \label{eregularity}
Given $(M,g)$ compact and $E_0<+\infty$ there exist $\varepsilon, \bar\rho>0$ such that for every $\rho\in(0,\bar\rho]$ there exist $r\in \left(0,\rho\right)$ and $C<+\infty$ such that the following holds:

If $(M,\g2(t))_{t\in [0, t_0)}$ is an isometric flow with $g_{\g2(t)}=g$ and $E(\g2(0))\leq E_0$, and if $x_0\in M$ is such that
\begin{equation*}
\Theta_{(x_0,t_0)}(\g2(t_0-\rho^2))<\varepsilon
\end{equation*}
then 
\begin{equation*}
\Lambda_r(x,t) |T_{\g2}(x,t)| \leq C r^{-1}
\end{equation*}
in $B(x_0,r)\times [t_0-r^2,t_0]$, where
\begin{equation*}
\Lambda_r(x,t) = \min\big(1-r^{-1}d_g(x_0,x), \sqrt{1- r^{-2} (t_0-t)}\big).
\end{equation*}
\end{thm}
\begin{proof} 
We prove this by contradiction. Suppose the result does not hold. Then for any sequences $\varepsilon_i\rightarrow 0$ and $\bar\rho_i\rightarrow 0$ there exist $\rho_i\in (0,\bar\rho_i]$ such that for any $r_i\in (0,\rho_i)$ and $C_i\rightarrow +\infty$ there are counterexamples $(M,\g2_i(t))_{t\in [0,t_i)}$ with $g_{\g2(t)}=g$, $E(\g2(0))\leq E_0$, and $x_i\in M$, such that
\begin{equation*}
\Theta_{(x_i,t_i)}(\g2_i(t_i-\rho_i^2)) <\varepsilon_i,
\end{equation*} 
but 
\begin{equation} \label{contradiction}
r_i \Big(\max_{(x,t)\in B(x_i,r_i)\times [t_i-r_i^2,t_i]} \Lambda_{r_i}(x,t) |T_{\g2_i}(x,t)| \Big) > C_i .
\end{equation}
Passing to a subsequence and applying Lemma~\ref{nearby} we can choose $r_i$ such that 
\begin{equation} \label{density_nearby}
\Theta_{(x,t)} (\g2_i(t_i-\rho_i^2)) < \frac{1}{i}
\end{equation}
for every $(x,t)\in B(x_i,r_i)\times [t_i-r_i^2,t_i]$.

Now set 
\begin{equation} \label{Lambdaieq}
\Lambda_i (x,t) =  \Lambda_{r_i}(x,t)=\min \big( 1-r_i^{-1}d_g(x_i,x), \sqrt{1-r_i^{-2}  (t_i-t)} \big),
\end{equation}
and let $(\bar x_i,\bar t_i)\in B(x_i,r_i)\times [t_i-r_i^2,t_i]$ attain the maximum in~\eqref{contradiction}. Then, setting 
\begin{equation} \label{scale_factor}
Q_i=|T_{\g2_i}(\bar x_i,\bar t_i)|,
\end{equation}
we have
\begin{equation}\label{torsion_local}
|T_{\g2_i} (x,t) | \Lambda_i (x,t) \leq Q_i \Lambda_i (\bar x_i,\bar t_i),
\end{equation}
for all $(x,t) \in B(x_i,r_i) \times [t_i - r_i^2, t_i]$. 
Moreover, by~\eqref{contradiction} we have
\begin{equation} \label{goes_to_infty}
Q_i \Lambda_i (\bar x_i,\bar t_i) \geq \rho_i Q_i \Lambda_i (\bar x_i,\bar t_i) \geq r_i Q_i \Lambda_i (\bar x_i,\bar t_i) \rightarrow +\infty,
\end{equation}
and thus
\begin{equation} \label{rhoQ}
\rho_i Q_i \rightarrow +\infty,
\end{equation}
since $\Lambda_i(\bar x_i,\bar t_i) \in [0,1]$. In particular, since $\rho_i \in (0, \bar \rho_i]$ and $\bar \rho_i \rightarrow 0$, we deduce from this and~\eqref{rhoQ} that
\begin{equation} \label{Q_to_infty}
\rho_i \rightarrow 0, \qquad Q_i \rightarrow + \infty.
\end{equation}

Now consider the rescaled flow
\begin{align*}
\tilde{\g2_i}(t)&=Q_i^3\g2_i(\bar t_i+tQ_i^{-2}), \\
g_i &= Q_i^2 g,
\end{align*}
for $t\leq 0$, and the pointed sequence $(M, \tilde{\g2_i}(t), g_i ,\bar x_i)$. By~\eqref{torsion_local}, each $\widetilde{\g2}_i$ satisfies
\begin{equation}\label{torsion_b}
|T_{\widetilde{\g2}_i}(x,t)|\leq \frac{\Lambda_i(\bar x_i,\bar t_i)}{\Lambda_i(x,\bar t_i+tQ_i^{-2})}.
\end{equation}
By~\eqref{rescaling} and the definition of $Q_i$ in~\eqref{scale_factor} we have
\begin{equation} \label{torsion_one}
|T_{\widetilde{\g2}_i} (\bar x_i , 0)| =1.
\end{equation}

If $d_{g_i}(\bar x_i,x) = Q_i d_g(\bar x_i,x) \leq \frac{Q_i}{2} (r_i - d_g(x_i,\bar x_i))$, then $r_i - d_g (x_i, \bar x_i) - d_g(\bar x_i, x) \geq \tfrac{1}{2} (r_i - d(x_i, \bar x_i))$. Using this and the triangle inequality we have 
\begin{equation}\label{space}
r_i - d_g(x_i,x) \geq r_i - d_g(x_i,\bar x_i) - d_g(\bar x_i,x)\geq \frac{r_i - d_g(x_i,\bar x_i ) }{2}.
\end{equation}
Also, if $|t| \leq \frac{3}{4} Q_i^2 (r_i^2 - (t_i - \bar t_i))$, then $ |t Q_i^{-2}| \leq \frac{3}{4}(r_i^2 - (t_i-\bar t_i)) $ and
\begin{equation}\label{time}
\begin{aligned}
r_i^2 - (t_i-(\bar t_i + t Q_i^{-2})) &= r_i^2 -(t_i -\bar t_i) +t Q_i^{-2} \geq r_i^2 -(t_i -\bar t_i) -\frac{3}{4} (r_i^2 - (t_i-\bar t_i))  \\
&\geq \frac{1}{4}(r_i^2 - (t_i-\bar t_i) ).
\end{aligned}
\end{equation}
Therefore, for $(x,t) \in B_{g_i}(\bar x_i, R_i) \times [-R_i^2 , 0]$, with 
\begin{equation*}
R_i =\frac{1}{2} Q_i r_i \Lambda_i(\bar x_i, \bar t_i) 
\end{equation*}
we have by~\eqref{space} and~\eqref{time} and the definition~\eqref{Lambdaieq} of $\Lambda_i$  that
\begin{equation*}
\begin{aligned}
\Lambda_i(x,\bar t_i+tQ_i^{-2})&=r_i^{-1}\min\left(r_i- d_g(x_i,x) ,\sqrt{r_i^2 - (t_i-(\bar t_i+tQ_i^{-2}))}\right)\\
&\geq \frac{r_i^{-1}}{2} \min\left(r_i - d_g(x_i,\bar x_i ), \sqrt{r_i^2 - (t_i-\bar t_i) } \right)\\
&=\frac{1}{2} \Lambda_i(\bar x_i,\bar t_i).
\end{aligned}
\end{equation*}
Hence, by the above inequality together with~\eqref{goes_to_infty} and~\eqref{torsion_b}, we deduce that for every $R<+\infty$ and $i$ large enough, any $(x,t)\in B_{g_i}(\bar x_i,R)\times [-R^2,0]$ satisfies
\begin{equation*}
|T_{\widetilde{\g2}_i}(x,t)|\leq \frac{\Lambda_i (\bar x_i,\bar t_i)}{\Lambda_i(x,\bar t_i+tQ_i^{-2})} \leq 2.
\end{equation*}
We now want to invoke the compactness Theorem~\ref{compactnessthm}. The remaining hypotheses of this theorem are satisfied trivially, because under these rescalings, the curvature goes to zero and the injectivity radius goes to infinity. For this reason the limiting manifold is the Euclidean $\R^7$. Hence, by Theorem~\ref{compactnessthm} we obtain a subsequence that converges to a limit ancient isometric flow $(\mathbb R^7, \g2_\infty(t),g_{\mathrm{Eucl}}, 0)_{t\in (-\infty,0]}$ with 
\begin{equation} \label{some_torsion}
|T_{\g2_\infty}(0,0)|=1,
\end{equation}
due to~\eqref{torsion_one}.

Let $s \in (0,1)$. It follows from~\eqref{Q_to_infty} and~\eqref{rhoQ} that, for $i$ sufficiently large, we have
\begin{equation*}
\bar t_i - 1 < \bar t_i - \rho_i^2 < \bar t_i - s Q_i^{-2} < \bar t_i.
\end{equation*}
Applying the almost monotonicity formula~\eqref{eq:AM-1} and using scale invariance of $\Theta$, we obtain
\begin{equation*}
\Theta_{(\bar x_i,0)}(\widetilde{\g2}_i(-s)) = \Theta_{(\bar x_i,\bar t_i)} ( \g2_i (\bar t_i - s Q_i^{-2}) ) < K_1 \Theta_{(\bar x_i,\bar t_i)}(\g2_i(\bar t_i-\rho_i^2)) +  K_2 (\rho_i^2 - s Q_i^{-2}) (E_0 + 1).
\end{equation*}
By~\eqref{density_nearby} and~\eqref{Q_to_infty}, the right hand side above goes to zero as $i \rightarrow \infty$.

Let $s \in (0,1)$. It follows from~\eqref{Q_to_infty} and~\eqref{rhoQ} that, for $i$ sufficiently large, we have
\begin{equation*}
\bar t_i - 1 < \bar t_i - \rho_i^2 < \bar t_i - s Q_i^{-2} < \bar t_i.
\end{equation*}
Applying the almost monotonicity formula~\eqref{eq:AM-1} and using scale invariance of $\Theta$, we obtain
\begin{equation*}
\Theta_{(\bar x_i,0)}(\widetilde{\g2}_i(-s)) = \Theta_{(\bar x_i,\bar t_i)} ( \g2_i (\bar t_i - s Q_i^{-2}) ) < K_1 \Theta_{(\bar x_i,\bar t_i)}(\g2_i(\bar t_i-\rho_i^2)) +  K_2 (\rho_i^2 - s Q_i^{-2}) (E_0 + 1).
\end{equation*}
By~\eqref{density_nearby} and~\eqref{Q_to_infty}, the right hand side above goes to zero as $i \rightarrow \infty$.

Hence, we have
\begin{equation}
|s|\int_M |T_{\tilde{\g2}_i}(y,-s)|^2 u^{g_i}_{(\bar x_i,0)}(y,-s)  \vol_{g_i} (y) \rightarrow 0.
\end{equation}
By standard estimates on the heat kernel, we also know that
\begin{equation}
\big| u^g_{(\bar x_i,\bar t_i)}(x,t) \big| \leq \frac{C}{(\bar t_i - t)^{\frac{7}{2}}},
\end{equation}
for all $t\in [\bar t_i - a, \bar t_i)$, where $a>0$ is a sufficiently small constant. As a consequence, the rescaled heat kernels
\begin{equation}
u_i (y,s) := u^{g_i}_{(\bar x_i,\bar t_i)} (y,s)= Q_i^{-7} u^g _{(\bar x_i,\bar t_i)}(y,\bar t_i +sQ_i^{-2})
\end{equation}
satisfy
\begin{equation} \label{h_k_bound}
|u_i(y,s)| = Q_i^{-7} |u^g _{(\bar x_i,\bar t_i)}(y,\bar t_i +sQ_i^{-2})| \leq Q_i^{-7} \frac{C}{Q_i^7 |s|^{\frac{7}{2}}} =\frac{C}{|s|^{\frac{7}{2}}},
\end{equation}
on $M\times [-aQ_i^2,0)$, where $a Q_i^2 \rightarrow +\infty$.

Recall that by Cheeger--Gromov convergence, $(M,g_i,\bar x_i)$ converge smoothly to $(\mathbb R^7,g_{Eucl},0)$ up to diffeomorphisms $F_i:\Omega_i \rightarrow F_i(\Omega_i)$ with $\Omega_i\subset \mathbb R^7$, as in Definition \ref{defn_converge}. Moreover, the functions $F_i^* u_i$ solve the backwards heat equation on $(\Omega_i, F_i^*g_i)$ with a bound of the form \eqref{h_k_bound}. This, together with parabolic estimates provide uniform bounds on all derivatives of $F_i^* u_i$, in compact subsets of $M\times (-\infty, 0)$ and large $i$. Hence, by passing to a subsequence we may assume they converge to a smooth limit backward solution $u_\infty(y,s)$ of the heat equation on $\mathbb R^7$ that starts from $\delta_0$ at time $s=0$, which by uniqueness is
\begin{equation*}
u_{\infty}(y,s) = \frac{e^{-\frac{|y|^2}{4|s|}}}{(4\pi |s|)^{\frac{7}{2}}}.
\end{equation*}

Now, since in addition the isometric flows $F_i^*\tilde{\g2}_i$ converge smoothly to the isometric flow $\g2_\infty$, uniformly in compact sets of the form $\mathcal K \times [-A,0]$, and the heat kernels decay exponentially, we see that for $s>0$
\begin{equation}
|s|\int_M |T_{\tilde{\g2}_i}(y,-s)|^2 u^{g_i}_{(\bar x_i,0)}(y,-s)  \vol_{g_i} (y) \rightarrow |s| \int_{\mathbb R^7} |T_{\g2_\infty}(y,-s)|^2 u_{\infty}(y,-s) \vol_{g_{Eucl}}(y).
\end{equation}

Hence, for $s<0$
\begin{equation*}
|s| \int_{\mathbb R^7} |T_{\g2_\infty}(y,-s)|^2 u_{\infty}(y,-s) \vol_{g_{Eucl}}(y)=0,
\end{equation*}
and we conclude that $|T_{\g2_\infty}(0,0)|=0$, which contradicts~\eqref{some_torsion}, thus completing the proof.
\end{proof}
We use the $\epsilon$-regularity Theorem~\ref{eregularity} in Section~\ref{singular} to study singularities of the flow.

Another powerful consequence of Theorem~\ref{almost_mon} is the following corollary, which is used in Section~\ref{longte} to establish long-time existence and convergence of the flow given small entropy of the initial data.

\begin{corr}[Small initial entropy controls torsion] \label{torsion_bound}
Let $(M^7,g)$ be a compact Riemannian manifold. For every $\sigma>0$ there exist $\varepsilon, t_0>0$ and $C<+\infty$ such that if a $\G2$-structure $\g2_0$ induces $g$ and
\begin{equation} \label{eq:lowentropy}
\lambda(\g2_0,\sigma)<\varepsilon
\end{equation}
then the isometric flow $\g2(t)$ starting at $\g2_0$ satisfies 
\begin{equation*}
\max_M |T_{\g2(t)}| \leq \frac{C}{\sqrt t}
\end{equation*}
for all $t\in (0, t_0]$.
\end{corr}
\begin{proof}
By~\eqref{eq:entropy} and~\eqref{eq:Theta}, the small entropy assumption~\eqref{eq:lowentropy} implies that
\begin{equation}
\Theta_{(x,t)}(\g2)<\varepsilon,
\end{equation}
for every $(x,t)\in M\times (0,\sigma)$. 

We argue again by contradiction. Suppose there exists a $\bar\sigma > 0$, a sequence $\varepsilon_i \rightarrow 0$ and counterexamples $(M_i,\g2(t))_{t\in [0,t_i]}$, with $t_i\rightarrow 0$ admitting $\bar t_i \in [0,t_i]$ and $x_i\in M$ such that
\begin{equation}
\sqrt{\bar t_i} |T_{\g2_i}(x_i, \bar t_i)|=\sqrt{\bar t_i} \max_M |T(\g2(t))| \geq i \rightarrow +\infty
\end{equation}
and $\lambda(\g2_i(0),\bar\sigma)<\varepsilon_i$. Thus in particular, we have $\Theta_{(x_i,\bar t_i)}(\g2_i(0))<\varepsilon_i$.

Let $Q_i= |T_{\g2_i} (x_i,\bar t_i)| \rightarrow +\infty$ and consider the rescaled pointed sequence 
\begin{equation}
\Big( M_i,\widetilde{\g2}_i(t)= Q_i^3 \g2_i(t Q_i^{-2} + \bar t_i ),x_i \Big).
\end{equation}
Let $s>0$. Note that $\bar t_i \to 0$. Applying scale invariance and the almost monotonicity formula~\eqref{eq:AM-1} as we did in the proof of Theorem~\ref{eregularity}, we find that for $i$ sufficiently large we have
\begin{equation*}
\Theta_{(x_i,0)} ( \widetilde{\g2}(-s) ) = \Theta_{(x_i,\bar t_i)} (\g2( \bar t_i - sQ_i^{-2}) ) \leq K_1 \Theta_{(x_i,\bar t_i)} (\g2_i (0)) + K_2 (E(\g2_i(0)) + 1)\bar t_i \rightarrow 0.
\end{equation*}
As in the proof of Theorem~\ref{eregularity}, it follows that the $\tilde{\g2_i}$ converge to a $\G2$-structure $\g2_\infty$ on $\mathbb R^7$, with $T_{{\g2}_{\infty}} \equiv 0$.
On the other hand, as in the proof of Theorem~\ref{eregularity}, we have
\begin{equation*}
|T_{\g2_\infty}(0,0)| =1,
\end{equation*}
giving us our contradiction.
\end{proof}

\begin{rmk}
Because the backwards heat kernel satisfies
\begin{equation}
u_{(x,t)}(y,s)\leq \frac{C}{(t-s)^{\frac{7}{2}}},
\end{equation}
we have 
\begin{align}
\begin{split}
(t-s)\int_M |T_{\g2(t)}|^2 u_{(x,t)} \vol_g &\leq \frac{C}{(t-s)^{\frac{5}{2}}} \int_M |T_{\g2(t)}|^2 \vol_g \\
&\leq \frac{2C}{(t-s)^{\frac{5}{2}}} E(\g2(t))\leq \frac{2C}{(t-s)^{\frac{5}{2}}} E(\g2_0).
\end{split}
\end{align}
Therefore, in order to be able to bound $\Theta_{(x,t)}(\g2(s))$ in terms of $E(\g2(0))$ for $s<t$, we would also need a positive lower bound on $t-s$, which unfortunately fails for small times $t$. This is precisely why control of the stronger quantity $\lambda(\g2_0,\sigma)$ is needed.
\end{rmk}

\subsection{Long time existence} \label{longte}

In this section we consider the isometric flow on a compact manifold $(M^7,g)$ from initial data satisfying certain smallness assumptions on the torsion, and prove long time existence and convergence.

In particular, we first prove the result under the assumption that the torsion of the initial $\G2$-structure is \emph{pointwise} small, in Theorem~\ref{smalltorsion}. Then, because small entropy implies controlled torsion for some time, by Corollary~\ref{torsion_bound}, we can combine our derivative estimates for the torsion in Theorem~\ref{shiestimatesthm} with interpolation (Lemma~\ref{interpol} below) to prove that the torsion in fact does become pointwise small after some time.

Crucial in the establishment of convergence is the convexity of the energy functional
\begin{equation*}
E(\g2)=\frac{1}{2} \int_M |T_{\g2}|^2 \vol_g
\end{equation*}
along an isometric flow with small torsion. (See Remark~\ref{convexity_important}.) For this, we begin with the following lemma.

\begin{lemma} \label{convexity}
Along an isometric flow $\g2(t)$ the energy 
\begin{equation*}
E(\g2(t))=\frac{1}{2}\int_M |T_{\g2(t)}|^2 \vol_g
\end{equation*}
satisfies
\begin{equation*}
\frac{d^2}{dt^2} E(\g2(t)) = -\frac{d}{dt}\int_M |\Div T|^2 \vol_g = 2\int_M \big( |\nabla \Div T|^2 + T_{pl} \nabla_p (\Div T)_q (\Div T)_k \g2_{lqk} \big) \vol_g.
\end{equation*} 
\end{lemma}
\begin{proof}
Recall that from the proof of Proposition~\ref{gradient} that
\begin{equation*}
\frac{d}{dt} E(\g2(t)) = -\int_M |\Div T|^2 \vol_g.
\end{equation*}
and that from~\eqref{evolT1} the torsion evolves under the isometric flow by
\begin{equation*}
\partial_t T_{pq} = T_{pl}\nabla_i T_{ik} \g2_{klq} + \nabla_p \nabla_i T_{iq}.
\end{equation*}
Thus we have
\begin{align*}
&\partial_t \nabla_p T_{pq}= \nabla_p \partial_t T_{pq}= \nabla_p ( T_{pl}\nabla_i T_{ik} \g2_{klq} + \nabla_p \nabla_i T_{iq}).
\end{align*}
This gives
\begin{align*}
\frac{d^2}{dt^2} E(\g2(t)) &= - \frac{d}{dt} \int_M |\Div T|^2 \vol_g,\\
&=-2 \int_M \partial_t \nabla_p T_{pq} \nabla_m T_{mq} \vol_g\\
&=-2 \int_M \nabla_p ( T_{pl}\nabla_i T_{ik} \g2_{klq} + \nabla_p \nabla_i T_{iq}) \nabla_m T_{mq} \vol_g.
\end{align*}
Integrating by parts, we get
\begin{align*}
\frac{d^2}{dt^2} E(\g2(t)) &= 2 \int_M (T_{pl}\nabla_i T_{ik} \g2_{klq} + \nabla_p \nabla_i T_{iq}) \nabla_p \nabla_m T_{mq} \vol_g\\
&=2\int_M \big( |\nabla \Div T|^2 + T_{pl} (\Div T)_k \g2_{klq} \nabla_p (\Div T)_q \big) \vol_g,
\end{align*}
which is what we wanted to show.
\end{proof}

The required convexity is provided by the following result.
\begin{lemma} \label{decay}
Let $\Lambda > 0$ be the first non-zero eigenvalue of the rough Laplacian
\begin{align*}
\nabla^* \nabla = - \Delta : \Gamma(TM) &\rightarrow \Gamma(TM)\\
X&\mapsto \nabla^* \nabla X = -\Delta X = -\nabla_k \nabla_k X.
\end{align*}
Then
\begin{equation*}
\frac{d}{dt} \int_M |\Div T|^2 \vol_g \leq -\frac{\Lambda}{2} \int_M |\Div T|^2 \vol_g,
\end{equation*}
as long as $|T|^2 \leq \frac{1}{14} \Lambda:=\alpha_0^2$.
\end{lemma}
\begin{proof}
Using Lemma~\ref{convexity}, Young's inequality, and $|\g2|^2 =7$, we estimate
\begin{align*}
\frac{d^2}{dt^2} E(\g2(t)) &= - \frac{d}{dt} \int_M |\Div T|^2 \vol_g \\
&\geq 2 \int_M \left( |\nabla \Div T|^2 - \frac{1}{2} |\nabla \Div T|^2 - \frac{7}{2} |T|^2 |\Div T|^2  \right) \vol_g \\
&= \int_M \left( |\nabla \Div T|^2  - 7 |T|^2 |\Div T|^2 \right) \vol_g.
\end{align*}

Now consider the non-negative elliptic operator $\nabla^* \nabla = - \Delta: \Omega^1(M) \rightarrow \Omega^1(M)$, namely $\nabla^* \nabla X_i= -\nabla_k\nabla_k X_i$. The operator $\nabla^* \nabla$ has discrete spectrum and by compactness its kernel consists of the parallel vector fields (that is, $\nabla X=0$). Therefore, $\nabla^* \nabla$ is strictly positive on the subspace $L^2$-orthogonal to the parallel vector fields. In other words, there exists $\Lambda>0$ such that
\begin{equation*}
\Lambda \int_M |X|^2 \vol_g \leq \int_M |\nabla X|^2 \vol_g
\end{equation*}
for every vector field $X \in (\ker \Delta)^{\perp_{L^2}}$.

Next we observe that $\Div T$ is always orthogonal to $\ker \Delta$, because for any vector field $X$ with $\nabla X=0$ we have
\begin{equation*}
\int_M \langle \Div T,  X \rangle  \vol_g = -\int_M T_{pq} \nabla_p X_q \vol_g=0.
\end{equation*}
It follows that 
\begin{equation*}
\Lambda \int_M |\Div T|^2 \vol_g \leq \int_M |\nabla \Div T|^2 \vol_g.
\end{equation*}
Hence, we deduce that
\begin{align*}
\frac{d^2}{dt^2} E(\g2(t)) &= - \frac{d}{dt} \int_M |\Div T|^2 \vol_g \\
&\geq  \int_M \left( |\nabla \Div T|^2  - 7 |T|^2 |\Div T|^2 \right) \vol_g \\
&\geq \int_M (\Lambda |\Div T|^2 - 7 |T|^2 |\Div T|^2 )\vol_g \\
& \geq \int_M (\Lambda - 7 |T|^2 )|\Div T|^2 \vol_g.
\end{align*}
Now suppose that $|T|^2 \leq \tfrac{1}{14}\Lambda$. Then the above becomes
\begin{align*}
\frac{d}{dt} \int_M |\Div T|^2 \vol_g \leq \int_M (-\Lambda + 7 |T|^2) |\Div T|^2 \vol_g \leq -\frac{\Lambda}{2} \int_M |\Div T|^2 \vol_g,
\end{align*}
as claimed.
\end{proof}

The following interpolation result is crucial.
\begin{lemma}[Interpolation] \label{interpol}
Let $\g2$ be a $\G2$-structure on a compact manifold $M$ inducing the Riemannian metric $g$, and let $T$ be its torsion. Suppose that $|\nabla T|\leq C$ and that for every $x\in M$ and $0<r\leq 1$ we have
\begin{equation*}
\vol_g(B(x,r))\geq v_0 r^7,
\end{equation*}
for some small constant $v_0>0$. Then, for every $\varepsilon>0$ there exists a $\delta(\varepsilon, C, v_0)>0$ such that if
\begin{equation} \label{Edelta}
E(\g2)<\delta,
\end{equation}
then $|T| <\varepsilon$.
\end{lemma}
\begin{proof}
The proof is quite standard, but we include it for the sake of completeness. The bound $|\nabla T|\leq C$ implies that 
\begin{equation*}
2 |T| |\nabla |T|| = |\nabla |T|^2|=2|\langle\nabla T, T\rangle|  \leq 2|\nabla T| |T| \leq 2C |T|.
\end{equation*}
for an appropriate constant $C<+\infty$. Therefore,
\begin{equation} \label{modTbound}
|\nabla |T|| \leq C,
\end{equation}
at points where $|T|\not = 0$.

Fix $\varepsilon > 0$. Suppose that for any $\delta > 0$, there exists $x \in M$ with 
\begin{equation*}
|T_{\g2}(x)| >\varepsilon.
\end{equation*}
Thus, by~\eqref{modTbound} and the mean value theorem on the smooth function $|T_{\g2}(x)|$, if we take  $r=\frac{\varepsilon}{2C}>0$ then
\begin{equation*}
|T_{\g2}|>\frac{\varepsilon}{2} \qquad \text{in $B_g(x,r)$.}
\end{equation*}
It follows that
\begin{align*}
4E(\g2)= 2\int_M |T_{\g2}|^2 \vol_g & \geq 2\int_{B_g(x,r)} |T_{\g2}|^2 \vol_g \\
& \geq \vol_g(B_g(x,r)) \varepsilon \geq v_0 r^7 \varepsilon= \frac{v_0\varepsilon^8}{(2C)^7}.
\end{align*}
However, if we choose $\delta=\frac{1}{8}\frac{v_0\varepsilon^8}{(2C)^7}$, then the above contradicts~\eqref{Edelta}. This completes the proof.
\end{proof}

We are now able to prove the following convergence result.

\begin{thm}[Long time existence given small initial torsion]\label{smalltorsion}
Let $\varphi_0$ be a $\G2$-structure on a compact manifold $M$ inducing the metric $g$. Then, for every $\delta>0$ there exists an $\varepsilon(\delta,g)>0$ such that if $|T_{\varphi_0}|<\varepsilon$ then the isometric flow $\g2(t)$ starting from $\varphi_0$ exists for all time and converges smoothly to a $\G2$-structure $\varphi_\infty$ inducing the metric $g$ on $M$, and satisfying
\begin{align*}
\Div T_{\varphi_\infty}&\equiv 0, \\
|T_{\g2_{\infty}}|&<\delta.
\end{align*}
\end{thm}
\begin{proof}
Recall that the isometric flow is the negative gradient flow of a multiple of the energy, and hence $E(\varphi(t))\leq E(\varphi(0))$ for all $t$ for which the flow exists. By rescaling, we may assume that $\vol(M,g)=1$. By the doubling time estimate in Proposition~\ref{dtestprop}, there is an $\varepsilon_0>0$ such that whenever 
\begin{equation*}
|T_{\varphi_0}| <\varepsilon_0,
\end{equation*} 
we have that
\begin{equation}
t_*=\max\{t\geq 0 :  |T_{\varphi(t)}|\leq 2\} > \delta.
\end{equation}
First suppose that $t_* < +\infty$. By the derivative estimates Theorem~\ref{shiestimatesthm} applied for $K=2$ in the time interval $[t_*-\delta,t_*]$ there is a constant $c_0$ such that
\begin{equation}
|\nabla T_{\varphi(t_*)}|<c_0.
\end{equation}
Hence, applying Lemma~\ref{interpol}, we deduce that for every $\alpha>0$ there exists a $\gamma_{\alpha} >0$ depending also on $c_0$ and $g$ such that if $E(\varphi_0)<\gamma_\alpha$ then $|T_{\varphi(t_*)}|<\alpha$, because $E(\varphi(t_*))\leq E(\varphi(0))$.

Therefore if we take $\varepsilon < \min\{\varepsilon_0, \gamma_{2}\}$ we obtain $|T_{\varphi(t_*)}| < 2$, which contradicts the maximality of $t_*$. Thus, we must have $t_*=+\infty$ and so in fact the flow exists for all time. 

Now let $\alpha_0= \tfrac{1}{\sqrt{14}} \Lambda^{\frac{1}{2}}$ be the constant of Lemma~\ref{decay}, where $\Lambda > 0$ is the first non-zero eigenvalue of the rough Laplacian acting on vector fields on $M$.

If we take $\varepsilon < \min\{\varepsilon_0, \gamma_{2}, \gamma_{\alpha_0} \}$, we obtain a flow that exists for all time and satisfies the conditions of Lemma~\ref{decay} for all time, since $E(\g2(t))$ is nondecreasing. Hence, 
\begin{equation*}
\frac{d}{dt} \int_M |\Div T_{\g2(t)}|^2 \vol_g \leq -\frac{\Lambda}{2} \int_M |\Div T_{\g2(t)}|^2 \vol_g,
\end{equation*}
for all time, which leads to the decay estimate 
\begin{equation} \label{decay2}
\int_M |\Div T_{\g2(t)}|^2 \vol_g \leq e^{-\frac{\Lambda t}{2}} \; \int_M |\Div T_{\g2(0)}|^2 \vol_g,
\end{equation} 
for all $t\geq 0$.

Thus, for every $s_1<s_2$, we can estimate
\begin{equation} \label{convergence}
\begin{aligned}
\int_M |\g2(s_2)-\g2(s_1)| \vol_g &= \int_M \int_{s_1}^{s_2} \left| \partial_t \g2(s)  \right| ds \vol_g\\
&= \int_{s_1}^{s_2} \int_M |\Div T_{\g2(s)}| \vol_g ds\\
&\leq \int_{s_1}^{s_2} \left(\int_M |\Div T_{\g2(s)}|^2 \vol_g \right)^{\frac{1}{2}} ds\\
&\leq c_2 \int_{s_1}^{s_2}  e^{-\frac{\Lambda s}{4}} ds. 
\end{aligned}
\end{equation}
Hence, $\g2(t)$ has a unique limit $\g2_{\infty}$ in $L^1$ as $t\rightarrow+\infty$, in fact exponentially. 

Moreover, the uniform torsion bound for all $t\geq 0$ gives estimates on all derivatives of the torsion, for all times $t\geq1$, by Theorem~\ref{shiestimatesthm}. This implies that given any sequence $t_n \rightarrow+\infty$, a subsequence of $\g2(t_n)$ converges smoothly to a limit, which must be $\g2_\infty$ by uniqueness. Therefore, the flow $\g2(t)$ converges smoothly to $\g2_\infty$ as $t\rightarrow +\infty$. Finally, the inequality~\eqref{decay2} implies that $\Div T(\g2_\infty)\equiv 0$, and choosing $\varepsilon>0$ small enough we can also achieve  $|T_{\g2_\infty} |<\delta$, using the interpolation Lemma~\ref{interpol}.
\end{proof}

\begin{rmk} \label{convexity_important}
Note that the convexity Lemma~\ref{decay} was crucial to obtain~\eqref{convergence}, which implies uniqueness of all limits of sequences $\g2(t_n)$ with $t_n\rightarrow +\infty$. Without it, we can only assert that any  sequence $\g2(t_n)$ with $t_n\rightarrow +\infty$ has a subsequence that converges smoothly to \textit{some} limit, which may depend on the subsequence.
\end{rmk}

We now have all we need to prove long time existence and convergence given small entropy.
\begin{thm}[Low entropy convergence] \label{thm:lowentropy}
Let $(M^7, \g2_0)$ be a compact manifold with $\G2$-structure inducing the metric $g$. Then, there exist constants $C_k <+\infty$ \emph{depending only on $(M,g)$} such that for every small $\delta>0$ and $\sigma>0$, there exists $\varepsilon(g,\delta,\sigma)>0$ such that if 
\begin{equation}
\lambda(\g2_0,\sigma) <\varepsilon,
\end{equation}
then the isometric flow starting at $\varphi_0$ exists for all time and converges smoothly to a $\G2$-structure $\varphi_\infty$ satisfying 
\begin{align*}
\Div T_{\varphi_\infty}=0,\\
|T_{\g2_{\infty}}| <\delta,
\end{align*}
and
\begin{equation*}
|\nabla^k T_{\varphi_\infty}|\leq C_k,
\end{equation*}
for all $k\geq 1$.
\end{thm}
\begin{proof}
By Corollary~\ref{torsion_bound}, if $\varepsilon>0$ is small enough then we obtain a solution $(\g2(t))_{t\in[0,\tau_0]}$ of the isometric flow satisfying
\begin{equation*}
|T_{\g2(t)}|\leq \frac{C}{\sqrt t},
\end{equation*}
for all $t\in (0,\tau_0]$. Moreover, by the derivative estimates in Theorem~\ref{shiestimatesthm}, $\g2(\tau_0)$ satisfies $|\nabla T_{\g2(\tau_0)}| \leq C'$, for some constant $C'<\infty$. Hence, by the interpolation Lemma~\ref{interpol}, for every $\delta>0$, if $\varepsilon>0$ is even smaller we obtain $|T_{\g2(\tau_0)}|<\delta$. Then by Theorem~\ref{smalltorsion}, the flow converges to a $\G2$-structure $\g2_\infty$  with divergence-free small torsion, with derivative bounds.
\end{proof}

We also have the following corollary.
\begin{corr} \label{lowentropycor}
Let $\sigma > 0$. Given a metric $g$, define $(\inf \lambda)(g,\sigma)$ by
\begin{equation}\label{inf_zero}
(\inf \lambda)(g,\sigma) = \inf\left\{\lambda(\g2,\sigma), \; g_{\varphi}=g\right\}.
\end{equation}
If $(\inf \lambda)(g,\sigma)=0$ then there exists a torsion-free $\G2$-structure that induces $g$, hence $g$ is Ricci flat with holonomy in $\G2$.
\end{corr}
\begin{proof}
Consider a sequence $\delta_i\rightarrow 0$ and let $\varepsilon_i>0$ be obtained by Theorem~\ref{thm:lowentropy}. By~\eqref{inf_zero} there is a sequence of $\G2$-structures $\g2_i$ inducing $g$ such that 
\begin{equation*}
\lambda(\g2_i,\sigma)<\varepsilon_i.
\end{equation*}
Theorem~\ref{thm:lowentropy} then implies that the isometric flows starting from each $\g2_i$ converge to $\G2$-structures $\bar{\g2_i}$ satisfying
\begin{equation*}
|T_{\bar{\g2_i}}| < \delta_i \rightarrow 0,
\end{equation*}
with uniform derivative estimates. By the compactness Theorem~\ref{strccompactnessthm} we obtain a limit torsion-free $\G2$-structure on $M$ inducing the Riemannian metric $g$.
\end{proof}

An interesting question is whether it is possible to prove Corollary~\ref{lowentropycor} without using the isometric flow, but rather by using direct minimization methods.

\subsection{Singularity structure} \label{singular}
 
In this section we investigate the structure of singularities of the isometric flow. Consider an isometric flow $(\g2(t))_{t\in[0,\tau)}$ on a compact $7$-manifold $M$ encountering a finite time singularity at $\tau<+\infty$. 
 
Fixing the constants $\varepsilon,\bar \rho>0$ of the $\varepsilon$-regularity Theorem~\ref{eregularity} we define the \emph{singular set}
\begin{equation} \label{Seq}
S=\{x\in M : \Theta_{(x,\tau)}(\g2(\tau-\rho^2)) \geq \varepsilon, \textrm{ for all } \rho \in (0,\bar \rho] \}.
\end{equation}

The next result explains why $S$ is called the singular set for the isometric flow.
\begin{lemma} \label{sing}
The isometric flow $(\g2(t))_{t \in [0,\tau)}$ restricted to $M\setminus S$ converges as $t\rightarrow \tau$, smoothly and uniformly away from $S$, to a smooth $\G2$-structure $\g2(\tau)$  on $M\setminus S$. In particular, $M\setminus S$ is open in $M$, and hence $S$ is closed. Moreover, for every $x\in S$ there is a sequence $(x_i,t_i)$ with $x_i \rightarrow x$ and $t_i \rightarrow \tau$ such that
\begin{equation*}
\lim_{i} |T_{\g2}(x_i,t_i)| = +\infty.
\end{equation*}
Thus, $S$ is indeed the singular set of the isometric flow.
\end{lemma}
\begin{proof}
By Theorem~\ref{eregularity}, for every $x\in M\setminus S$ there exist $r_x>0$ and $C_x<+\infty$ such that 
\begin{equation*}
\Lambda_{r_x}(y,t) |T_{\g2}(y,t)| \leq C r_x^{-1}
\end{equation*}
for all $y \in B(x,r_x)\times [t_0-r_x^2,t_0]$, where
\begin{equation*}
\Lambda_{r_x}(y,t) = \min\big(1-r_x^{-1}d_g(x,y), \sqrt{1- r_x^{-2} (t_0-t)}\big).
\end{equation*}
Thus, with $\hat r_x = \tfrac{1}{2} r_x$, we have
\begin{equation*}
|T_{\g2}(y,t)| \leq \tilde C_x,
\end{equation*}
for all $y \in B_g(x,\hat r_x) \times [\tau- \hat r_x^2, \tau]$. Hence, by the local derivative estimates Theorem~\ref{localestthm} there exist constants $C_{x,j}$, for any $j\geq 1$, such that
\begin{equation*}
|\nabla^j T_{\g2}(y,t)| \leq C_{x,j},
\end{equation*}
in $B_g(x,\tfrac{1}{2} \hat r_x) \times [\tau- \frac{1}{4} \hat r_x^2, \tau]$.

As in the proof of Theorem~\ref{ltethm}, it follows that as $t\rightarrow \tau$, the flow $\g2(t)$ converges smoothly and uniformly away from $S$ to a $\G2$-structure $\g2(\tau)$ on $M\setminus S$, which induces the same Riemannian metric $g$ on $M$.
\end{proof}

We now establish an upper bound on the ``size'' of the singular set $S$.
\begin{thm}[Singularity structure] \label{singularities}
Let $\varphi_0$ be a $\G2$-structure inducing the metric $g$ with
\begin{equation}\label{initial_energy}
E(\g2_0)=\frac{1}{2}\int_M |T_{\varphi_0}|^2 \vol_g \leq E_0
\end{equation}
and consider the maximal smooth isometric flow $(\g2(t))_{t\in [0,\tau)}$ with $\g2(0)=\g2_0$. 

Suppose that  $\tau<+\infty$. Then as $t\rightarrow \tau$ the flow converges smoothly to a $\G2$-structure $\g2_{\tau}$ outside a closed set $S$ with finite $5$-dimensional Hausdorff measure satisfying 
\begin{equation*}
\mathcal H^5(S) \leq C E_0,
\end{equation*}
for some constant $C<\infty$ depending on $g$. In particular the Hausdorff dimension of $S$ is at most $5$. 
\end{thm}
\begin{proof}
From Lemma~\ref{sing}, all that remains is to prove the estimate on $\mathcal H^5(S)$, where $\mathcal H^5$ denotes the $5$-dimensional Hausdorff measure on $(M,g)$.

Consider any subset $S'\subset S$ with finite $\mathcal H^5$ measure. As in~\cite{HamGray96}, there is $\tilde S\subset S'$ such that
\begin{equation} \label{good_set}
\mathcal H^5(\tilde S)\geq \frac{1}{2} \mathcal H^5 (S')
\end{equation}
and 
\begin{equation*}
u_{\tilde S}(y,s) := \int_{\tilde S} u_{(x,\tau)} (y,s) \vol_g(x)
\end{equation*}
satisfies
\begin{equation} \label{kernel_estimate}
u_{\tilde S}(y,s) \leq \frac{C}{\tau-s}.
\end{equation}
Then for every $\rho\in (0,\bar\rho]$, using the definition~\eqref{Seq} of $S$ and that $\tilde S \subset S$, we can then estimate
\begin{equation*}
\varepsilon \mathcal H^5(\tilde S) = \int_{\tilde S} \varepsilon d\mathcal H^5(x) \leq \int_{\tilde S} \Theta_{(x,\tau)} (\g2(\tau-\rho^2)) d\mathcal H^5(x).
\end{equation*}
By the definition~\eqref{eq:Theta} of $\Theta$, and the estimate~\eqref{kernel_estimate} and~\eqref{initial_energy}, this becomes
\begin{align*}
\varepsilon \mathcal H^5(\tilde S) &\leq  \int_{\tilde S} \int_M \rho^2 |T(y,\tau-\rho^2)|^2 u_{(x,\tau)}(y,\tau-\rho^2) \vol_g(y) d\mathcal H^5(x)\\
&\leq \int_M \rho^2 |T(y,\tau-\rho^2)|^2 u_{\tilde S}(y, \tau-\rho^2) \vol_g(y)\\
&\leq C \int_M |T(y,\tau-\rho^2)|^2 \vol_g(y)\\
&\leq C E_0.
\end{align*}
The result now follows from~\eqref{good_set} and the arbitrariness of $S'$.
\end{proof}

\begin{rmk} \label{rmk:singular-set}
Theorem~\ref{singularities} says that the singular set $S$ is \emph{at most} $5$-dimensional. It would be interesting to find a geometric interpretation of the singular set $S$ in terms of $\G2$ geometry. If such a description exists, then it is likely that $S$ would be at most $4$-dimensional, as there are no distinguished $5$-dimensional subspaces in $\G2$ geometry.
\end{rmk}

Finally, we prove that if a singularity is of Type-$\tm{\RNum{1}}$ then a sequence of blow-ups of the flow admits a subsequence that converges to a shrinking soliton of the flow.
\begin{thm}[Type I singularities] \label{typeI}
Let $\varphi_0$ be a $\G2$-structure inducing the metric $g$ on a compact $7$-manifold $M$, and consider the maximal smooth isometric flow $(\g2(t))_{t\in [0,\tau)}$, with $\g2(0)=\g2_0$. Suppose that $\tau<+\infty$ and the flow encounters a Type I singularity. That is,
\begin{equation*}
\max_M |T_{\g2(t)}| \leq \frac{C}{\sqrt{\tau-t}}.
\end{equation*}
Let $x\in M$ and $\mu_i\searrow 0$ and consider the rescaled sequence $\g2_i(t)=\mu_i^{-3} \g2(\tau-\mu_i^2 t)$. Then, after possibly passing to a subsequence, $(M,\g2_i(t),x)$ converges smoothly to an ancient isometric flow $(\g2_\infty(t))_{t<0}$ on $(\mathbb R^7,g_{\mathrm{Eucl}})$ induced by a shrinking soliton. That is,
\begin{equation*}
\Div T_{\g2_\infty}(x,t)= -\frac{x}{2t} \hk T_{\g2_\infty}.
\end{equation*}
Moreover  $x\in M\setminus S$ if and only if $\g2_\infty(t)$ is the stationary flow induced by a torsion-free $\G2$-structure $\g2_\infty$ on $(\mathbb R^7,g_{\mathrm{Eucl}})$.
\end{thm}
\begin{proof}
The subconvergence of the blow-up sequence to an ancient isometric flow on $\mathbb R^7$ follows directly from the compactness theorem. That the limit is a shrinking soliton is a consequence of the almost monotonicity formula~\eqref{eq:AM-1}, just as in~\cite{HamGray96}.
\end{proof}

\begin{rmk} \label{final-rmk}
It is an interesting open problem whether there exist any nontrivial shrinking solitons on the Euclidean $\R^7$. If there do not exist any such solitons, then Theorem~\ref{typeI} would imply that no Type I singularities can occur along the isometric flow.
\end{rmk}

\bibliographystyle{amsplain}
\bibliography{isometric-flow}

\end{document}